\documentclass[12pt,twoside]{article}
\usepackage[latin1]{inputenc}
\usepackage[english]{babel}
\usepackage{clm-macros}
\usepackage{nishad-macros}
\usepackage{amsmath}
\usepackage{psfrag}
\usepackage{epsfig}
\usepackage{graphicx}
\usepackage{verbatim}
\usepackage{tikz}
\usepackage{enumerate}
\usepackage{datetime}
\usepackage{multicol}
\usepackage{fullpage}
\usepackage{enumerate}

\usetikzlibrary{decorations}
\usetikzlibrary{decorations.pathmorphing}
\usetikzlibrary{arrows}
\tikzstyle{every node}=[circle, draw, fill=white,inner sep=0pt, minimum width=4pt]
\tikzstyle{nodelabel}=[rounded corners,fill=none,inner sep=5pt,draw=none]
\tikzstyle{matching} = [ultra thick]
\tikzset{snake/.style={decorate, decoration=snake}}

\newdateformat{UKvardate}{
\THEDAY\ \monthname[\THEMONTH], \THEYEAR}

\def\Notname{Notation}
\def\casname{Case}
\def\cnjname{Conjecture}
\def\corname{Corollary}
\def\Defname{Definition}
\def\exename{Exercise}
\def\exaname{Example}
\def\lemname{Lemma}
\def\listassertionname{List of Assertions}
\def\prbname{Problem}
\def\prpname{Proposition}
\def\proofOfname{\proofname{} of}
\def\ssbname{Case}
\def\subname{Case}
\def\thmname{Theorem}

\begin{document}

\title{Generating Simple Near-Bipartite Bricks}
\author{Nishad Kothari\footnote{Partially supported by {\sc NSERC} grant (RGPIN-2014-04351, J. Cheriyan).} \and
Marcelo H. de Carvalho\footnote{Supported by {\sc Fundect-MS} and {\sc CNP}q.}}
\UKvardate
\date{14 January, 2020}
  \maketitle 
  \thispagestyle{empty}

\begin{abstract} 
A {\it brick} is a $3$-connected graph such that the graph obtained from it
by deleting any two distinct vertices has a perfect matching.
A brick $G$ is {\it near-bipartite} if it has a pair of edges $\alpha$ and $\beta$
such that $G-\{\alpha,\beta\}$ is bipartite and matching covered; examples
are $K_4$ and the triangular prism~$\overline{C_6}$.
The significance of near-bipartite bricks
arises from the theory of ear decompositions of {\mcg}s.

\smallskip
The object of this paper is to establish a generation procedure which is specific
to the class of simple near-bipartite bricks. In particular, we prove
that every simple near-bipartite brick $G$ has an edge~$e$ such that the graph
obtained from $G-e$ by contracting each edge that is incident with a vertex
of degree two is also a simple near-bipartite brick, unless $G$ belongs to any
of eight well-defined infinite families.
This  is a refinement of the brick generation theorem of Norine and Thomas
\cite{noth07}
which is appropriate for the restricted class of near-bipartite bricks.

\smallskip
Earlier, the first author proved a similar generation theorem for
(not necessarily simple) near-bipartite bricks \cite{koth16,koth19};
we deduce our main result from this theorem. Our proof is based on
the strategy of Carvalho, Lucchesi and Murty \cite{clm08} and uses
several of their techniques and results.
The results presented here also appear in the Ph.D. thesis of the first
author \cite{koth16}.
\end{abstract}

\tableofcontents

\bigskip
\bigskip
This paper is a sequel to a recent paper
of the first author \cite{koth19}. 
In the following section, we recall the most important definitions
and results from \cite{koth19}, and we explain the contributions
of this paper.
Please see \cite{koth19} for any missing definitions.
The reader, who is interested in gaining a deeper understanding,
should perhaps read the first two sections of \cite{koth19}.
\section{Brick Generation}
\label{sec:brick-generation}

A \mcg\ free of nontrivial tight cuts is called
a {\it brace} if it is bipartite; otherwise, it is called a {\it brick}.
Lov{\'a}sz~\cite{lova87} proved the remarkable result that any two tight cut
decompositions of a \mcg~$G$ yield the same list of bricks
and braces (except possibly for multiplicities of edges).
In particular, any two tight cut decompositions of~$G$ yields the same
number of bricks; this number is denoted by $b(G)$.

\smallskip
Edmonds, Lov{\'a}sz and Pulleyblank~\cite{elp82} established the following important
equivalence.
\begin{thm}
A graph~$G$ is a brick if and only if it is $3$-connected and bicritical.
\end{thm}

An edge~$e$ of a \mcg\ is {\it removable} if $G-e$ is also matching covered;
furthermore, it is {\it \binv} if $b(G-e)=b(G)$.
A \binv\ edge~$e$ of a brick~$G$ is {\it thin}
if the retract of $G-e$ is also a brick.
(The retract is the graph obtained by contracting each edge that is incident
with a vertex of degree two.) Carvalho, Lucchesi and Murty~\cite{clm06}
established the following.
\begin{thm}
{\sc [Thin Edge Theorem]}
\label{thm:clm-thin-bricks}
Every brick distinct from $K_4$, $\overline{C_6}$ and the Petersen graph
has a thin edge.
\end{thm}

As a consequence, every brick may be constructed from one of $K_4$,
$\overline{C_6}$ and the Petersen graph using four elementary
``expansion operations'' that are described in their paper~\cite{clm06}.

\smallskip
In order to establish a recursive procedure for generating simple bricks,
one needs the notion of a strictly thin edge.
A thin edge~$e$ of a simple brick~$G$ is {\it strictly thin} if the retract
of~$G-e$ is simple.
For each brick shown in Figure~\ref{fig:NT-bricks},
its thin edges are indicated by bold lines;
it is easily seen that none of these edges is strictly thin.

\smallskip
As an example, consider the Tricorn, shown in Figure~\ref{fig:Tricorn}, which has precisely three
thin edges indicated by bold lines;
deleting one of them, say~$e$, and taking the retract
yields the simple odd wheel~$W_5$. Thus each thin edge of the Tricorn
is strictly thin.

\begin{figure}[!ht]
\centering
\begin{tikzpicture}[scale=0.7]
\draw (90:3.2)node[nodelabel]{$e$};
\draw[ultra thick] (70:3) -- (110:3);
\draw[ultra thick] (190:3) -- (230:3);
\draw[ultra thick] (310:3) -- (350:3);
\draw (110:3) -- (190:3);
\draw (230:3) -- (310:3);
\draw (350:3) -- (70:3);
\draw (0:0) -- (90:1.5);
\draw (0:0) -- (210:1.5);
\draw (0:0) node{} -- (330:1.5);
\draw (90:1.5) -- (70:3)node{};
\draw (110:3)node{} -- (90:1.5)node{};
\draw (210:1.5) -- (190:3)node{};
\draw (230:3)node{} -- (210:1.5)node{};
\draw (330:1.5) -- (310:3)node{};
\draw (350:3)node{} -- (330:1.5)node{};
\draw (90:-3) node[nodelabel]{Tricorn};
\end{tikzpicture}
\hspace*{0.9in}
\vspace*{-0.2in}
\begin{tikzpicture}[scale=0.7]
\draw (210:1.5) to [out=90,in=210] (90:1.5);
\draw (330:1.5) to [out=90,in=330] (90:1.5);
\draw (230:3) to [out=90,in=180] (90:1.5);
\draw (310:3) to [out=90,in=0] (90:1.5);
\draw (230:3) -- (310:3);
\draw (0:0) -- (90:1.5);
\draw (0:0) -- (210:1.5);
\draw (0:0) node{} -- (330:1.5);
\draw (210:1.5)node{} -- (230:3)node{};
\draw (330:1.5)node{} -- (310:3)node{};
\draw (90:1.5)node{};
\draw (90:-3.35) node[nodelabel]{$W_5$};
\end{tikzpicture}
\caption{Thin edges of the Tricorn}
\label{fig:Tricorn}
\end{figure}

Next, we describe infinite
families of bricks that do not contain any strictly thin edges. Norine and Thomas \cite{noth07} proved that
these families, together with the Petersen graph, include all the bricks that are free of strictly thin edges; for
this reason, we refer to these families as {\it Norine-Thomas families}, and we refer to their members as {\it Norine-Thomas bricks}.

\subsection{Norine-Thomas families}
\label{sec:NT-bricks}

\bigskip
\noindent {\sc Odd Wheels.}
The {\it odd wheel} $W_{2k+1}$, for $k \ge 1$, is defined to be the join of
an odd cycle $C_{2k+1}$ and $K_1$.
See Figure~\ref{fig:NT-bricks}a. The smallest odd wheel
is $K_4$. If $k \ge 2$, then~$W_{2k+1}$ has exactly one vertex of degree
$2k+1$, called its {\it hub}, and the edges incident at the hub
are called its {\it spokes}. The remaining $2k+1$ vertices lie on a cycle, called
the {\it rim}, and they are referred to as {\it rim vertices}.

\bigskip
Each member of the remaining four families contains a bipartite matching covered subgraph
which is either a `ladder' or a `partial biwheel'. These bipartite graphs are
also the main building blocks of additional families of bricks which are of interest
in Section~\ref{sec:new-families}. For this reason, we start with a description
of these two families of bipartite graphs.

\bigskip
\noindent
{\sc Ladders.}
Let $x_0x_1 \dots x_j$ and $y_0y_1 \dots y_j$ be two
vertex-disjoint paths, where $j \geq 2$.
The graph~$K$ obtained by the union of these two paths, and by adding
edges $x_iy_i$ for $0 \leq i \leq j$, is called a {\it ladder}, and its edges
joining $x_i$ and $y_i$ are referred to as its {\it rungs}.
See Figure~\ref{fig:biwheels-ladders}.
The two rungs
$x_0y_0$ and $x_jy_j$ are {\it external}, and the remaining rungs are {\it internal}.
We say that $K$ is {\it odd} ({\it even}) if it has an odd (even) number of rungs.

\bigskip
\noindent
{\sc Partial Biwheels.}
Let $x_0x_1 \dots x_{2j+1}$ be an odd path, where $j \geq 1$.
The graph~$K$ obtained by adding two new vertices $u$ and $w$,
joining $u$ to vertices in $\{x_0, x_2, \dots, x_{2j}\}$, and joining
$w$ to vertices in $\{x_1, x_3, \dots, x_{2j+1}\}$, is called a
{\it partial biwheel}; the vertices $x_0$~and~$x_{2j+1}$ are referred to
as its {\it ends},
whereas $u$~and~$w$ are referred to as its {\it hubs}; and an edge
incident with a hub is called a {\it spoke}.
See Figure~\ref{fig:biwheels-ladders}.
The two spokes $ux_0$~and~$wx_{2j+1}$ are {\it external},
and the remaining spokes are {\it internal}.

\begin{figure}[!ht]
\medskip
\centering
\begin{tikzpicture}[scale=0.7]
\draw (2,0) -- (5,0);

\draw (2,0) -- (3.5,2);
\draw (3,0) -- (3.5,-2);
\draw (4,0) -- (3.5,2);
\draw (5,0) -- (3.5,-2);

\draw (2,0) node{}node[left,nodelabel]{$a$};
\draw (3,0) node[fill=black]{};
\draw (4,0) node{};
\draw (5,0) node[fill=black]{}node[right,nodelabel]{$b$};
\draw (3.5,2) node[fill=black]{}node[above,nodelabel]{$u$};
\draw (3.5,-2) node{}node[below,nodelabel]{$w$};
\end{tikzpicture}
\hspace*{0.3in}
\begin{tikzpicture}[scale=0.7]
\draw (1,0) -- (6,0);

\draw (1,0) -- (3.5,2);
\draw (2,0) -- (3.5,-2);
\draw (3,0) -- (3.5,2);
\draw (4,0) -- (3.5,-2);
\draw (5,0) -- (3.5,2);
\draw (6,0) -- (3.5,-2);

\draw (1,0) node{}node[left,nodelabel]{$a$};
\draw (2,0) node[fill=black]{};
\draw (3,0) node{};
\draw (4,0) node[fill=black]{};
\draw (5,0) node{};
\draw (6,0) node[fill=black]{}node[right,nodelabel]{$b$};
\draw (3.5,2) node[fill=black]{}node[above,nodelabel]{$u$};
\draw (3.5,-2) node{}node[below,nodelabel]{$w$};
\end{tikzpicture}

\vspace*{0.1in}
\begin{tikzpicture}[scale=0.8]

\draw (1,0) -- (1,2);
\draw (2,0) -- (2,2);
\draw (3,0) -- (3,2);

\draw (1,0) -- (3,0);
\draw (1,2) -- (3,2);

\draw (1,0) node{}node[left,nodelabel]{$a$};
\draw (1,2) node[fill=black]{}node[left,nodelabel]{$u$};
\draw (2,2) node{};
\draw (2,0) node[fill=black]{};
\draw (3,0) node{}node[right,nodelabel]{$w$};
\draw (3,2) node[fill=black]{}node[right,nodelabel]{$b$};

\end{tikzpicture}
\hspace*{0.5in}
\begin{tikzpicture}[scale=0.8]

\draw (1,0) -- (1,2);
\draw (2,0) -- (2,2);
\draw (3,0) -- (3,2);
\draw (4,0) -- (4,2);

\draw (1,0) -- (4,0);
\draw (1,2) -- (4,2);

\draw (1,0) node{}node[left,nodelabel]{$a$};
\draw (1,2) node[fill=black]{}node[left,nodelabel]{$u$};
\draw (2,2) node{};
\draw (2,0) node[fill=black]{};
\draw (3,0) node{};
\draw (3,2) node[fill=black]{};
\draw (4,2) node{}node[right,nodelabel]{$w$};
\draw (4,0) node[fill=black]{}node[right,nodelabel]{$b$};

\end{tikzpicture}
\smallskip
\caption{Partial biwheels (top) and Ladders (bottom)}
\label{fig:biwheels-ladders}
\bigskip
\end{figure}
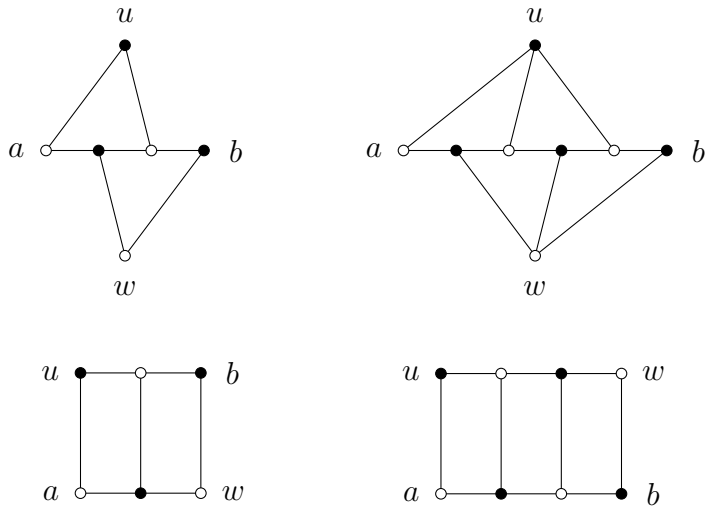

When referring to a ladder or to a partial biwheel, say~$K[A,B]$, with external
rungs/spokes $au$ and $bw$, we adopt the convention that
$a,w \in A$ and $b, u \in B$; furthermore, when $K$ is a partial biwheel,
$u$ and $w$ shall denote its hubs; as shown in Figure~\ref{fig:biwheels-ladders}.
(Sometimes, we may also use subscript
notation, such as $A_i$, $B_i$, $a_iu_i$ and $b_iw_i$ where $i$ is an integer, and this
convention extends naturally.)

\smallskip
It should be noted that a partial biwheel of order six is also a ladder. However,
a partial biwheel of order eight or more has only two vertices of degree two, namely, its
ends; whereas every ladder has four such vertices. We remark that, a {\it biwheel},
as defined by McCuaig \cite{mccu01}, has order
at least eight and contains an additional edge joining its ends; and
these constitute an important class of braces.

\bigskip
We now proceed to describe the remaining four Norine-Thomas families using
ladders and partial biwheels.

\bigskip
\noindent
{\sc Prisms, M{\"o}bius Ladders and Truncated Biwheels.}
Let $H[A,B]$ denote either a ladder or a partial biwheel of order~$n$,
with external rungs/spokes $au$ and $bw$,
and let $G$ be the graph obtained from~$H$ by adding two edges,
namely, $aw$ and $bu$. If $H$ is an odd ladder then $G$ is a {\it prism} and it
is denoted by~$P_n$, see Figure~\ref{fig:NT-bricks}b.
If $H$ is an even ladder then $G$ is a {\it M{\"o}bius ladder}
and it is denoted by~$M_n$, see Figure~\ref{fig:NT-bricks}f.
Finally, if $H$ is a partial biwheel then $G$ is a
{\it truncated biwheel} and it is denoted by~$T_n$, see Figure~\ref{fig:NT-bricks}c.
Note that $\overline{C_6}$ is the smallest prism as well as the smallest truncated
biwheel. For convenience, we shall consider $K_4$ to be the smallest M{\"o}bius ladder.

\begin{figure}[!ht]
\centering
\bigskip
\begin{tikzpicture}[scale=0.5]
\foreach \x in {0, 1, ..., 6}
{
\draw[ultra thick] (0:0) -- (\x*51.43+90:2.5);
\draw (\x*51.43+90:2.5) -- (\x*51.43+90+51.43:2.5);
}
\foreach \x in {0, 1, ..., 6}
{
\draw (\x*51.43+90:2.5) node{};
}
\draw (0:0) node{};
\draw (90:-3.6) node[nodelabel]{(a)};
\end{tikzpicture}
\hspace{0.5in}
\begin{tikzpicture}[scale=0.5]
\foreach \x in {0, 1, ..., 4}
{
\draw (\x*72+90+72:1.3) -- (\x*72+90:1.3);
\draw[ultra thick] (\x*72+90:1.3) -- (\x*72+90:2.5);
\draw (\x*72+90:2.5) -- (\x*72+90+72:2.5);
}
\foreach \x in {0, 1, ..., 4}
{
\draw (\x*72+90:1.3) node{};
\draw (\x*72+90:2.5) node{};
}
\draw (90:-3.6) node[nodelabel]{(b)};
\end{tikzpicture}
\hspace{0.5in}
\begin{tikzpicture}[scale=0.5]
\draw (0,0) -- (6,0);
\draw (0,0) -- (3,2.5) -- (6,0);
\draw (0,0)node{} -- (3,-2.5) -- (6,0)node{};
\draw[ultra thick] (2.4,0) -- (3,2.5) -- (4.8,0);
\draw[ultra thick] (1.2,0) -- (3,-2.5) -- (3.6,0);
\draw (2.4,0)node{};
\draw (3,2.5)node{};
\draw (4.8,0)node{};
\draw (1.2,0)node{};
\draw (3,-2.5)node{};
\draw (3.6,0)node{};
\draw (3,-3.6) node[nodelabel]{(c)};
\end{tikzpicture}

\vspace*{-0.1in}
\begin{tikzpicture}[scale=0.5]
\draw (0,0) -- (1.5,1.5) -- (0,3) -- (0,0);
\draw (1.5,1.5)node{} -- (4.5,1.5)node{};
\draw (0,3)node{} -- (6, 3);
\draw (0,0)node{} -- (6,0);
\draw (6,0) -- (4.5,1.5)node{} -- (6,3)node{} -- (6,0)node{};
\draw[ultra thick] (3,1.5) -- (3,3);
\draw (3,1.5)node{};
\draw (3,3)node{};
\draw (3,-1.7) node[nodelabel]{(d)};
\end{tikzpicture}
\hspace{0.4in}
\begin{tikzpicture}[scale=0.5]
\draw (0,0) -- (1.5,1.5) -- (0,3) -- (0,0);
\draw (1.5,1.5)node{} -- (4.5,1.5)node{};
\draw (0,3)node{} -- (6, 3);
\draw (0,0)node{} -- (6,0);
\draw (6,0) -- (4.5,1.5)node{} -- (6,3)node{} -- (6,0)node{};
\draw[ultra thick] (2.5,1.5) -- (2.5,3);
\draw[ultra thick] (3.5,1.5) -- (3.5,3);
\draw (2.5,1.5)node{};
\draw (2.5,3)node{};
\draw (3.5,1.5)node{};
\draw (3.5,3)node{};
\draw (3,-1.7) node[nodelabel]{(e)};
\end{tikzpicture}
\hspace{0.3in}
\begin{tikzpicture}[scale=0.5]
\draw[ultra thick] (0,0) -- (0,3);
\draw[ultra thick] (2,0) -- (2,3);
\draw[ultra thick] (4,0) -- (4,3);
\draw[ultra thick] (6,0) -- (6,3);
\draw (0,0) to [out=120,in=210] (-0.5,4) to [out=30,in=150] (6,3);
\draw (6,0) to [out=60,in=330] (6.5,4) to [out=150,in=30] (0,3);
\draw (0,0)node{} -- (6,0)node{};
\draw (0,3)node{} -- (6,3)node{};
\draw (2,3)node{};
\draw (4,3)node{};
\draw (2,0)node{};
\draw (4,0)node{};
\draw (3,-1.7) node[nodelabel]{(f)};
\end{tikzpicture}
\caption{(a) Odd wheel $W_7$, (b) Prism $P_{10}$,
(c) Truncated biwheel~$T_8$, (d)~Staircase $St_8$,
(e) Staircase $St_{10}$, (f) M{\"o}bius ladder $M_8$}
\label{fig:NT-bricks}
\end{figure}
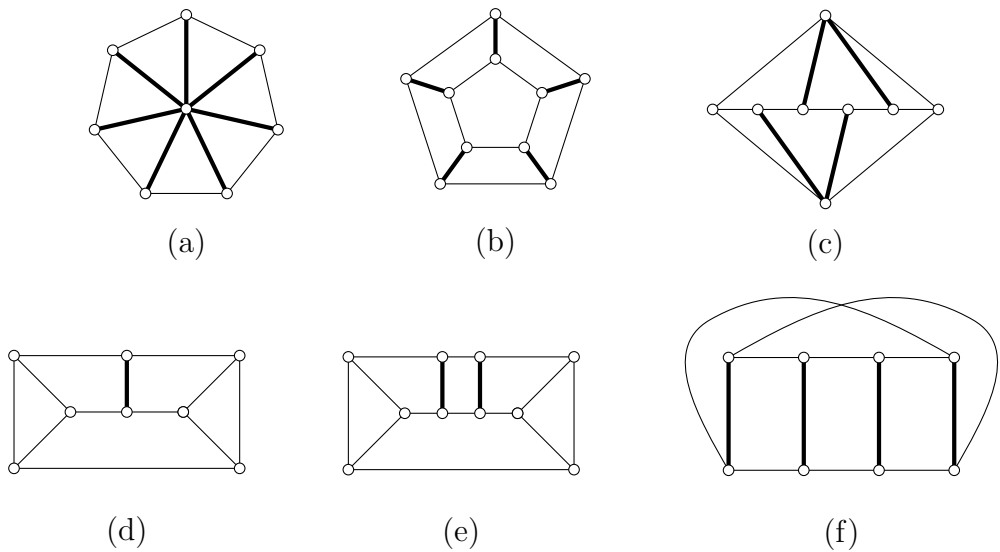

\bigskip
\noindent
{\sc Staircases.}
Let $K[A_1,B_1]$ denote a ladder of order $n$, with external rungs $a_1u_1$ and $b_1w_1$.
Then the graph~$G$ obtained from~$K$, by adding two new vertices
$a_2$ and $b_2$, and by adding five new edges $a_1a_2, u_1a_2, b_1b_2, w_1b_2$
and $a_2b_2$, is called a {\it staircase},
and it is denoted by $St_{n+2}$.
See Figures~\ref{fig:NT-bricks}d and \ref{fig:NT-bricks}e.

\bigskip
Using this terminology,
the theorem of Norine and Thomas \cite{noth07} may be stated as follows.

\begin{thm}
{\sc [Strictly Thin Edge Theorem]}
\label{thm:nt-strictly-thin-bricks}
Let $G$ be a simple brick. If $G$ is free of strictly thin edges then $G$ is
either the Petersen graph, or it is an odd wheel, a prism,
a M{\"o}bius ladder, a truncated biwheel or a staircase.
\end{thm}

It should be noted that Norine and Thomas did not state their results in terms of strictly thin edges.
Subsequently, Carvalho, Lucchesi and Murty \cite{clm08} used their
Thin Edge Theorem (\ref{thm:clm-thin-bricks}) to deduce the
Strictly Thin Edge Theorem (\ref{thm:nt-strictly-thin-bricks}).
The following result of Norine and Thomas \cite{noth07} is
an immediate consequence of Theorem~\ref{thm:nt-strictly-thin-bricks}.
\begin{thm}
\label{thm:nt-simple-brick-reduction}
Given any simple brick~$G$, there exists a sequence $G_1, G_2, \dots, G_k$ of simple bricks such that:
\begin{enumerate}[(i)]
\item $G_1$ is a Norine-Thomas brick,
\item $G_k := G$, and
\item for $2 \leq i \leq k$, there exists a strictly thin edge~$e_i$ of~$G_i$ such that $G_{i-1}$ is the retract of~$G_i - e_i$.
\end{enumerate}
\end{thm}

The above theorem implies that every simple brick can be
generated from one of the Norine-Thomas bricks by means
of four expansion operations as described by Carvalho, Lucchesi and Murty (see \cite{clm06}).
These expansion operations are simply the inverse
of the operation of deleting a strictly thin edge and then taking the retract.

\smallskip
We remark that Norine and Thomas proved a generalization of
Theorem~\ref{thm:nt-simple-brick-reduction},
which they refer to as the `splitter theorem for bricks',
since it is motivated by the splitter theorem for $3$-connected graphs
due to Seymour \cite{seym80}.
The notions of thin and strictly thin edges are easily
generalized to braces (see \cite{clm08}).
A `splitter theorem for braces' was established by McCuaig~\cite{mccu01}.

\subsection{Near-Bipartite Bricks}

A nonbipartite \mcg~$G$ is {\it \nb} if it has a pair $R:=\{\alpha,\beta\}$
of edges such that $H:=G-R$ is matching covered and bipartite.
Such a pair $R$ is called a {\it removable doubleton}.
Furthermore, if $G$ happens to be a brick, we say that $G$ is a
{\it \nb\ brick}. For instance, $K_4$ and $\overline{C_6}$
are \nb\ bricks, and each of them has three distinct removable doubletons.
On the other hand, the Petersen graph is not \nb.
(A result of Carvalho, Lucchesi and Murty \cite{clm02b} implies
that if $G$ is a \nb\ \mcg\ then $b(G)=1$.)

\smallskip
Observe that
the edge $\alpha$ joins two vertices in one color class
of~$H$, and that $\beta$ joins two vertices in the other
color class. Consequently, if $M$ is any perfect matching
of~$G$ then $\alpha \in M$ if and only~if $\beta \in M$.
(In particular, neither $\alpha$ nor $\beta$ is a removable edge of~$G$.)
It is easily verified that every Norine-Thomas brick, except
for the odd wheels and for the Petersen graph, is near-bipartite.

\smallskip
The difficulty in using Theorem~\ref{thm:nt-simple-brick-reduction}
as an induction tool for studying \nb\ bricks,
is that even if $G_k := G$ is a \nb\ brick,
there is no guarantee that all of the intermediate bricks
$G_1, G_2, \dots G_{k-1}$ are also \nb.
For instance, the brick shown in
Figure~\ref{fig:double-biwheel-of-typeI}a
is \nb\ with a (unique)
removable doubleton~\mbox{$R:=\{\alpha,\beta\}$}.
Although the edge~$e$ is strictly thin;
the retract of~$G-e$, as shown in
Figure~\ref{fig:double-biwheel-of-typeI}b,
is not \nb\ since it has three edge-disjoint triangles.

\begin{figure}[!ht]
\centering
\begin{tikzpicture}[scale=0.85]
\draw (0,0) to [out=300,in=180] (4,-2.5) to [out=0,in=240] (8,0);

\draw (0,0) -- (1,0);
\draw (1,0) -- (2,0);
\draw (2,0) -- (3,0);
\draw (5,0) -- (6,0);
\draw (6,0) -- (7,0);
\draw (7,0) -- (8,0);

\draw[ultra thick] (0,0) -- (4,1.5);
\draw (0,0)node[fill=black]{};
\draw (1,0) -- (4,-1.5);
\draw (1,0)node{};
\draw (2,0) -- (4,1.5);
\draw (2,0)node[fill=black]{};
\draw[ultra thick] (3,0) -- (4,-1.5);

\draw[ultra thick] (5,0) -- (4,-1.5);
\draw (6,0) -- (4,1.5);
\draw (6,0)node[fill=black]{};
\draw (7,0) -- (4,-1.5);
\draw (7,0)node{};
\draw (4,-1.5)node[fill=black]{};
\draw[ultra thick] (8,0) -- (4,1.5);
\draw (8,0)node[fill=black]{};
\draw (4,1.5)node{};

\draw (3,0)node{} -- (5,0)node{};

\draw (4,-0.14)node[above,nodelabel]{$\alpha$};
\draw (4,-2.24)node[nodelabel]{$\beta$};

\draw (1.75,0.88)node[nodelabel]{$e$};
\draw (4,-3.2)node[nodelabel]{(a)};
\end{tikzpicture}
\hspace*{0.3in}
\begin{tikzpicture}[scale=0.85]
\draw (2,0) to [out=300,in=170] (4,-1.8) to [out=350,in=210] (8,0);
\draw (8,0) -- (4,-1.5);

\draw (2,0) -- (3,0);
\draw (5,0) -- (6,0);
\draw (6,0) -- (7,0);
\draw (7,0) -- (8,0);

\draw (2,0) -- (4,1.5);
\draw (2,0)node{};
\draw (3,0) -- (4,-1.5);

\draw (5,0) -- (4,-1.5);
\draw (6,0) -- (4,1.5);
\draw (6,0)node{};
\draw (7,0) -- (4,-1.5);
\draw (7,0)node{};
\draw (4,-1.5)node{};
\draw (8,0) -- (4,1.5);
\draw (8,0)node{};
\draw (4,1.5)node{};

\draw (3,0) -- (5,0);
\draw (3,0)node{};
\draw (5,0)node{};
\draw (5,-3.2)node[nodelabel]{(b)};
\end{tikzpicture}
\vspace*{-0.3in}
\caption{(a) A \nb\ brick~$G$ with a thin edge~$e$ ;
(b) The retract of~$G-e$ is not \nb}
\label{fig:double-biwheel-of-typeI}
\end{figure}

\smallskip
In other words, deleting an arbitrary thin edge
may not preserve the property of being \nb. In this sense,
the Thin Edge Theorem~(\ref{thm:clm-thin-bricks}) and the Strictly Thin Edge Theorem~(\ref{thm:nt-strictly-thin-bricks})
are inadequate for obtaining inductive proofs of results
that pertain only to the class of \nb\ bricks.

\smallskip
To fix this problem, the first author started this line of investigation and decided
to look for thin edges whose deletion preserves the property of being \nb.
Kothari \cite{koth16,koth19} proved a `thin edge theorem' for near-bipartite bricks;
in particular, he showed that every near-bipartite brick~$G$ distinct from $K_4$
and $\overline{C_6}$ has a thin edge $e$ such that the retract of $G-e$ is also
near-bipartite (see Theorem~\ref{thm:Rthin-nb-bricks}).
In the present paper, we use this to deduce a `strictly thin edge theorem'
for near-bipartite bricks. This is similar to the approach of Carvalho, Lucchesi and Murty \cite{clm08} ---
they use their Thin Edge Theorem (\ref{thm:clm-thin-bricks}) to deduce
the Strictly Thin Edge Theorem (\ref{thm:nt-strictly-thin-bricks}) of Norine and Thomas.

\smallskip
As in \cite{koth16,koth19}, we find it convenient to fix a removable doubleton~$R$ (of the brick under consideration),
and then look for a strictly thin edge whose deletion preserves this removable doubleton.
To make this precise, we will first define a special type of removable edge which
we call `\Rcomp'.

\subsubsection{\Rcomp\ Edges}
\label{sec:Rcompatible-edges}

We use the abbreviation {\it \Rgraph} for a \nb\ graph~$G$
with (fixed) removable doubleton~$R$, and we
shall refer to $H:=G-R$ as its {\it underlying bipartite graph}.
In the same spirit,
an {\it \Rbrick} is a brick with a removable doubleton~$R$.

\smallskip
A removable edge~$e$ of an \Rgraph~$G$ is {\it \Rcomp} if it is
removable in~$H$ as well.
Equivalently, an edge~$e$ is \Rcomp\ if $G-e$ and $H-e$ 
are both matching covered.
For instance, the graph~$St_8$ shown in Figure~\ref{fig:NT-bricks}d
has two removable doubletons \mbox{$R:=\{\alpha,\beta\}$} and
$R' := \{\alpha', \beta'\}$, and its unique removable edge~$e$ is
\Rcomp\ as well as \comp{R'}.

\smallskip
Now, let $G$ denote the \Rbrick\ shown in
Figure~\ref{fig:double-biwheel-of-typeI}a,
where $R:=\{\alpha,\beta\}$.
The thin edge~$e$ is incident with an edge of~$R$
at a cubic vertex; consequently, $H-e$ has a vertex
whose degree is only one, and so it is not matching covered.
In particular, $e$ is not \Rcomp.

\smallskip
The brick shown in Figure~\ref{fig:pseudo-biwheel}
has two distinct removable doubletons $R:=\{\alpha,\beta\}$
and $R':=\{\alpha',\beta'\}$. Its edges $e$~and~$f$
are both \comp{R'}, but neither of them is \Rcomp.

\begin{figure}[!ht]
\centering
\begin{tikzpicture}[scale=0.85]


\draw (5.3,1.7)node[nodelabel]{$e$};
\draw (7.7,0.3)node[nodelabel]{$f$};

\draw (3,0.77)node[nodelabel]{$\alpha$};
\draw (10,1.25)node[,nodelabel]{$\beta$};
\draw (4,2.15)node[nodelabel]{$\alpha'$};
\draw (9,-0.2)node[nodelabel]{$\beta'$};

\draw (2,1) to [out=80,in=180] (6.5,3.7) to [out=0,in=100] (11,1);

\draw (2,1) -- (4,1);
\draw (9,1) -- (11,1);
\draw (6.5,2.8) -- (2,1);
\draw (6.5,-0.8) -- (11,1);

\draw (2,1)node{};
\draw (11,1)node[fill=black]{};

\draw (4,1) -- (9,1);

\draw (4,1)node{} -- (6.5,2.8);
\draw (5,1)node[fill=black]{} -- (6.5,-0.8);
\draw (6,1)node{} -- (6.5,2.8);
\draw (7,1)node[fill=black]{} -- (6.5,-0.8);
\draw (8,1)node{} -- (6.5,2.8)node[fill=black]{};
\draw (9,1)node[fill=black]{} -- (6.5,-0.8)node{};

\end{tikzpicture}
\caption{$e$ and $f$ are \comp{R'},
but they are not \Rcomp}
\label{fig:pseudo-biwheel}
\end{figure}

\smallskip
Observe that, if $e$ is an \Rcomp\ edge of an \Rgraph~$G$,
then $R$ is a removable doubleton of~$G-e$,
whence $G-e$ is also \nb\ and thus $b(G-e)=1$.
Consequently, every \Rcomp\ edge is \binv.

\smallskip
Furthermore, as shown in \cite{koth16,koth19}, if $e$ is an \Rcomp\ edge of an
\Rbrick~$G$ then the unique brick~$J$ of~$G-e$ is also an \Rbrick;
in particular, $J$ is \nb.
The following is a special case of a
theorem of Carvalho, Lucchesi and Murty \cite{clm99}.

\begin{thm}
{\sc [\Rcomp\ Edge Theorem]}
\label{thm:clm-Rcompatible-nb-bricks}
Every \Rbrick\ distinct from $K_4$ and $\overline{C_6}$ has an \Rcomp\ edge.
\end{thm}

In \cite{clm99},
they proved a stronger result. In particular,
they showed the existence of an \Rcomp\ edge in {\Rgraph}s
with minimum degree at least three.
(They did not use the term `\Rcomp'.)
Using the notion of \mbox{$R$-compatibility}, we
now define a type of thin edge whose deletion preserves
the property of being \nb.

\subsubsection{\Rthin\ and Strictly \Rthin\ Edges}
\label{sec:Rthin-edges}

A thin edge~$e$ of an \Rbrick~$G$
is {\it \Rthin} if it is \Rcomp.
Equivalently, an edge~$e$ is \Rthin\ if it is \Rcomp\ as well as thin,
and in this case, the retract of~$G-e$ is also an \Rbrick.

\smallskip
As noted earlier, the graph $St_8$, shown in Figure~\ref{fig:NT-bricks}d,
has two removable doubletons $R$~and~$R'$. Its unique removable edge~$e$
is \Rthin\ as well as \thin{R'}.
Using the \Rcomp\ Edge Theorem (\ref{thm:clm-Rcompatible-nb-bricks})
of Carvalho, Lucchesi and Murty,
the following `thin edge theorem' was proved by Kothari \cite{koth16,koth19}.
\begin{thm}
{\sc [\Rthin\ Edge Theorem]}
\label{thm:Rthin-nb-bricks}
Every \Rbrick\ distinct from $K_4$ and $\overline{C_6}$ has an \Rthin\ edge.
\end{thm}

Each \Rcomp\ edge of an \Rbrick\ may be associated with two integer
parameters
--- {\it rank} and {\it index} --- as defined in \cite{koth19}.
In fact, Kothari proved the following
stronger
result that immediately implies the
\Rthin\ Edge Theorem (\ref{thm:Rthin-nb-bricks})
since the rank and index are bounded quantities.
\begin{thm}\label{thm:rank-plus-index}
Let $G$ be an \Rbrick\ which is distinct from $K_4$~and~$\overline{C_6}$,
and let $e$ denote an \Rcomp\ edge of~$G$.
Then one of the following alternatives hold:
\begin{itemize}
\item either $e$ is \Rthin,
\item or there exists another \Rcomp\ edge~$f$ such that:
\begin{enumerate}[(i)]
\item $f$ has an end each of whose neighbours in~$G-e$ lies in a barrier of~$G-e$, and
\item ${\sf rank}(f) + {\sf index}(f) > {\sf rank}(e) + {\sf index}(e)$.
\end{enumerate}
\end{itemize}
\end{thm}

An \Rthin\ edge~$e$ of a simple \Rbrick~$G$
is {\it \sRthin} if it is strictly thin.
In other words, a \sRthin\ edge~$e$ is one
which is \Rcomp\ as well as strictly thin;
and in this case, the retract of~$G-e$ is also a simple \Rbrick.

\smallskip
For instance, let $G$ denote the \Rbrick\ shown
in Figure~\ref{fig:strictly-Rthin}(a),
where \mbox{$R:=\{\alpha,\beta\}$}.
The retract of~$G-e$ is the truncated biwheel~$T_8$ shown in
Figure~\ref{fig:strictly-Rthin}(b); consequently, $e$ is \sRthin.

\begin{figure}[!ht]
\centering
\begin{tikzpicture}[scale=1.0]
\draw (0.3,-0.75)node[nodelabel]{$\alpha$};
\draw (4.7,0.75)node[nodelabel]{$\beta$};
\draw[ultra thick] (2.5,1.5) to [out=120,in=90] (-0.7,0) to [out=270,in=240] (2.5,-1.5);
\draw (-0.9,0)node[nodelabel]{$e$};
\draw (0,0) -- (5,0);
\draw (1.5,1.5) -- (3.5,1.5);
\draw (1.5,-1.5) -- (3.5,-1.5);
\draw (0,0) -- (1.5,-1.5)node{} -- (1,0)node[fill=black]{};
\draw (2,0)node{} -- (1.5,1.5)node[fill=black]{} -- (0,0)node{};
\draw (3,0)node[fill=black]{} -- (3.5,-1.5)node{} -- (5,0);
\draw (4,0)node{} -- (3.5,1.5)node[fill=black]{} -- (5,0)node[fill=black]{};
\draw (2.5,1.5)node{};
\draw (2.5,-1.5)node[fill=black]{};
\draw (2.5,-2.5)node[nodelabel]{(a)};
\end{tikzpicture}
\hspace*{0.5in}
\begin{tikzpicture}[scale=1.0]
\draw (0.5,-0.85)node[nodelabel]{$\alpha$};
\draw (4.5,0.85)node[nodelabel]{$\beta$};
\draw (2.5,-2.5)node[nodelabel]{(b)};
\draw (0,0) -- (5,0);
\draw (0,0) -- (2.5,1.5);
\draw[dashed] (2,0) -- (2.5,1.5);
\draw (2,0)node{};
\draw[dashed] (4,0) -- (2.5,1.5);
\draw (4,0)node{};
\draw (5,0) -- (2.5,1.5)node[fill=black]{};
\draw (0,0)node{} -- (2.5,-1.5);
\draw[dashed] (1,0) -- (2.5,-1.5);
\draw (1,0)node[fill=black]{};
\draw[dashed] (3,0) -- (2.5,-1.5);
\draw (3,0)node[fill=black]{};
\draw (5,0)node[fill=black]{} -- (2.5,-1.5)node{};
\end{tikzpicture}
\vspace*{-0.15in}
\caption{Edge $e$ is \sRthin}
\label{fig:strictly-Rthin}
\bigskip
\end{figure}

Recall that the Norine-Thomas bricks are precisely those simple bricks
which are free of strictly thin edges.
In particular, every $R$-brick, which is a member of the Norine-Thomas families,
is free of \sRthin\ edges.
A natural question arises as to whether there are any simple {\Rbrick}s, different
from the Norine-Thomas bricks, which are also free of \sRthin\ edges.
It turns out that there indeed are such bricks;
we have already encountered two examples
in Figures~\ref{fig:double-biwheel-of-typeI}a and \ref{fig:pseudo-biwheel},
as explained below.

\smallskip
Let $G$ denote the \Rbrick, shown in
Figure~\ref{fig:double-biwheel-of-typeI}a, where $R:=\{\alpha, \beta\}$
is its unique removable doubleton. It can be checked that
$G$ has precisely four strictly thin edges, depicted by bold lines;
these are similar under the automorphisms of the graph.
As noted earlier, if $e$ is any of these edges, then $e$
is not \Rcomp; furthermore, the retract of~$G-e$
is isomorphic to the graph shown in Figure~\ref{fig:double-biwheel-of-typeI}b,
which is not even near-bipartite as it has three edge-disjoint triangles.
Thus, the generation of~$G$ using the
Norine-Thomas procedure
cannot be achieved within the class of \nb\ bricks.

\smallskip
Now, let $G$ denote the brick shown in Figure~\ref{fig:pseudo-biwheel};
it has two removable doubletons $R:=\{\alpha,\beta\}$
and $R':=\{\alpha',\beta'\}$. It may be verified that
$G$ has precisely two strictly thin edges, namely $e$ and $f$,
each of which is \comp{R'} but neither is \Rcomp.
In particular, $G$ is free of \sRthin\ edges; in this sense it is similar
to the graph in Figure~\ref{fig:double-biwheel-of-typeI}a.
On the other hand, $G$ has strictly \thin{R'}\ edges;
if $e$ is any such edge then
the retract of~$G-e$ is a simple \nb\ brick with removable
doubleton~$R'$.
In this sense,
$G$ is different from the graph in Figure~\ref{fig:double-biwheel-of-typeI}.

\subsection{Families of {\Rbrick}s free of Strictly \Rthin\ Edges}
\label{sec:new-families}

\smallskip
We will introduce seven infinite families of simple {\Rbrick}s
which are free of \sRthin\ edges, and are different from the Norine-Thomas families.
The members of these
will be described using their specific bipartite subgraphs, each of which is either
a ladder or a partial biwheel; see Figure~\ref{fig:biwheels-ladders}.
The occurrence of these subgraphs may be justified as follows.
Let $G$ be a simple \Rbrick\ which is free of
\sRthin\ edges. If $e$ is any \Rthin\ edge of~$G$, at least one end of~$e$
is cubic and the retract of~$G-e$ has multiple edges. These strictures
can be used to deduce that $G$ contains either a ladder or a partial biwheel,
or both, as subgraphs.

\smallskip
In our descriptions of these families, we use $\alpha$~and~$\beta$ to denote
the edges of the (fixed) removable doubleton~$R$. Apart from~$R$, a member
may have at most one removable doubleton which will be denoted as~$R':=\{\alpha',\beta'\}$.
We adopt the notational conventions stated in Section~\ref{sec:NT-bricks}.
(Recall that a partial biwheel of order six is also a ladder; for this reason, some
of our families overlap.)

\bigskip
\noindent
{\sc Pseudo-Biwheels.}
Let $K[A_1,B_1]$ denote a partial biwheel, of order at least eight, and
with external spokes $a_1u_1$ and $b_1w_1$.
Then the graph~$G$ obtained from~$K$, by adding two new vertices
$a_2$ and $b_2$, and by adding five new edges
$\alpha:=a_1a_2, \alpha':=u_1a_2, \beta:=b_1b_2, \beta':=w_1b_2$
and $a_2b_2$, is called a {\it pseudo-biwheel}.
Figure~\ref{fig:pseudo-biwheel} shows the smallest pseudo-biwheel.

\bigskip
It is worth comparing the above with our desription of staircases
in Section~\ref{sec:NT-bricks}.
Although a pseudo-biwheel $G$ is free of \sRthin\ edges,
the two external spokes of~$K$, namely $a_1u_1$ and $b_1w_1$,
are both strictly \thin{R'}.

\bigskip
In order to describe the members of the remaining six families,
we need two (sub)graphs.
For \mbox{$i \in \{1,2\}$}, let $K_i[A_i,B_i]$ denote either a
ladder or a partial biwheel with external rungs/spokes $a_iu_i$ and $b_iw_i$,
such that $K_1$ and $K_2$ are disjoint.

\begin{figure}[!ht]
\bigskip
\centering
\begin{tikzpicture}[scale=0.7]


\draw (0,5.4)node[nodelabel]{$a_2$};
\draw (4.5,5.4)node[nodelabel]{$b_2$};

\draw (0,-0.4)node[nodelabel]{$a_1$};
\draw (4.5,-0.4)node[nodelabel]{$b_1$};

\draw (-1.6,2.5)node[nodelabel]{$\alpha$};
\draw (6.1,2.5)node[nodelabel]{$\beta$};

\draw[thick] (0,0) to [out=150,in=210] (0,5);
\draw[thick] (4.5,0) to [out=30,in=330] (4.5,5);

\draw (0,5)node{} -- (4.5,5)node[fill=black]{} -- (4.5,2.5) -- (3,3.5) -- (1.5,3.5) -- (0,2.5) -- (0,5);
\draw (1.5,5)node[fill=black]{} -- (1.5,3.5)node{};
\draw (3,5)node{} -- (3,3.5)node[fill=black]{};

\draw[dashed] (0,0) -- (4.5,0) -- (4.5,2.5) -- (3,1.5) -- (1.5,1.5) -- (0,2.5) -- (0,0);
\draw[dashed] (1.5,0) -- (1.5,1.5);
\draw[dashed] (3,0) -- (3,1.5);

\draw (0,0)node{};
\draw (4.5,0)node[fill=black]{};
\draw (4.5,2.5)node{};
\draw (0,2.5)node[fill=black]{};
\draw (1.5,0)node[fill=black]{};
\draw (1.5,1.5)node{};
\draw (3,0)node{};
\draw (3,1.5)node[fill=black]{};

\draw (2.25, -1.65)node[nodelabel]{(a)};
\end{tikzpicture}
\hspace*{0.2in}
\begin{tikzpicture}[scale=0.7]

\draw (-1.2,1.5)node[nodelabel]{$\alpha$};
\draw (6.7,2.25)node[nodelabel]{$\beta$};

\draw (0,0) to [out=180,in=180] (0,3);
\draw (5,0) to [out=0,in=0] (5,4.5);

\draw (0,0) -- (5,0);
\draw (2.5,1.5) -- (1,0)node[fill=black]{};
\draw (2.5,1.5) -- (3,0)node[fill=black]{};
\draw (2.5,1.5)node{} -- (5,0)node[fill=black]{};
\draw (2.5,-1.5) -- (0,0)node{};
\draw (2.5,-1.5) -- (2,0)node{};
\draw (2.5,-1.5)node[fill=black]{} -- (4,0)node{};

\draw (0,0.4)node[nodelabel]{$a_2$};
\draw (5,-0.4)node[nodelabel]{$b_2$};

\draw (2.5,1.85)node[nodelabel]{$w_2$};
\draw (2.5,-1.9)node[nodelabel]{$u_2$};

\draw (0,3) -- (5,3)node{} -- (5,4.5)node[fill=black]{} -- (0,4.5)node[fill=black]{} --
(0,3)node{};
\draw (5,2.6)node[nodelabel]{$w_1$};
\draw (5,4.85)node[nodelabel]{$b_1$};
\draw (0,4.85)node[nodelabel]{$u_1$};
\draw (0,2.6)node[nodelabel]{$a_1$};

\draw (1.25,3)node[fill=black]{} -- (1.25,4.5)node{};
\draw (2.5,3)node{} -- (2.5,4.5)node[fill=black]{};
\draw (3.75,3)node[fill=black]{} -- (3.75,4.5)node{};

\draw (2.5,-3.2)node[nodelabel]{(b)};
\end{tikzpicture}
\caption{(a) A double ladder of type~I ; (b) A laddered biwheel of type~I
is obtained by identifying $u_1$ with $u_2$ and likewise $w_1$ with $w_2$}
\label{fig:typeI-families}
\bigskip
\end{figure}
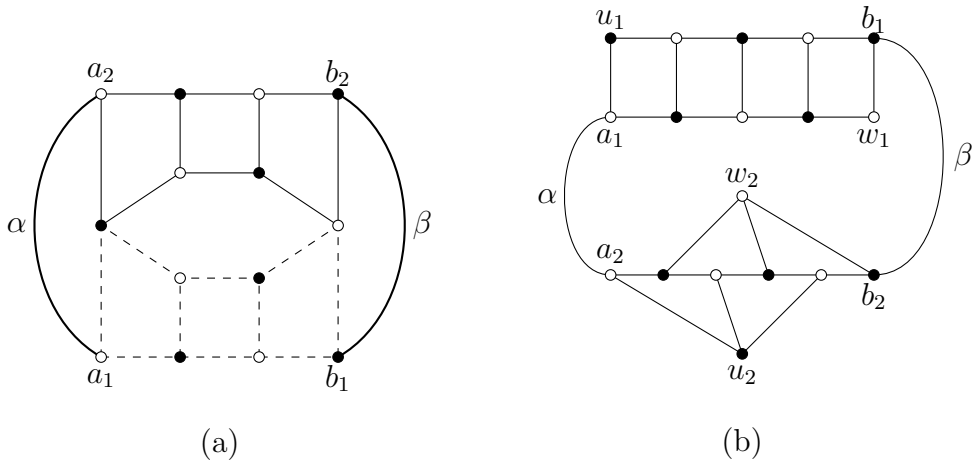

\bigskip
\noindent
{\sc Double Biwheels, Double Ladders and Laddered Biwheels of Type~I.}
Let the graph~$G$ be obtained from $K_1 \cup K_2$,
by adding edges $\alpha:=a_1a_2$ and $\beta:=b_1b_2$, by
identifying vertices $u_1$~and~$u_2$,
and by identifying vertices $w_1$~and~$w_2$.
There are three possibilities depending on the graphs $K_1$ and $K_2$.
In the case in which $K_1$ and $K_2$ are
both partial biwheels, $G$ is a {\it double biwheel of type~I}. Likewise,
in the case in which $K_1$ and $K_2$ are both ladders, $G$ is a
{\it double ladder of type~I}. Finally, when one of $K_1$ and $K_2$ is a partial
biwheel and the other one is a ladder, $G$ is a
{\it laddered biwheel of type~I}.

\smallskip
A member of any of these families has a unique removable doubleton~$R$,
and is free of \sRthin\ edges.
The graph in Figure~\ref{fig:double-biwheel-of-typeI}a is the smallest
member of each of these families,
although its drawing is suggestive of a double biwheel.
Figure~\ref{fig:typeI-families}a shows a double ladder.
A laddered biwheel is obtained
from the graph in Figure~\ref{fig:typeI-families}b by identifying
$u_1$~with~$u_2$, and likewise, $w_1$~with~$w_2$.

\bigskip
\noindent
{\sc Double Biwheels, Double Ladders and Laddered Biwheels of type~II.}
Let the graph~$G$ be obtained
from $K_1 \cup K_2$, by adding four edges,
namely, $\alpha:=a_1a_2$, \mbox{$\beta:=b_1b_2$},
$\alpha':=u_1w_2$ and $\beta':=w_1u_2$.
As before, we have three possibilities.
In the case in which $K_1$~and~$K_2$ are
both partial biwheels of order at least eight,
$G$ is a {\it double biwheel of type~II}. Likewise,
in the case in which $K_1$~and~$K_2$ are both ladders, $G$ is a
{\it double ladder of type~II}.
Finally, when one of $K_1$~and~$K_2$ is a partial
biwheel of order at least eight, and the other one is a ladder,
$G$ is a {\it laddered biwheel of type~II}.

\bigskip
A member of any of these families has two removable doubletons $R$ and $R'$,
and it is free of \sRthin\ edges.
However, a double biwheel or a laddered biwheel as shown in
Figure~\ref{fig:typeII-families} has \sthin{R'}\ edges;
these are the external spokes of a partial
biwheel of order at least eight as depicted by the bold lines in the figure.

\begin{figure}[!ht]
\centering
\begin{tikzpicture}[scale=0.62]



\draw (4.5,-2.75)node[above,nodelabel]{$\alpha$};
\draw (3.1,1.9)node[nodelabel]{$\beta$};
\draw (3.5,3.4)node[nodelabel]{$\alpha'$};
\draw (3.85,-0.8)node[nodelabel]{$\beta'$};

\draw[dashed] (0,0) to [out=270,in=180] (4.5,-1.8) to [out=0,in=270] (9,1);
\draw[dashed] (0,2) to [out=40,in=180] (6.5,2.8);

\draw[dashed] (2,2) -- (4,1);
\draw[dashed] (2,0) -- (6.5,-0.8);

\draw (0,0) -- (2,0)node{} -- (2,2)node[fill=black]{} -- (0,2)node[fill=black]{} -- (0,0)node{};
\draw (1,0)node[fill=black]{} -- (1,2)node{};

\draw (4,1) -- (9,1);

\draw[ultra thick] (4,1) -- (6.5,2.8);
\draw[ultra thick] (9,1) -- (6.5,-0.8);
\draw (4,1)node[fill=black]{} -- (6.5,2.8);
\draw (5,1)node{} -- (6.5,-0.8);
\draw (6,1)node[fill=black]{} -- (6.5,2.8);
\draw (7,1)node{} -- (6.5,-0.8);
\draw (8,1)node[fill=black]{} -- (6.5,2.8)node{};
\draw (9,1)node{} -- (6.5,-0.8)node[fill=black]{};

\draw (4.5,-3.6)node[nodelabel]{(a)};
\end{tikzpicture}
\hspace*{0.15in}
\begin{tikzpicture}[scale=0.62]

\draw (9.75,0.65)node[nodelabel]{$\alpha$};
\draw (9.75,-1.2)node[nodelabel]{$\alpha'$};
\draw (9.75,2.4)node[nodelabel]{$\beta'$};
\draw (9.75,4.35)node[nodelabel]{$\beta$};

\draw[dashed] (9,1) -- (10.5,1);
\draw[dashed] (6.5,2.8) -- (13,2.8);
\draw[dashed] (6.5,-0.8) -- (13,-0.8);
\draw[dashed] (4,1) to [out=90,in=180] (9.75,4) to [out=0,in=90] (15.5,1);

\draw (4,1) -- (9,1);

\draw[ultra thick] (4,1) -- (6.5,2.8);
\draw[ultra thick] (9,1) -- (6.5,-0.8);
\draw (4,1)node[fill=black]{} -- (6.5,2.8);
\draw (5,1)node{} -- (6.5,-0.8);
\draw (6,1)node[fill=black]{} -- (6.5,2.8);
\draw (7,1)node{} -- (6.5,-0.8);
\draw (8,1)node[fill=black]{} -- (6.5,2.8)node{};
\draw (9,1)node{} -- (6.5,-0.8)node[fill=black]{};

\draw (10.5,1) -- (15.5,1);

\draw[ultra thick] (10.5,1) -- (13,2.8);
\draw[ultra thick] (15.5,1) -- (13,-0.8);
\draw (10.5,1)node{} -- (13,2.8);
\draw (11.5,1)node[fill=black]{} -- (13,-0.8);
\draw (12.5,1)node{} -- (13,2.8);
\draw (13.5,1)node[fill=black]{} -- (13,-0.8);
\draw (14.5,1)node{} -- (13,2.8)node[fill=black]{};
\draw (15.5,1)node[fill=black]{} -- (13,-0.8)node{};

\draw (9.75,-2.7)node[nodelabel]{(b)};
\end{tikzpicture}
\caption{(a) A laddered biwheel of type II ;
(b) A double biwheel of type II}
\label{fig:typeII-families}
\end{figure}
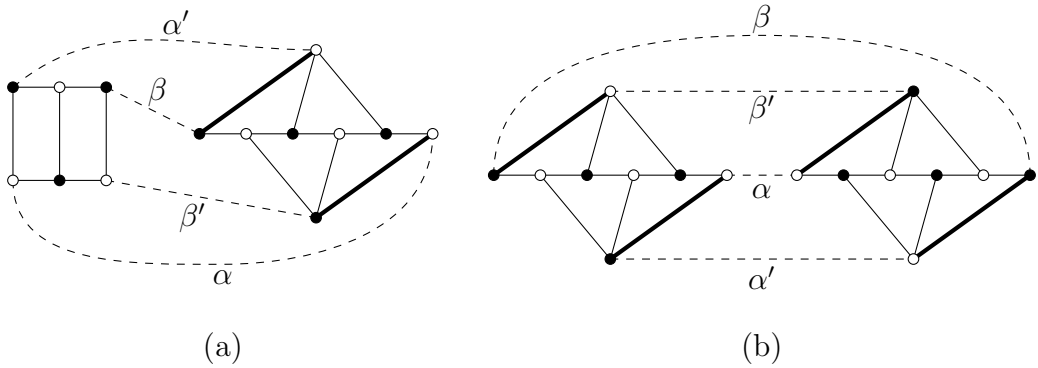

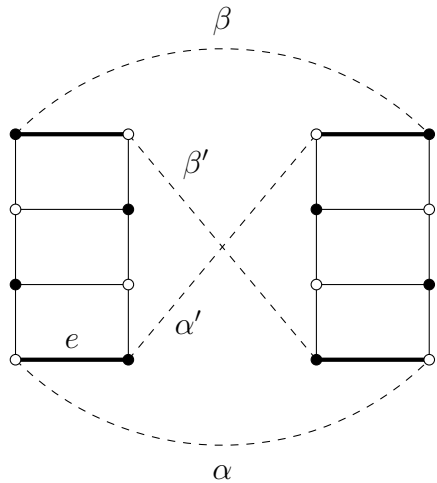
\begin{figure}[!ht]
\bigskip
\centering
\begin{tikzpicture}[scale=1]
\draw (0.75,0.25)node[nodelabel]{$e$};

\draw[ultra thick] (0,0) -- (1.5,0);
\draw[ultra thick] (0,3) -- (1.5,3);
\draw[ultra thick] (4,0) -- (5.5,0);
\draw[ultra thick] (4,3) -- (5.5,3);

\draw (2.75,-1.5)node[nodelabel]{$\alpha$};
\draw (2.75,4.5)node[nodelabel]{$\beta$};
\draw (2.3,0.5)node[nodelabel]{$\alpha'$};
\draw (2.4,2.6)node[nodelabel]{$\beta'$};

\draw[dashed] (1.5,0) -- (4,3);
\draw[dashed] (1.5,3) -- (4,0);
\draw[dashed] (0,0) to [out=315,in=225] (5.5,0);
\draw[dashed] (0,3) to [out=45,in=135] (5.5,3);

\draw (0,0) -- (1.5,0)node[fill=black]{} -- (1.5,3)node{} --
(0,3)node[fill=black]{} -- (0,0)node{};
\draw (0,1)node[fill=black]{} -- (1.5,1)node{};
\draw (0,2)node{} -- (1.5,2)node[fill=black]{};

\draw (4,0) -- (5.5,0)node{} -- (5.5,3)node[fill=black]{} --
(4,3)node{} -- (4,0)node[fill=black]{};
\draw (4,1)node{} -- (5.5,1)node[fill=black]{};
\draw (4,2)node[fill=black]{} -- (5.5,2)node{};
\end{tikzpicture}
\caption{A double ladder of type II}
\label{fig:double-ladder-of-typeII}
\end{figure}

On the other hand, a double ladder,
as shown in Figure~\ref{fig:double-ladder-of-typeII},
is free of \sthin{R'}\ edges as well.
This may be explained as follows. Every double ladder is cubic, and it
has precisely four strictly thin edges;
these are the external rungs of the two ladders,
depicted by bold lines in the figure.
One end of any such edge, say~$e$, 
is incident with an edge of~$R$
and the other end is incident with an edge of~$R'$;
since each end of~$e$ is cubic,
it is neither \Rcomp\ nor \comp{R'}.

\smallskip
Using a strengthening (see Theorem~\ref{thm:rank-plus-index}) of the \Rthin\ Edge
Theorem~(\ref{thm:Rthin-nb-bricks}), we
will prove that the seven
families described above
and four of the Norine-Thomas families are
the only simple {\Rbrick}s which are free
of \sRthin\ edges.
\begin{thm}
{\sc [Strictly \Rthin\ Edge Theorem]}
\label{thm:strictly-Rthin-nb-bricks}
Let $G$ be a simple \Rbrick. If $G$ is free of \sRthin\ edges
then $G$ belongs to one of the following infinite families:
\begin{multicols}{2}
\begin{enumerate}[(i)]
\item Truncated biwheels
\item Prisms
\item M{\"o}bius ladders
\item Staircases
\item Pseudo-biwheels
\item Double biwheels of type I
\item Double ladders of type I
\item Laddered biwheels of type I
\item Double biwheels of type II
\item Double ladders of type II
\item Laddered biwheels of type II
\end{enumerate}
\end{multicols}
\end{thm}

We present a proof of the above theorem in
Section \ref{sec:proof-of-strictly-Rthin-edge-theorem}.
As mentioned earlier, our proof is inspired by the proof of the
Strictly Thin Edge Theorem (\ref{thm:nt-strictly-thin-bricks})
given by Carvalho et al. \cite{clm08}, and
uses several of their results and techniques.

\smallskip
We shall denote by $\mathcal{N}$ the
union of all of the eleven families which appear in
the statement of Theorem~\ref{thm:strictly-Rthin-nb-bricks}.
The following is an immediate consequence.

\begin{thm}
\label{thm:simple-nb-brick-reduction}
Given any simple \Rbrick~$G$,
there exists a sequence $G_1, G_2, \dots, G_k$ of simple {\Rbrick}s such that:
\begin{enumerate}[(i)]
\item $G_1 \in \mathcal{N}$,
\item $G_k:=G$, and
\item for $2 \leq i \leq k$, there exists an \Rthin\ edge~$e_i$
of~$G_i$ such that $G_{i-1}$ is the retract of~$G_i - e_i$.
\end{enumerate}
\end{thm}

In other words, every simple \Rbrick\ can be generated from some 
member of~$\mathcal{N}$ by repeated application of the expansion operations
such that at each step we have a simple \Rbrick.

\smallskip
Finally, recall that members of three of the aforementioned families
do have \sthin{R'}\ edges,
where $R':=\{\alpha',\beta'\}$ in our description of these families;
these are pseudo-biwheels, double biwheels of type~II and
laddered biwheels of type~II.
In view of this, we say that a strictly thin edge~$e$ of a simple \nb\ brick~$G$
is {\it compatible} if it is \Rcomp\ for some removable doubleton~$R$.
We thus have the following theorem (with eight infinite families)
alluded to in the abstract.

\begin{thm}
\label{thm:compatible-strictly-thin-nb-bricks}
Let $G$ be a simple near-bipartite brick.
If $G$ is free of compatible strictly thin edges then $G$ belongs
to one of the following infinite families:
\begin{multicols}{2}
\begin{enumerate}[(i)]
\item Truncated biwheels
\item Prisms
\item M{\"o}bius ladders
\item Staircases
\item Double biwheels of type I
\item Double ladders of type I
\item Laddered biwheels of type I
\item Double ladders of type II
\end{enumerate}
\end{multicols}
\end{thm}

Four of the families in the above theorem are
Norine-Thomas families; these are free of strictly thin edges.
As we did in Figure~\ref{fig:double-biwheel-of-typeI},
it may be verified that if $G$ is a member of any of the remaining four
families and $e$ is any strictly thin edge of~$G$ then the retract~$J$ of~$G-e$
is not near-bipartite.
(For example, consider the graph~$G$ and edge~$e$ shown in
Figure~\ref{fig:double-ladder-of-typeII}, and let $J$ be the retract of~$G-e$.
It can be checked that $J$ has four odd cycles, $C_0, C_1, C_2$ and $C_3$,
such that $C_1, C_2$ and $C_3$ are edge-disjoint with $C_0$, and furthermore,
there is no single edge which belongs to all three of them.)

\bigskip
For the rest of this paper, our goal is to present a complete proof of
Theorem~\ref{thm:strictly-Rthin-nb-bricks}.
This result and its proof also appear in the Ph.D. thesis of the first author \cite[Chapter 6]{koth16}.

\noindent
{\bf Organization of this paper:}

\smallskip
\noindent
In Section~\ref{sec:R-configurations}, we define two special types of subgraphs, namely, an `\Rbiwheel' and an `\Rladder',
we state a few theorems related to these configurations without proofs,
and we conclude
with a proof sketch of Theorem~\ref{thm:strictly-Rthin-nb-bricks}. In Sections~\ref{sec:R-thin} and \ref{sec:properties-of-Rconfigurations}, we prove the
technical results and theorems that are stated in Section~\ref{sec:R-configurations}. Finally, in Section~\ref{sec:proof-of-strictly-Rthin-edge-theorem},
we provide a complete proof
of Theorem~\ref{thm:strictly-Rthin-nb-bricks}.

\section{{\Rconf}s}\label{sec:R-configurations}
\label{sec:R-configurations}

Recall the definitions of ladders and partial biwheels from
Section~\ref{sec:NT-bricks}.
In our descriptions of the eleven families that appear in the statement of Theorem~\ref{thm:strictly-Rthin-nb-bricks},
we constructed their members using either one or two
disjoint bipartite matching covered
graphs, each of which is either a ladder or a partial biwheel, and thereafter,
adding a few vertices and/or edges and possibly identifying two pairs of vertices.
As we will see,
these constructions are indicative of how these graphs appear
in our proof of Theorem~\ref{thm:strictly-Rthin-nb-bricks}.
In this section, we will define two special types of subgraphs,
namely, an `\Rbiwheel' and an `\Rladder';
we will conclude with a proof sketch of Theorem~\ref{thm:strictly-Rthin-nb-bricks}.

\smallskip
For the rest of this paper, we adopt the following notational and figure conventions.

\begin{Not}
\label{Not:Rbrick-doubleton}
For a simple \Rbrick~$G$, we shall denote by $H[A,B]$ the underlying bipartite
graph $G-R$. We let $\alpha$ and $\beta$ denote the constituent edges
of~$R$, and we adopt the convention that $\alpha:=a_1a_2$
has both ends in~$A$, whereas $\beta:=b_1b_2$ has both
ends in~$B$. We denote by~$V(R)$ the set~$\{a_1,a_2,b_1,b_2\}$.
Furthermore, in all of the figures, the hollow vertices are in~$A$,
and the solid vertices are in~$B$.
\end{Not}

We will also adopt the following notational conventions for
a subgraph which is either a ladder or a partial biwheel.

\begin{Not}
\label{Not:biwheel-ladder-convention}
When referring to a subgraph~$K$ of~$H$,
such that $K$ is either a ladder
or a partial biwheel with
\external\ $au$ and $bw$, we adopt the convention that
$a,w \in A$ and $b, u \in B$; furthermore, when $K$ is a partial biwheel,
$u$~and~$w$ shall denote its hubs; as shown in
Figures~\ref{fig:Rbiwheel-configuration}~and~\ref{fig:Rladder-configuration}.
(We may also use subscript
notation, such as $a_iu_i$ and $b_iw_i$ where $i$ is an integer, and this
convention extends naturally.)
\end{Not}

\subsection{{\Rbiwheel}s}
\label{sec:Rbiwheel-configurations}

Let $K$ be a subgraph of $H$ such that~$K$ is a partial biwheel
with external spokes $au$~and~$bw$; see Figure~\ref{fig:Rbiwheel-configuration}.
We say that $K$ is an {\it \Rbiwheel} of~$G$
if it satisfies the following conditions:

\begin{enumerate}[{\it (i)}]
\item in~$G$, the hubs $u$ and $w$ are both noncubic, and
every other vertex of~$K$ is cubic,
\item the ends of~$K$, namely $a$~and~$b$, both lie in~$V(R)$, and,
\item in~$G$, every internal spoke of~$K$ is an \Rthin\ edge whose index is one.
\end{enumerate}

\begin{figure}[!ht]
\centering
\begin{tikzpicture}[scale=0.8]
\draw (1,0) -- (6,0);

\draw (1,0) -- (3.5,2);
\draw (2,0) -- (3.5,-2);
\draw (3,0) -- (3.5,2);
\draw (4,0) -- (3.5,-2);
\draw (5,0) -- (3.5,2);
\draw (6,0) -- (3.5,-2);

\draw (1,0) node{}node[left,nodelabel]{$a \in V(R)$};
\draw (2,0) node[fill=black]{};
\draw (3,0) node{};
\draw (4,0) node[fill=black]{};
\draw (5,0) node{};
\draw (6,0) node[fill=black]{}node[right,nodelabel]{$b \in V(R)$};
\draw (3.5,2) node[fill=black]{}node[above,nodelabel]{$u$};
\draw (3.5,-2) node{}node[below,nodelabel]{$w$};

\end{tikzpicture}
\vspace*{-0.1in}
\caption{An \Rbiwheel; in~$G$, the free corners (hubs) $u$ and $w$ are
noncubic, and every other vertex is cubic}
\label{fig:Rbiwheel-configuration}
\bigskip
\end{figure}
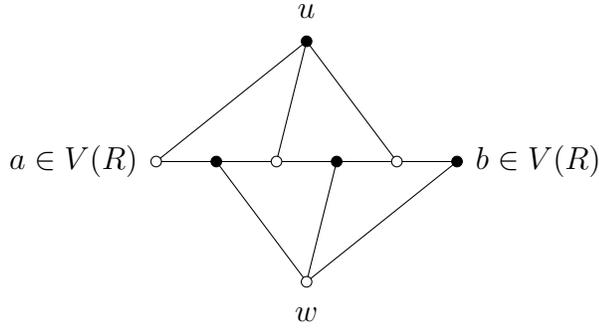

A pseudo-biwheel, as shown in Figure~\ref{fig:pseudo-biwheel-Rconfiguration},
has two removable doubletons~$R:=\{\alpha,\beta\}$
and $R':=\{\alpha',\beta'\}$. The subgraph~$K$, depicted by solid lines,
is an $R$-biwheel configuration. (To see this, note that every internal
spoke of~$K$ is an \Rthin\ edge of index one.)
However, $K$ is not an $R'$-biwheel configuration
because its ends $a$ and $b$ are not incident with edges of~$R'$.

\begin{figure}[!ht]
\bigskip
\centering
\begin{tikzpicture}[scale=1]

\draw (4,0.7)node[nodelabel]{$a$};
\draw (9,1.3)node[nodelabel]{$b$};
\draw (6.5,3.1)node[nodelabel]{$u$};
\draw (6.5,-1.1)node[nodelabel]{$w$};

\draw (3,0.77)node[nodelabel]{$\alpha$};
\draw (10,1.25)node[,nodelabel]{$\beta$};
\draw (4,2.15)node[nodelabel]{$\alpha'$};
\draw (9,-0.2)node[nodelabel]{$\beta'$};

\draw[dashed] (2,1) to [out=80,in=180] (6.5,4) to [out=0,in=100] (11,1);

\draw[dashed] (2,1) -- (4,1);
\draw[dashed] (9,1) -- (11,1);
\draw[dashed] (6.5,2.8) -- (2,1);
\draw[dashed] (6.5,-0.8) -- (11,1);

\draw (2,1)node{};
\draw (11,1)node[fill=black]{};

\draw (4,1) -- (9,1);

\draw (4,1)node{} -- (6.5,2.8);
\draw (5,1)node[fill=black]{} -- (6.5,-0.8);
\draw (6,1)node{} -- (6.5,2.8);
\draw (7,1)node[fill=black]{} -- (6.5,-0.8);
\draw (8,1)node{} -- (6.5,2.8)node[fill=black]{};
\draw (9,1)node[fill=black]{} -- (6.5,-0.8)node{};

\end{tikzpicture}
\vspace*{-0.1in}
\caption{A pseudo-biwheel has only one $R$-biwheel configuration}
\label{fig:pseudo-biwheel-Rconfiguration}
\end{figure}
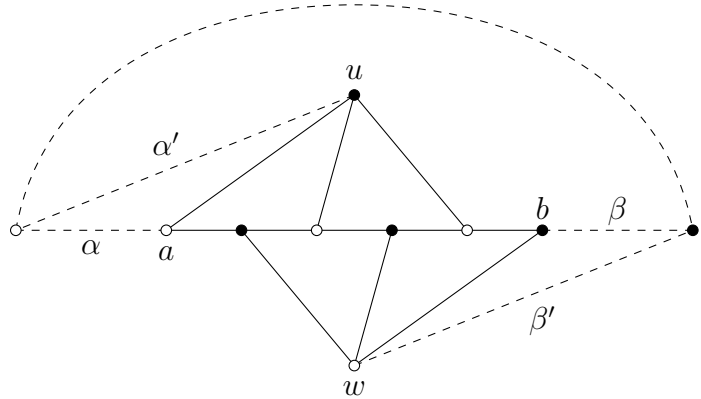

\subsection{{\Rladder}s}
\label{sec:Rladder-configurations}

Let $K$ be a subgraph of~$H$ such that $K$ is a ladder with external rungs
$au$ and $bw$; see Figure~\ref{fig:Rladder-configuration}.
We say that $K$ is an {\it \Rladder} of~$G$ if it satisfies the following conditions:

\begin{enumerate}[{\it (i)}]
\item in~$G$, every vertex of~$K$, except possibly for~$u$ and $w$, is cubic,
\item the vertices $a$ and $b$ both lie in~$V(R)$, and,
\item in~$G$, every internal rung of~$K$ is an \Rthin\ edge whose index is two.
\end{enumerate}

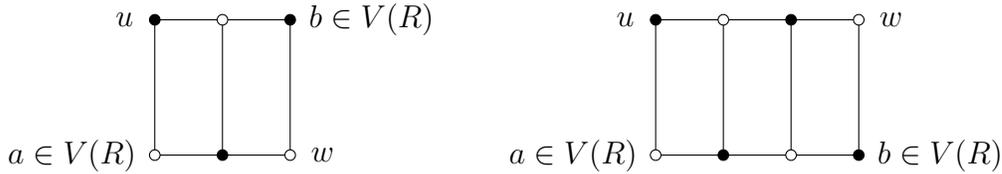
\begin{figure}[!ht]
\centering
\vspace*{-0.23in}
\begin{tikzpicture}[scale=0.9]

\draw (1,0) -- (1,2);
\draw (2,0) -- (2,2);
\draw (3,0) -- (3,2);

\draw (1,0) -- (3,0);
\draw (1,2) -- (3,2);

\draw (1,0) node{}node[left,nodelabel]{$a \in V(R)$};
\draw (1,2) node[fill=black]{}node[left,nodelabel]{$u$};
\draw (2,2) node{};
\draw (2,0) node[fill=black]{};
\draw (3,0) node{}node[right,nodelabel]{$w$};
\draw (3,2) node[fill=black]{}node[right,nodelabel]{$b \in V(R)$};

\end{tikzpicture}
\hspace*{0.1in}
\begin{tikzpicture}[scale=0.9]

\draw (1,0) -- (1,2);
\draw (2,0) -- (2,2);
\draw (3,0) -- (3,2);
\draw (4,0) -- (4,2);

\draw (1,0) -- (4,0);
\draw (1,2) -- (4,2);

\draw (1,0) node{}node[left,nodelabel]{$a \in V(R)$};
\draw (1,2) node[fill=black]{}node[left,nodelabel]{$u$};
\draw (2,2) node{};
\draw (2,0) node[fill=black]{};
\draw (3,0) node{};
\draw (3,2) node[fill=black]{};
\draw (4,2) node{}node[right,nodelabel]{$w$};
\draw (4,0) node[fill=black]{}node[right,nodelabel]{$b \in V(R)$};

\end{tikzpicture}
\vspace*{-0.2in}
\caption{Two {\Rladder}s of different parities; each vertex, except possibly for the free corners $u$~and~$w$, is cubic in~$G$}
\label{fig:Rladder-configuration}
\bigskip
\end{figure}

A prism of order~$n$ has $\frac{n}{2}$ removable doubletons.
If $R:=\{\alpha,\beta\}$
is a fixed removable doubleton of a prism~$G$ of order ten or more,
then the graph~$H=G-R$ is itself an
\Rladder, as shown in Figure~\ref{fig:prism-Rconfiguration}.
(An analogous statement holds for M{\"o}bius ladders of order eight or more.)
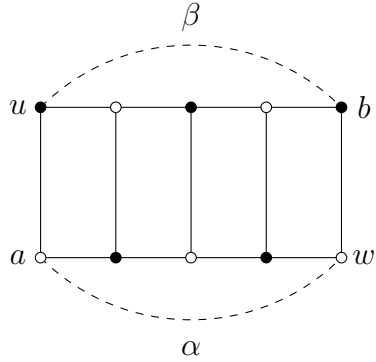
\begin{figure}[!ht]
\centering
\begin{tikzpicture}

\draw (-0.3,0)node[nodelabel]{$a$};
\draw (4.3,0)node[nodelabel]{$w$};
\draw (-0.3,2)node[nodelabel]{$u$};
\draw (4.3,2)node[nodelabel]{$b$};

\draw (2,3.2)node[nodelabel]{$\beta$};
\draw (2,-1.2)node[nodelabel]{$\alpha$};

\draw[dashed] (0,0) to [out=315,in=225] (4,0);
\draw[dashed] (0,2) to [out=45,in=135] (4,2);

\draw (0,0) -- (4,0)node{} -- (4,2)node[fill=black]{} -- (0,2)node[fill=black]{} -- (0,0)node{};
\draw (1,0)node[fill=black]{} -- (1,2)node{};
\draw (2,0)node{} -- (2,2)node[fill=black]{};
\draw (3,0)node[fill=black]{} -- (3,2)node{};

\end{tikzpicture}
\vspace*{-0.15in}
\caption{A prism has only one $R$-ladder configuration}
\label{fig:prism-Rconfiguration}
\end{figure}

\subsection{Corners, Rungs and Spokes}

We shall often need the flexibility of referring to a subgraph~$K$
which is either an \Rladder\ or an \Rbiwheel,
and in this case, we simply write that $K$ is an {\it \Rconf}.
Additionally, we may also state that
$K$ has \external\ $au$~and~$bw$ (possibly with subscript notation);
in this case,
we implicitly adopt the conventions stated in
Notation~\ref{Not:biwheel-ladder-convention},
and we refer to $a, u, b$ and $w$ as the {\it corners} of~$K$.
Furthermore, as shown in
Figures~\ref{fig:Rbiwheel-configuration} and \ref{fig:Rladder-configuration},
we will assume that $a, b \in V(R)$.
We refer to $u$ and $w$ as the
{\it free corners} of~$K$; these may lie in~$V(R)$
as in Figure~\ref{fig:prism-Rconfiguration}, or they may not lie in~$V(R)$
as in Figure~\ref{fig:pseudo-biwheel-Rconfiguration}.
Observe that any vertex of~$K$, which is not a corner, does not lie in~$V(R)$.

\smallskip
For any two distinct rungs/spokes of an \Rconf~$K$, say $e$ and $f$, we say that
$e$~and~$f$ are {\it consecutive}, or equivalently,
that $e$ is {\it consecutive with}~$f$, whenever an end of~$e$ which is not
a free corner is adjacent with an end of~$f$ which is also not a free corner.
Clearly, each internal rung (spoke)
is consecutive with two rungs (spokes); whereas each external rung (spoke)
is consecutive with only one rung (spoke) and the latter is internal.
Now, let $e$ denote an internal rung (spoke) of~$K$, and let $f$~and~$g$
denote the two rungs (spokes) with which $e$ is consecutive. By definition,
$e$ is an \Rthin\ edge of~$G$. Observe that $f$~and~$g$
are multiple edges in the retract of~$G-e$; consequently, $e$ is not strictly thin.

\subsection{Two distinct $R$-configurations}

\begin{figure}[!ht]
\centering
\smallskip
\begin{tikzpicture}[scale=1.1]

\draw (-0.4,0)node[nodelabel]{$a_1$};
\draw (-0.4,2)node[nodelabel]{$u_1$};
\draw (2,-0.3)node[nodelabel]{$w_1$};
\draw (2.2,2.22)node[nodelabel]{$b_1$};

\draw (6.5,3.1)node[nodelabel]{$w_2$};
\draw (6.5,-1.1)node[nodelabel]{$u_2$};
\draw (4,0.7)node[nodelabel]{$b_2$};
\draw (9,1.3)node[nodelabel]{$a_2$};

\draw (4.5,-1.9)node[above,nodelabel]{$\alpha$};
\draw (3.1,1.77)node[nodelabel]{$\beta$};
\draw (3.5,2.7)node[nodelabel]{$\alpha'$};
\draw (3.85,-0.6)node[nodelabel]{$\beta'$};

\draw[dashed] (0,0) to [out=270,in=180] (4.5,-1.8) to [out=0,in=270] (9,1);
\draw[dashed] (0,2) to [out=40,in=180] (6.5,2.8);

\draw[dashed] (2,2) -- (4,1);
\draw[dashed] (2,0) -- (6.5,-0.8);

\draw (0,0) -- (2,0)node{} -- (2,2)node[fill=black]{} -- (0,2)node[fill=black]{} -- (0,0)node{};
\draw (1,0)node[fill=black]{} -- (1,2)node{};

\draw (4,1) -- (9,1);

\draw (4,1)node[fill=black]{} -- (6.5,2.8);
\draw (5,1)node{} -- (6.5,-0.8);
\draw (6,1)node[fill=black]{} -- (6.5,2.8);
\draw (7,1)node{} -- (6.5,-0.8);
\draw (8,1)node[fill=black]{} -- (6.5,2.8)node{};
\draw (9,1)node{} -- (6.5,-0.8)node[fill=black]{};

\end{tikzpicture}
\caption{A laddered biwheel of type~II has two vertex-disjoint {\Rconf}s}
\label{fig:laddered-biwheel-typeII-Rconfigurations}
\smallskip
\end{figure}

A laddered biwheel of type~II,
as shown in Figure~\ref{fig:laddered-biwheel-typeII-Rconfigurations},
has two removable doubletons \mbox{$R:=\{\alpha,\beta\}$} and
\mbox{$R':=\{\alpha',\beta'\}$}.
Observe that the graph obtained by removing the edge set~$R \cup R'$
has two connected components, of which one is an \Rladder\ with
external rungs $a_1u_1$ and $b_1w_1$, and the other is an
\Rbiwheel\ with external spokes $a_2u_2$ and $b_2w_2$.
In this case, the two $R$-configurations are vertex-disjoint.

\smallskip
On the other hand, a double ladder of type~I, as shown
in Figure~\ref{fig:double-ladder-typeI-Rconfigurations},
has only one removable doubleton~$R:=\{\alpha,\beta\}$
and it has two {\Rladder}s which share their free corners
$u_1$~and~$w_1$, but are otherwise vertex-disjoint.
One of these is depicted by dashed lines, and it has
external rungs $a_1u_1$~and~$b_1w_1$, whereas the other one
has external rungs $a_2u_1$~and~$b_2w_1$.

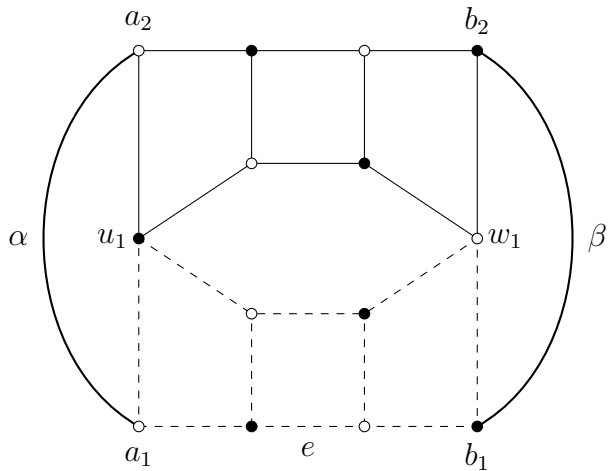
\begin{figure}[!ht]
\medskip
\centering
\begin{tikzpicture}[scale=1]

\draw (2.25,-0.3)node[nodelabel]{$e$};

\draw (0,5.4)node[nodelabel]{$a_2$};
\draw (4.5,5.4)node[nodelabel]{$b_2$};

\draw (0,-0.4)node[nodelabel]{$a_1$};
\draw (4.5,-0.4)node[nodelabel]{$b_1$};
\draw (-0.35,2.5)node[nodelabel]{$u_1$};
\draw (4.87,2.5)node[nodelabel]{$w_1$};

\draw (-1.6,2.5)node[nodelabel]{$\alpha$};
\draw (6.1,2.5)node[nodelabel]{$\beta$};

\draw[thick] (0,0) to [out=150,in=210] (0,5);
\draw[thick] (4.5,0) to [out=30,in=330] (4.5,5);

\draw (0,5)node{} -- (4.5,5)node[fill=black]{} -- (4.5,2.5) -- (3,3.5) -- (1.5,3.5) -- (0,2.5) -- (0,5);
\draw (1.5,5)node[fill=black]{} -- (1.5,3.5)node{};
\draw (3,5)node{} -- (3,3.5)node[fill=black]{};

\draw[dashed] (0,0) -- (4.5,0) -- (4.5,2.5) -- (3,1.5) -- (1.5,1.5) -- (0,2.5) -- (0,0);
\draw[dashed] (1.5,0) -- (1.5,1.5);
\draw[dashed] (3,0) -- (3,1.5);

\draw (0,0)node{};
\draw (4.5,0)node[fill=black]{};
\draw (4.5,2.5)node{};
\draw (0,2.5)node[fill=black]{};
\draw (1.5,0)node[fill=black]{};
\draw (1.5,1.5)node{};
\draw (3,0)node{};
\draw (3,1.5)node[fill=black]{};

\end{tikzpicture}
\vspace*{-0.1in}
\caption{A double ladder of type I has two {\Rconf}s which share their free corners
but are otherwise vertex-disjoint}
\label{fig:double-ladder-typeI-Rconfigurations}
\bigskip
\end{figure}

The reader is advised to check that members of all of the eleven families
that appear in Theorem~\ref{thm:strictly-Rthin-nb-bricks},
except for $K_4$ and $\overline{C_6}$,
have either one or two $R$-configurations for an appropriately
chosen removable doubleton~$R$.
(The choice of~$R$ matters only in the case of three families, namely,
pseudo-biwheels, double biwheels of Type~II and laddered biwheels of Type~II.
Figure~\ref{fig:pseudo-biwheel-Rconfiguration}
shows a pseudo-biwheel and its two removable doubletons.)

\smallskip
In order to sketch a proof of Theorem~\ref{thm:strictly-Rthin-nb-bricks},
we will require a few results which are stated next; their proofs will
appear in later sections.
In particular,
the following proposition states that two distinct {\Rconf}s
are either vertex-disjoint, or they have the same free corners but are otherwise
vertex-disjoint; its proof appears in
Section~\ref{sec:proof-Rconfigurations-almost-disjoint}.

\begin{prop}
{\sc [$R$-configurations are Almost Disjoint]}
\label{prop:Rconfigurations-almost-disjoint}
Let $G$ be a simple \Rbrick, and let $K_1$
denote an \Rconf\ with free corners $u_1$ and $w_1$.
If $K_2$ is any \Rconf\ distinct from~$K_1$,
then precisely one of the following statements holds:
\begin{enumerate}[(i)]
\item $K_1$ and $K_2$ are vertex-disjoint, or,
\item $u_1$ and $w_1$ are the free corners of~$K_2$,
and $K_2$ is otherwise vertex-disjoint with~$K_1$.
\end{enumerate}
\end{prop}

By the above proposition,
the only vertices that can be possibly shared between
two distinct {\Rconf}s are their respective free corners.
The remaining two corners of
each $R$-configuration lie in~$V(R)$.
Since $|V(R)|=4$, we immediately have the following consequence.

\begin{cor}
\label{cor:at-most-two-Rconfigurations}
A simple $R$-brick has at most two distinct $R$-configurations. \qed
\end{cor}

For instance, if $G$ is a Norine-Thomas brick or if it is a pseudo-biwheel
then it has only one \Rconf. On the other hand,
if $G$ is a double biwheel or a double ladder or a laddered biwheel,
then it has two {\Rconf}s, say $K_1$~and~$K_2$. Furthermore,
if $G$ is of type~II then $K_1$ and $K_2$ are vertex-disjoint
as in Proposition~\ref{prop:Rconfigurations-almost-disjoint}{\it (i)};
whereas, if~$G$ is of type~I then $K_1$ and $K_2$
have the same free corners but they do not have any other vertices in common
as in Proposition~\ref{prop:Rconfigurations-almost-disjoint}{\it (ii)}.

\subsection{The $R$-biwheel and $R$-ladder Theorems}

It is easily verified that
if~$G$ is any \Rbrick\ in~$\mathcal{N}$, then every \Rthin\ edge of~$G$
lies in an \Rconf. Here, we state two theorems which show that this is not
a coincidence.

\smallskip
Now, let $G$ be a simple $R$-brick which is free of strictly \Rthin\ edges.
Given any \Rthin\ edge~$e$ of~$G$,
we may invoke one of these theorems (depending on the index of~$e$) to find an
\Rconf~$K$ containing the edge~$e$. In particular, if the index of~$e$ is one,
we apply Theorem~\ref{thm:Rbiwheel-configuration} and in this case $K$ is an \Rbiwheel;
whereas, if the index of~$e$ is two, we apply Theorem~\ref{thm:Rladder-configuration}
and in this case $K$ is an \Rladder.

\begin{thm}
{\sc [$R$-biwheel Theorem]}
\label{thm:Rbiwheel-configuration}
Let $G$ be a simple \Rbrick\ which is free of strictly \Rthin\ edges,
and let $e$ denote an \Rthin\ edge whose index is one.
Then $G$ contains an \Rbiwheel, say~$K$,
such that $e$ is an internal spoke of~$K$.
\end{thm}

The proof of the above theorem appears in
Section~\ref{sec:proof-Rbiwheel-theorem},
and it is along the same lines as the proof of \cite[Theorem 4.6]{clm08}.

\smallskip
Given the statement of Theorem~\ref{thm:Rbiwheel-configuration}, one would expect that,
likewise, if $e$ is an $R$-thin edge whose index is two then $G$ contains an \Rladder, say~$K$,
such that $e$ is an internal rung of~$K$. Unfortunately, this is not true, in general.
Consider the double ladder of type~I, shown in Figure~\ref{fig:double-ladder-typeI-Rconfigurations}; $e$ is an $R$-thin edge of index two,
and although it is part of an \Rladder,
it is not a rung of that ladder.
We instead prove the following slightly weaker statement concerning
\Rthin\ edges of index two.

\begin{thm}
{\sc [$R$-ladder Theorem]}
\label{thm:Rladder-configuration}
Let $G$ be a simple \Rbrick\ which is free of strictly \Rthin\ edges,
and let $e$ denote an \Rthin\ edge whose index is two.
Then $G$ contains an \Rladder, say~$K$,
such that $e \in E(K)$.
\end{thm}

The proof of the above theorem appears in
Section~\ref{sec:proof-Rladder-theorem} and it
is significantly longer than that of
the \mbox{$R$-biwheel} Theorem~(\ref{thm:Rbiwheel-configuration}).
These two theorems (\ref{thm:Rbiwheel-configuration} and
\ref{thm:Rladder-configuration})
are central to our proof of
the Strictly $R$-thin Edge Theorem (\ref{thm:strictly-Rthin-nb-bricks}).

\subsection{Proof Sketch of Theorem~\ref{thm:strictly-Rthin-nb-bricks}}
\label{sec:proof-sketch}

As in the statement of the theorem, let $G$ be a simple \Rbrick\ which is free of
strictly \Rthin\ edges. Our goal is to show that $G$ is a member of one of the
eleven infinite families which appear in the statement of the theorem,
that is, to show that $G \in \mathcal{N}$.
We adopt Notation~\ref{Not:Rbrick-doubleton}.

\smallskip
We may assume that $G$ is different from $K_4$ and $\overline{C_6}$,
and thus, by the $R$-thin Edge Theorem
(\ref{thm:Rthin-nb-bricks}), $G$ has an $R$-thin edge, say~$e_1$.
Depending on the index of~$e_1$, we invoke
either the \mbox{$R$-biwheel} Theorem (\ref{thm:Rbiwheel-configuration})
or the \mbox{$R$-ladder} Theorem (\ref{thm:Rladder-configuration}) to
deduce that $G$ has an \Rconf, say~$K_1$, such that
$e_1 \in E(K_1)$. We shall let $a_1u_1$ and $b_1w_1$ denote
the \external\ of~$K_1$, and adjust notation so that $u_1$ and $w_1$ are
its free corners.

\smallskip
We will show that either $u_1$ and $w_1$ both lie in~$V(R)$, or otherwise neither
of them lies in~$V(R)$. In the former case, we will conclude that $G$ is either
a prism or a M{\"o}bius ladder or a truncated biwheel, and we are done.

\smallskip
Now suppose that $u_1, w_1 \notin V(R)$. In this case, we will show that
either $G$ is a staircase or a pseudo-biwheel, and we are done;
or otherwise, $G$ has an $R$-compatible
edge which is not in~$E(K_1)$. In the latter case, we will apply
Theorem~\ref{thm:rank-plus-index} to deduce that $G$ has an $R$-thin edge, say~$e_2$,
which is not in~$E(K_1)$. Depending on the index of~$e_2$, we may once again
use either the \mbox{$R$-biwheel} Theorem (\ref{thm:Rbiwheel-configuration})
or the \mbox{$R$-ladder} Theorem (\ref{thm:Rladder-configuration}) to
conclude that $G$ has an \Rconf, say~$K_2$, such that $e_2 \in E(K_2)$.

\smallskip
By Proposition~\ref{prop:Rconfigurations-almost-disjoint}, either $K_1$ and $K_2$ are vertex-disjoint, or otherwise $K_2$
has the same free corners as~$K_1$ but is otherwise vertex-disjoint with~$K_1$.
In the latter case, we will conclude that $G$ is either a double
biwheel or a double ladder or a laddered biwheel, each of type~I, and we are done.

\smallskip
Now suppose that $K_1$ and $K_2$ are vertex-disjoint.
We will argue that either $G$ is a double biwheel or a double ladder or a laddered
biwheel, each of type~II, and we are done; or otherwise, $G$ has an
$R$-compatible edge which is not in~$E(K_1 \cup K_2)$. In the latter case,
we will once again apply Theorem~\ref{thm:rank-plus-index} to conclude that
$G$ has an $R$-thin edge, say~$e_3$, which is not in~$E(K_1 \cup K_2)$. As
usual, depending on the index of~$e_3$, we invoke either the
\mbox{$R$-biwheel} Theorem (\ref{thm:Rbiwheel-configuration}) or the
\mbox{$R$-ladder} Theorem (\ref{thm:Rladder-configuration}) to deduce
that $G$ has an \Rconf, say~$K_3$, such that $e_3 \in E(K_3)$.

\smallskip
We have thus located three distinct {\Rconf}s in the brick~$G$, namely,
$K_1, K_2$ and $K_3$. However, this contradicts
Corollary~\ref{cor:at-most-two-Rconfigurations}, and completes
the proof sketch of the Strictly \Rthin\ Edge Theorem~(\ref{thm:strictly-Rthin-nb-bricks}).

\section{$R$-thin edges}
\label{sec:R-thin}

Here, we will prove the
\mbox{$R$-biwheel} Theorem (\ref{thm:Rbiwheel-configuration}) and
the \mbox{$R$-ladder} Theorem (\ref{thm:Rladder-configuration}).
Our proofs are inspired by the work of Carvalho et al. \cite{clm08}.
In the next section, we will review conditions under which an $R$-thin edge
is not strictly thin, and we will state a few key lemmas
(\ref{lem:index-one-non-removable}, \ref{lem:index-two-non-removable}
and \ref{lem:removable-not-thin})
from \cite{clm08}
which are used in our proofs.
Before that, we state a few preliminary facts; their proofs
may be found in \cite[Chapter 4]{koth16}.

\smallskip
The removable edges of a bipartite graph satisfy the following
`exchange property'.
\begin{prop}
\label{prop:exchange-property-removable-bipmcg}
Let $H$ denote a bipartite matching covered graph, and
let $e$ denote a \mbox{removable} edge of~$H$. If $f$
is a removable edge of~$H-e$, then:
\begin{enumerate}[(i)]
\item $f$ is removable in~$H$, and
\item $e$ is removable in~$H-f$. \qed
\end{enumerate}
\end{prop}

A matching covered subgraph $K$ of a matching covered graph $H$ is {\it conformal}
if the graph $H-V(K)$ has a perfect matching; equivalently; $K$ is conformal if each perfect
matching of $K$ extends to a perfect matching of~$H$.
The following is a generalization of Proposition~\ref{prop:exchange-property-removable-bipmcg}
that is easily proved using the theory of ear decompositions (see \cite{koth16}).
\begin{prop}
\label{prop:conformal-exchange-property-removable-bipmcg}
Let $K$ be a conformal matching covered subgraph of a bipartite matching covered graph~$H$.
Let $e$ denote a removable edge of~$K$.
Then $e$ is removable in~$H$ as well. \qed
\end{prop}

The following is a useful characterization of non-removable edges
in bipartite graphs.
\begin{prop}
\label{prop:characterization-non-removable-bipartite}
Let $H[A,B]$ denote a bipartite matching covered graph
on four or more vertices.
An edge~$e$ of~$H$ is non-removable
if and only if there exist partitions $(A_0,A_1)$ of~$A$
and $(B_0,B_1)$ of~$B$ such that $|A_0| = |B_0|$
and $e$ is the only edge joining
a vertex in~$B_0$ to a vertex in~$A_1$. \qed
\end{prop}

This fact yields the following corollary.
\begin{cor}
\label{cor:quadrilateral-admissible-removable}
Suppose that $Q$ is a $4$-cycle of a bipartite matching covered graph~$H$,
and let $e$ and $f$ denote two nonadjacent edges of~$Q$.
If $f$ is admissible in~$H-e$ then $e$ is removable in~$H$. \qed
\end{cor}

\subsection{Multiple Edges in Retracts}
\label{sec:multiple-edges-retracts}

Throughout this section, $G$ is a simple \Rbrick,
and we adopt Notation~\ref{Not:Rbrick-doubleton}.
Furthermore, we shall let
$e$ denote an \Rthin\ edge which is not strictly thin,
and $J$ the retract of~$G-e$.
Since $e$ is not strictly thin,
$J$ is not simple, and we shall let $f$ and $g$ denote two multiple (parallel)
edges of~$J$. It should be noted that since $J$ is also an \Rbrick, neither
edge of~$R$ is a multiple edge of~$J$.
In particular, $f$ and $g$ do not lie in~$R$.

\smallskip
We denote the ends of~$e$ by letters $y$~and~$z$ with subscripts~$1$;
that is, $e:=y_1z_1$.
Adjust notation so that $y_1 \in A$ and $z_1 \in B$.
If either end of~$e$ is cubic, then we denote
its two neighbours in~$G-e$ by subscripts $0$~and~$2$. For example,
if $y_1$ is cubic then $N(y_1) = \{z_1,y_0,y_2\}$.

\smallskip
As $G$ is simple, it follows that $J$ has a contraction vertex which is incident
with both $f$~and~$g$. We infer that one end of~$e$, say~$y_1$,
is cubic, and that
$f$ is incident with~$y_0$, and $g$ is incident with~$y_2$.
See Figure~\ref{fig:Rthin-not-strictly-thin}.
As noted earlier, $f \notin R$; consequently, $e$ and $f$
are nonadjacent. Likewise, $e$ and $g$ are nonadjacent.

\begin{figure}[!ht]
\medskip
\centering
\begin{tikzpicture}

\draw[dashed] (-1.5,0) -- (-1.5,2);
\draw[dashed] (1.5,0) -- (1.5,2);

\draw (-1.75,1)node[nodelabel]{$f$};
\draw (1.75,1)node[nodelabel]{$g$};

\draw (0,0) -- (0,2);
\draw (0,0) -- (-1.5,0);
\draw (0,0) -- (1.5,0);

\draw (0,0) node{}node[nodelabel,below]{$y_1$};
\draw (0,2) node[fill=black]{}node[nodelabel,above]{$z_1$};
\draw (-1.5,0) node[fill=black]{}node[nodelabel,below]{$y_0$};
\draw (1.5,0) node[fill=black]{}node[nodelabel,below]{$y_2$};
\draw (0.25,1)node[nodelabel]{$e$};

\end{tikzpicture}
\caption{$f$ and $g$ are multiple edges in the retract~$J$ of~$G-e$;
the vertex~$y_1$ is cubic}
\label{fig:Rthin-not-strictly-thin}
\bigskip
\end{figure}
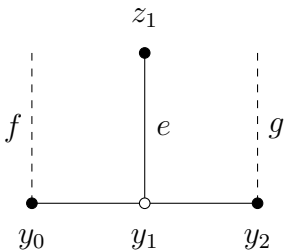

We will consider two separate cases depending on whether the
edges $f$~and~$g$ are adjacent (in~$G$) or not.
In the case in which they are adjacent,
we shall denote their common end by~$w$, as shown
in Figure~\ref{fig:f-and-g-adjacent-or-not}a.
Now suppose that $f$ and $g$ are nonadjacent. Since
they are multiple (parallel) edges of~$J$, we infer that both
ends of~$e$ are cubic, and that $f$ and $g$ join
the two contraction vertices of~$J$.
This proves the following
proposition; see Figure~\ref{fig:f-and-g-adjacent-or-not}b.

\begin{prop}
\label{prop:f-g-nonadjacent}
Suppose that $f$ and $g$ are nonadjacent in~$G$.
Then the following hold:
\begin{enumerate}[(i)]
\item each end of~$e$ is cubic,
\item consequently, the index of~$e$ is two, and
\item one of $f$ and $g$ is incident with $z_0$ whereas the other one
is incident with~$z_2$. \qed
\end{enumerate}
\end{prop}

In view of statement {\it (iii)},
whenever $f$ and $g$ are nonadjacent,
we shall assume without loss of generality that $f:=y_0z_0$ and $g:=y_2z_2$,
as shown in Figure~\ref{fig:f-and-g-adjacent-or-not}b.

\begin{figure}[!ht]
\centering
\begin{tikzpicture}
\draw (0,-2) -- (-0.35,-2.35);
\draw (0,-2) -- (0.35,-2.35);

\draw (-1,-1.2)node[nodelabel]{$f$};
\draw (1,-1.2)node[nodelabel]{$g$};

\draw (-1.5,0) -- (0,-2);
\draw (1.5,0) -- (0,-2);
\draw (0,-2)node{};
\draw (0,-2.35)node[nodelabel]{$w$};

\draw (0,0) -- (0,2);
\draw (0,0) -- (-1.5,0);
\draw (0,0) -- (1.5,0);

\draw (0,0) node{};
\draw (0,-0.35)node[nodelabel]{$y_1$};
\draw (0,2) node[fill=black]{}node[nodelabel,above]{$z_1$};
\draw (-1.5,0) node[fill=black]{}node[nodelabel,left]{$y_0$};
\draw (1.5,0) node[fill=black]{}node[nodelabel,right]{$y_2$};
\draw (0.25,1)node[nodelabel]{$e$};

\draw (0,-3.3)node[nodelabel]{(a)};

\end{tikzpicture}
\hspace*{1in}
\begin{tikzpicture}

\draw (-1.5,0) -- (-1.5,2);
\draw (1.5,0) -- (1.5,2);

\draw (-1.5,2)node{}node[nodelabel,above]{$z_0$} --
(1.5,2)node{}node[nodelabel,above]{$z_2$};

\draw (-1.75,1)node[nodelabel]{$f$};
\draw (1.75,1)node[nodelabel]{$g$};

\draw (0,0) -- (0,2);
\draw (0,0) -- (-1.5,0);
\draw (0,0) -- (1.5,0);

\draw (0,0) node{}node[nodelabel,below]{$y_1$};
\draw (0,2) node[fill=black]{}node[nodelabel,above]{$z_1$};
\draw (-1.5,0) node[fill=black]{}node[nodelabel,below]{$y_0$};
\draw (1.5,0) node[fill=black]{}node[nodelabel,below]{$y_2$};
\draw (0.25,1)node[nodelabel]{$e$};

\draw (0,-1.5)node[nodelabel]{(b)};

\end{tikzpicture}
\caption{(a) when $f$ and $g$ are adjacent; (b) when $f$ and $g$ are nonadjacent}
\label{fig:f-and-g-adjacent-or-not}
\bigskip
\end{figure}
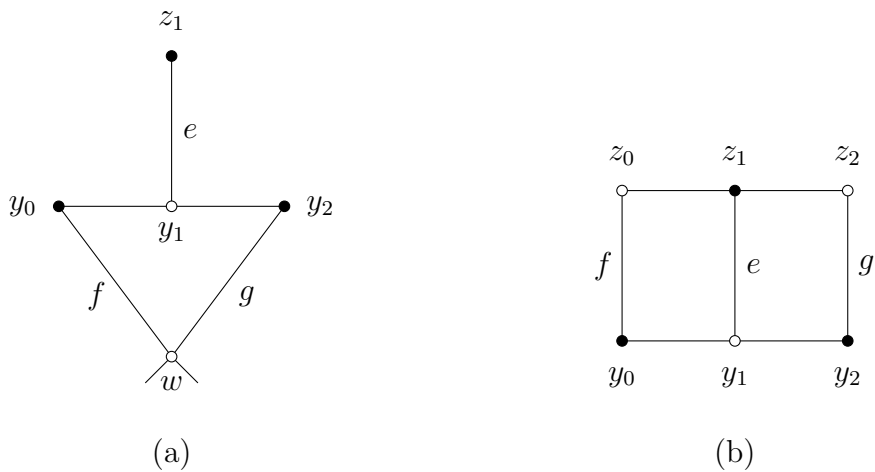

Let us now focus on the case in which $f$ and $g$ are adjacent,
as shown in Figure~\ref{fig:f-and-g-adjacent-or-not}a.
We remark that, in this case, the index of~$e$ is not determined; that is,
its index could be either one or two depending on the degree of
its end~$z_1$.
Instead, we are able to say something about the degree of~$w$.

\begin{prop}
\label{prop:f-g-adjacent}
Suppose that $f$ and $g$ are adjacent in~$G$, and let $w$ be their common end.
Then $w$ has degree four or more.
\end{prop}
\begin{proof}
First suppose that $w$ is not a neighbour of~$z_1$. In this case,
$w$ is not affected by the bicontractions in~$G-e$. Consequently, $w$ is a vertex of
the brick~$J$, whence it has at least three distinct neighbours.
Since $f$ and $g$ are
multiple edges, $w$ has degree four or more.

\smallskip
Now suppose that $w$ is a neighbour of~$z_1$.
Observe that the neighbours of~$y_1$ are precisely $y_0, y_2$ and $z_1$;
each of which is adjacent with~$w$.
See Figure~\ref{fig:f-and-g-adjacent-or-not}a.
Note that, if $w$ is cubic,
then its neighbourhood is the same as that of~$y_1$; and in this case,
$\{y_0,y_2,z_1\}$ is a barrier of the brick~$G$; this is absurd.
Thus $w$ has degree four or more.
\end{proof}

Note that $f$ and $g$, being multiple edges of~$J$, are both \Rthin\ in~$J$.
We shall now examine conditions under which one of them, say~$f$,
fails to be \Rthin\ in~$G$.
This may be the case for three different reasons; firstly, $f$ is non-removable in the bipartite graph~$H=G-R$; secondly, $f$ is non-removable in~$G$;
and thirdly, $f$ is removable in~$G$ but it is not thin.

\smallskip
We begin with the situation in which $f$ is non-removable in~$H$.
Note that, if an end of~$f$ is cubic (in~$G$) and if it also lies in~$V(R)$,
then it has degree two in~$H$, rendering $f$ non-removable.
We will now argue that the converse also holds.

\begin{lem}\label{lem:non-removable-in-H}
The edge $f$ is non-removable in~$H$ if and only if
it has a cubic end which lies in~$V(R)$.
\end{lem}
\begin{proof}
Suppose that $f$ has no cubic end which lies in~$V(R)$.
Consequently, each end of~$f$ has degree two or more in~$H-f$.
Furthermore, since $e$ and $f$ are nonadjacent,
each end of~$f$ has degree two or more in~$H-e-f$ as well.
We will argue $H-e-f$ is matching covered, that is, $f$ is removable in~$H-e$.
The exchange property (Proposition~\ref{prop:exchange-property-removable-bipmcg})
then implies that $f$ is also removable in~$H$.

\smallskip
Note that $f$ is a multiple edge of~$J-R$, whence $J-R-f$ is matching
covered. Note that any graph obtained from a matching covered graph by
means of bi-splitting a vertex is also matching covered;
see \cite[Section~1.5.2]{koth16}.
We will argue that~$H-e-f$ may be obtained from~$J-R-f$ by means of
bi-splitting one or two vertices.

\smallskip
Note that $J$ is obtained from $G-e$ by means of bicontracting one or two vertices (of degree two);
likewise, $J-R$ may be obtained from $H-e$ by means of bicontractions.
Conversely, $H-e$ may be obtained from~$J-R$
by means of bi-splitting one or two vertices;
these are the contraction vertices of~$J$.
As noted earlier, since each end of~$f$ has degree two or more in~$H-e-f$,
we may similarly
obtain~$H-e-f$ from~$J-R-f$ by means of bi-splitting the same vertices.
As discussed above, $H-e-f$ is matching covered; consequently,
$f$ is removable in~$H$.
\end{proof}

We now turn to the situation in which $f$ is non-removable in~$G$.
For convenience, we will state two lemmas
(\ref{lem:index-one-non-removable} and \ref{lem:index-two-non-removable}),
depending on the index of~$e$. These appear in the work
of Carvalho et al. \cite[Lemma~4.2]{clm08} as a single lemma. (In their work,
they deal with the more general context in which $e$ is a thin edge of a brick~$G$,
which need not be near-bipartite.)

\smallskip
The first lemma (\ref{lem:index-one-non-removable})
considers the scenario in which the index of~$e$ is one.
By Proposition~\ref{prop:f-g-nonadjacent}{\it (ii)}, $f$ and $g$ are adjacent;
and by Proposition~\ref{prop:f-g-adjacent}, their
common end~$w$ is non-cubic.

\begin{lem}
{\rm \cite{clm08}}
\label{lem:index-one-non-removable}
Suppose that the index of~$e$ is one.
If $f$ is non-removable in~$G$ then~$f$ has a cubic end which
is adjacent with both ends of~$e$. {\rm (}In particular, the cubic end of~$f$
lies in~$V(R)$.{\rm )} \qed
\end{lem}

As $w$ is non-cubic,
$y_0$ is the cubic end of~$f$,
and it is adjacent with~$z_1$,
as shown in Figure~\ref{fig:index-one-non-removable}a.
Clearly, the edge joining $y_0$ and $z_1$ is none other than $\beta$.

\begin{figure}[!ht]
\smallskip
\centering
\begin{tikzpicture}
\draw (0,-2) -- (-0.35,-2.35);
\draw (0,-2) -- (0.35,-2.35);

\draw (-1.5,0) -- (0,2);
\draw (-1,1)node[nodelabel]{$\beta$};

\draw (-1,-1.2)node[nodelabel]{$f$};
\draw (1,-1.2)node[nodelabel]{$g$};

\draw (-1.5,0) -- (0,-2);
\draw (1.5,0) -- (0,-2);
\draw (0,-2)node{};
\draw (0,-2.35)node[nodelabel]{$w$};

\draw (0,0) -- (0,2);
\draw (0,0) -- (-1.5,0);
\draw (0,0) -- (1.5,0);

\draw (0,0) node{};
\draw (0,-0.35)node[nodelabel]{$y_1$};
\draw (0,2) node[fill=black]{}node[nodelabel,above]{$z_1$};
\draw (-1.5,0) node[fill=black]{}node[nodelabel,left]{$y_0$};
\draw (1.5,0) node[fill=black]{}node[nodelabel,right]{$y_2$};
\draw (0.25,1)node[nodelabel]{$e$};

\draw (0,-3.3)node[nodelabel]{(a)};
\end{tikzpicture}
\hspace*{0.3in}
\begin{tikzpicture}[scale=1]

\draw (7.1,1.9)node[nodelabel]{$e$};
\draw (7.9,-0.2)node[nodelabel]{$f$};
\draw (6.6,0)node[nodelabel]{$g$};

\draw (4,1) -- (6.5,-1);
\draw (9,1) -- (6.5,3);
\draw (5,-0.2)node[nodelabel]{$\alpha$};
\draw (8,2.2)node[nodelabel]{$\beta$};

\draw (4,1) -- (9,1);

\draw (4,1)node{} -- (6.5,3);
\draw (5,1)node[fill=black]{} -- (6.5,-1);
\draw (6,1)node{} -- (6.5,3);
\draw (7,1)node[fill=black]{} -- (6.5,-1);
\draw (8,1)node{} -- (6.5,3)node[fill=black]{};
\draw (9,1)node[fill=black]{} -- (6.5,-1)node{};

\draw (6.5,-2.3)node[nodelabel]{(b)};
\end{tikzpicture}
\vspace*{-0.1in}
\caption{Illustration for Lemma~\ref{lem:index-one-non-removable}}
\label{fig:index-one-non-removable}
\bigskip
\end{figure}

The situation in Lemma~\ref{lem:index-one-non-removable} arises
in truncated biwheels,
as shown in Figure~\ref{fig:index-one-non-removable}.
Note that, every perfect matching which contains $e$ also contains~$f$,
rendering $f$ non-removable.

\smallskip
The second lemma (\ref{lem:index-two-non-removable}) deals with the
scenario in which the index of~$e$ is two, that is, each end of~$e$ is cubic.

\begin{lem}
\label{lem:index-two-non-removable}
{\rm \cite{clm08}}
Suppose that the index of~$e$ is two.
If $f$ is non-removable in~$G$ then the following hold:
\begin{enumerate}[(i)]
\item each end of~$f$ is cubic,
\item consequently, $f$ and $g$ are nonadjacent, and
\item the ends of~$f$ have a common neighbour.
\end{enumerate}
{\rm (}In particular, one of the ends of~$f$ is cubic and it also lies in~$V(R)$.{\rm )}
\qed
\end{lem}

By statement {\it (i)}, each end of~$f$ is cubic;
thus $f$ and $g$ are nonadjacent (see Proposition~\ref{prop:f-g-adjacent}).
By Proposition~\ref{prop:f-g-nonadjacent},
and as per our notation,
$f=y_0z_0$ and $g=y_2z_2$, as shown in
Figure~\ref{fig:index-two-non-removable}a.
By statement {\it (iii)}, $y_0$~and~$z_0$ have a common neighbour, say~$x$.
Clearly, one of $xy_0$ and $xz_0$ is an edge
of~$R$, depending on whether $x$ lies in~$A$ or in~$B$; however,
these cases are symmetric. Adjust notation so that $x \in B$; thus $xy_0$
is the edge~$\beta$.
Using the fact that $G$ is free of nontrivial barriers, it is easily
verified that $x$ is not an end of~$g$.

\begin{figure}[!ht]
\bigskip
\centering
\begin{tikzpicture}
\draw (-3.3,1)node[nodelabel]{$x$};

\draw (-2.4,0.3)node[nodelabel]{$\beta$};

\draw (-1.5,0) -- (-3,1) -- (-1.5,2);
\draw (-3,1) node[fill=black]{};

\draw (-1.5,0) -- (-1.5,2);
\draw (1.5,0) -- (1.5,2);

\draw (-1.5,2)node{}node[nodelabel,above]{$z_0$} --
(1.5,2)node{}node[nodelabel,above]{$z_2$};

\draw (-1.75,1)node[nodelabel]{$f$};
\draw (1.75,1)node[nodelabel]{$g$};

\draw (0,0) -- (0,2);
\draw (0,0) -- (-1.5,0);
\draw (0,0) -- (1.5,0);

\draw (0,0) node{}node[nodelabel,below]{$y_1$};
\draw (0,2) node[fill=black]{}node[nodelabel,above]{$z_1$};
\draw (-1.5,0) node[fill=black]{}node[nodelabel,below]{$y_0$};
\draw (1.5,0) node[fill=black]{}node[nodelabel,below]{$y_2$};
\draw (0.25,1)node[nodelabel]{$e$};

\draw (0,-1.5)node[nodelabel]{(a)};

\end{tikzpicture}
\hspace*{0.5in}
\begin{tikzpicture}

\draw (-2.4,0.3)node[nodelabel]{$\beta$};
\draw (2.3,1.65)node[nodelabel]{$\alpha$};

\draw (-3,1) to [out=80,in=180] (0,3.5) to [out=0,in=100] (3,1);
\draw (1.5,0) -- (3,1) -- (1.5,2);
\draw (3,1)node{};

\draw (-1.5,0) -- (-3,1) -- (-1.5,2);
\draw (-3,1) node[fill=black]{};

\draw (-1.5,0) -- (-1.5,2);
\draw (1.5,0) -- (1.5,2);

\draw (-1.5,2)node{} -- (1.5,2)node{};

\draw (-1.75,1)node[nodelabel]{$f$};
\draw (1.75,1)node[nodelabel]{$g$};

\draw (0,0) -- (0,2);
\draw (0,0) -- (-1.5,0);
\draw (0,0) -- (1.5,0);

\draw (0,0) node{};
\draw (0,2) node[fill=black]{};
\draw (-1.5,0) node[fill=black]{};
\draw (1.5,0) node[fill=black]{};
\draw (0.25,1)node[nodelabel]{$e$};

\draw (0,-1.5)node[nodelabel]{(b)};

\end{tikzpicture}
\vspace*{-0.1in}
\caption{Illustration for Lemma~\ref{lem:index-two-non-removable}}
\label{fig:index-two-non-removable}
\bigskip
\end{figure}

The situation in Lemma~\ref{lem:index-two-non-removable} is observed
in staircases, as shown in Figure~\ref{fig:index-two-non-removable}b.
The edge~$f$ is non-removable since every perfect matching which contains~$e$
also contains~$f$.

\smallskip
Finally, we turn to the case in which $f$ is removable in~$G$ but it is not thin.
This is handled by Lemma~\ref{lem:removable-not-thin}
which appears in the work of Carvalho et al. \cite[Lemma~4.3]{clm08}.

\begin{lem}
\label{lem:removable-not-thin}
{\rm \cite{clm08}}
If $f$ is removable in~$G$ but it is not thin then the following hold:
\begin{enumerate}[(i)]
\item the index of~$e$ is two,
\item $f$ and $g$ are adjacent and their common end $w$ is not adjacent with
any end of~$e$,
\item $g$ is a thin edge, and
\item $N(y_0) \subseteq N(z_1) \cup \{w\}$;
recall that $y_0$ is the other end of~$f$,
and $z_1$ is the end of~$e$ not adjacent with~$y_0$. \qed
\end{enumerate}
\end{lem}

The lemma concludes that the index of~$e$ is two; that is, its end~$z_1$ is cubic,
and as per our notation,
the neighbours of~$z_1$ are precisely $y_1, z_0$ and $z_2$.
Furthermore, it concludes that $f$ and $g$ are adjacent and that their
common end~$w$ is distinct from each of $z_0$~and~$z_2$,
as shown in Figure~\ref{fig:removable-not-thin}a.
Another consequence which may be inferred from their proof
is that all of the neighbours of~$y_0$ lie in the
set~$N(z_1) \cup \{w\} = \{w,y_1,z_0,z_2\}$. (This is not stated explicitly
in the statement of \cite[Lemma 4.3]{clm08}.)
Since $y_0$ has degree at least three, we may adjust notation so that
$y_0$ is adjacent with~$z_0$, and it may or may not be adjacent with~$z_2$.

\begin{figure}[!ht]
\medskip
\centering
\begin{tikzpicture}
\draw (0,-2) -- (-0.35,-2.35);
\draw (0,-2) -- (0.35,-2.35);

\draw (-1.5,0) -- (-1.5,2);

\draw (-1.5,2) -- (1.5,2);
\draw (-1.5,2) node{}node[nodelabel,above]{$z_0$};
\draw (1.5,2) node{}node[nodelabel,above]{$z_2$};

\draw (-1,-1.2)node[nodelabel]{$f$};
\draw (1,-1.2)node[nodelabel]{$g$};

\draw (-1.5,0) -- (0,-2);
\draw (1.5,0) -- (0,-2);
\draw (0,-2)node{};
\draw (0,-2.35)node[nodelabel]{$w$};

\draw (0,0) -- (0,2);
\draw (0,0) -- (-1.5,0);
\draw (0,0) -- (1.5,0);

\draw (0,0) node{};
\draw (0,-0.35)node[nodelabel]{$y_1$};
\draw (0,2) node[fill=black]{}node[nodelabel,above]{$z_1$};
\draw (-1.5,0) node[fill=black]{}node[nodelabel,left]{$y_0$};
\draw (1.5,0) node[fill=black]{}node[nodelabel,right]{$y_2$};
\draw (0.25,1)node[nodelabel]{$e$};

\draw (0,-3.3)node[nodelabel]{(a)};

\end{tikzpicture}
\hspace*{0.3in}
\begin{tikzpicture}[scale=1]

\draw (3.75,1.7)node[nodelabel]{$f$};
\draw (4.75,1.25)node[nodelabel]{$g$};

\draw (2.25,0.25)node[nodelabel]{$e$};

\draw (0,-0.4)node[nodelabel]{$z_2$};
\draw (4.5,-0.4)node[nodelabel]{$y_2$};
\draw (4.87,2.5)node[nodelabel]{$w$};

\draw (1.5,-0.4)node[nodelabel]{$z_1$};
\draw (3,-0.4)node[nodelabel]{$y_1$};
\draw (2.95,1.85)node[nodelabel]{$y_0$};
\draw (1.55,1.85)node[nodelabel]{$z_0$};

\draw (-1.6,2.5)node[nodelabel]{$\alpha$};
\draw (6.1,2.5)node[nodelabel]{$\beta$};

\draw[thick] (0,0) to [out=150,in=210] (0,5);
\draw[thick] (4.5,0) to [out=30,in=330] (4.5,5);

\draw (0,5)node{} -- (4.5,5)node[fill=black]{} -- (4.5,2.5) -- (3,3.5) -- (1.5,3.5) -- (0,2.5) -- (0,5);
\draw (1.5,5)node[fill=black]{} -- (1.5,3.5)node{};
\draw (3,5)node{} -- (3,3.5)node[fill=black]{};

\draw (0,0) -- (4.5,0) -- (4.5,2.5) -- (3,1.5) -- (1.5,1.5) -- (0,2.5) -- (0,0);
\draw (1.5,0) -- (1.5,1.5);
\draw (3,0) -- (3,1.5);

\draw (0,0)node{};
\draw (4.5,0)node[fill=black]{};
\draw (4.5,2.5)node{};
\draw (0,2.5)node[fill=black]{};
\draw (1.5,0)node[fill=black]{};
\draw (1.5,1.5)node{};
\draw (3,0)node{};
\draw (3,1.5)node[fill=black]{};

\draw (2.25,-1.2)node[nodelabel]{(b)};

\end{tikzpicture}
\caption{Illustration for Lemma~\ref{lem:removable-not-thin}}
\label{fig:removable-not-thin}
\end{figure}

The situation in Lemma~\ref{lem:removable-not-thin} is best illustrated
by a double ladder of type~I in which at least one of the two {\Rladder}s
is of order eight, as shown in Figure~\ref{fig:removable-not-thin}b.
The edge $e$ is $R$-thin; deleting it and taking the
retract yields the staircase~$St_{10}$ with multiple edges, two of which
are $f$ and $g$. It may be verified that both $f$ and $g$ are removable,
but of them only $g$ is thin.

\subsection{Proof of the $R$-biwheel Theorem}
\label{sec:proof-Rbiwheel-theorem}

In this section, we prove the $R$-biwheel
Theorem (\ref{thm:Rbiwheel-configuration});
our proof is along the same lines as that of \cite[Theorem~4.6]{clm08}.
Before that, we need one more lemma pertaining to the structure
of \Rthin\ edges of index one (in an \Rbrick\ which is free of \sRthin\ edges).

\begin{lem}
\label{lem:index-one-Rthin-edge}
Let $G$ be a simple $R$-brick which is free of strictly $R$-thin edges,
$e$ an \Rthin\ edge whose index is one, and $y_1$ the cubic end of~$e$.
Let $y_0$ and $y_2$ denote the neighbours of~$y_1$ in~$G-e$.
Then $y_0$ and $y_2$ are both cubic, and they have a common neighbour~$w$
which is non-cubic. Let $f:=wy_0$ and $g:=wy_2$. Furthermore,
the following statements hold:
\begin{enumerate}[(i)]
\item if $f$ is not $R$-compatible then $y_0 \in V(R)$, and
\item if $f$ is $R$-compatible then it is $R$-thin and its index is one.
\end{enumerate}
(Similar statements also apply to~$g$.)
\end{lem}
\begin{proof}
Let $J$ denote the retract of~$G-e$, that is, $J$ is obtained from~$G-e$
by bicontracting the vertex~$y_1$. By hypothesis, $e$ is not strictly thin, whence
$J$ has multiple edges. This implies that $G$ has a vertex~$w$,
distinct from~$y_1$, that is adjacent to both $y_0$ and $y_2$,
as shown in Figure~\ref{fig:f-and-g-adjacent-or-not}a.
As in the statement of the lemma, let $f:=wy_0$ and $g:=wy_2$.
By Proposition~\ref{prop:f-g-adjacent}, $w$ has degree four or more.

\smallskip
First consider the case in which $f$ is not \Rcomp. That is, either $f$
is not removable in~$H$ or it is not removable in~$G$, and it follows from
Lemma~\ref{lem:non-removable-in-H} or from
Lemma~\ref{lem:index-one-non-removable},
respectively, that the end~$y_0$ of $f$ is cubic and it lies in~$V(R)$.

\smallskip
Now consider the case in which $f$ is \Rcomp.
Since the index of~$e$ is one,
Lemma~\ref{lem:removable-not-thin} implies that $f$ is thin, whence it is \Rthin.
By hypothesis, $f$ is not strictly \Rthin. Consequently, the end~$y_0$ of $f$ is cubic,
and the index of~$f$ is one. Applying a similar argument to the edge~$g$, we may
conclude that $y_2$ is also cubic.
\end{proof}

\begin{proofOf}{the $R$-biwheel
Theorem~(\ref{thm:Rbiwheel-configuration})}
As in the statement of the theorem, let $G$ be a simple $R$-brick which
is free of strictly \Rthin\ edges, and let $e$ denote an \Rthin\ edge
whose index is one.
Our goal is to show that $G$ has an \Rbiwheel\ of which $e$ is an internal
spoke.

\smallskip
As in the statement of Lemma~\ref{lem:index-one-Rthin-edge}, we let
$y_1$ denote the cubic end of~$e$, and $y_0$~and~$y_2$ the neighbours
of~$y_1$ in~$G-e$. By the lemma,
$y_0$~and~$y_2$ are both cubic, and they have a common neighbour~$w$
which is non-cubic. We denote by~$u$ the non-cubic end of~$e$,
as shown in Figure~\ref{fig:index-one-Rthin-edge}.
Observe that $y_0y_1y_2$ is a path in~$H-\{u,w\}$.

\begin{figure}[!ht]
\centering
\begin{tikzpicture}
\draw (0,-2) -- (-0.35,-2.35);
\draw (0,-2) -- (0.35,-2.35);

\draw (1.5,0) -- (2,0);
\draw (-1.5,0) -- (-2,0);

\draw (-1,-1.2)node[nodelabel]{$f$};
\draw (1,-1.2)node[nodelabel]{$g$};

\draw (-1.5,0) -- (0,-2);
\draw (1.5,0) -- (0,-2);
\draw (0,-2)node{};
\draw (0,-2.35)node[nodelabel]{$w$};

\draw (0,0) -- (0,2);
\draw (0,0) -- (-1.5,0);
\draw (0,0) -- (1.5,0);

\draw (0,0) node{};
\draw (0,-0.35)node[nodelabel]{$y_1$};
\draw (0,2) node[fill=black]{}node[nodelabel,above]{$u$};
\draw (-1.5,0) node[fill=black]{};
\draw (-1.5,0.35) node[nodelabel]{$y_0$};
\draw (1.5,0) node[fill=black]{};
\draw (1.5,0.35) node[nodelabel]{$y_2$};
\draw (0.25,1)node[nodelabel]{$e$};

\end{tikzpicture}
\caption{$e$ is an $R$-thin edge of index one; $y_0, y_1$ and $y_2$ are cubic;
$u$ and $w$ are non-cubic}
\label{fig:index-one-Rthin-edge}
\bigskip
\end{figure}

We let $P:=v_1v_2 \dots v_j$, where $j \geq 3$, be a path of maximum
length in the graph \mbox{$H-\{u,w\}$} that has the following properties
(see Figure~\ref{fig:illustration-for-Rbiwheel-theorem}):

\begin{enumerate}[(i)]
\item $y_1$ is an internal vertex of~$P$,
\item every vertex of~$P$ is cubic in~$G$; furthermore, if it lies in~$A$ then
it is adjacent with~$u$, and if it lies in~$B$ then it is adjacent with~$w$, and
\item for every internal vertex~$v_i$ of~$P$, the edge that joins $v_i$
to one of $u$ and $w$ is \Rthin\ of index one.
\end{enumerate}

(Note that the path $y_0y_1y_2$ shown in Figure~\ref{fig:index-one-Rthin-edge}
satisfies all of the above properties; thus such a path~$P$ exists.)

\begin{figure}[!ht]
\centering
\begin{tikzpicture}

\draw (-0.5,0) -- (9.5,0);
\draw (0,0)node[fill=black]{} -- (4.5,-2.5);
\draw (1.5,0)node{} -- (4.5,2.5);
\draw (3,0)node[fill=black]{} -- (4.5,-2.5);
\draw (4.5,0)node{} -- (4.5,2.5);
\draw (6,0)node[fill=black]{} -- (4.5,-2.5);
\draw (7.5,0)node{} -- (4.5,2.5)node[fill=black]{};
\draw (9,0)node[fill=black]{} -- (4.5,-2.5)node{};

\draw (0,0.35)node[nodelabel]{$v_1$};
\draw (1.5,-0.35)node[nodelabel]{$v_2$};
\draw (3,0.35)node[nodelabel]{$v_3$};
\draw (4.5,-0.35)node[nodelabel]{$v_4$};
\draw (6,0.35)node[nodelabel]{$v_5$};
\draw (7.5,-0.35)node[nodelabel]{$v_6$};
\draw (9,0.35)node[nodelabel]{$v_7$};

\draw (4.5,2.85)node[nodelabel]{$u$};
\draw (4.5,-2.85)node[nodelabel]{$w$};

\end{tikzpicture}
\caption{Illustration for the $R$-biwheel Theorem}
\label{fig:illustration-for-Rbiwheel-theorem}
\bigskip
\end{figure}

We adjust notation so that $v_1$ lies in~$B$
as shown in Figure~\ref{fig:illustration-for-Rbiwheel-theorem}.
It should be noted that the other end of~$P$, namely $v_j$,
may lie in~$A$ or in~$B$, depending on whether $P$ is an odd path
or even.
We shall let $K$ denote the subgraph of~$H$, which has vertex set
$V(P) \cup \{u,w\}$ and edge set
$E(P) \cup \{v_iw : 1 \leq i \leq j{\rm ~and~}i {\rm ~odd}\}
\cup \{v_iu : 1 \leq i \leq j{\rm ~and~}i {\rm ~even}\}$.

\smallskip
Our goal is to show that $K$ is an \Rbiwheel. To this end, we need
to establish two additional properties of the path~$P$: first, that it is an odd path;
and second, that both its ends $v_1$ and $v_j$ lie in~$V(R)$.

\smallskip
We begin by arguing that the two ends of~$P$ are nonadjacent (in~$G$).
Suppose not, that is, say~$v_1v_j$ is an edge of~$G$.
Since each vertex of~$P$ is cubic, it follows that $V(G) = V(K)$; since otherwise
$\{u,w\}$ is a $2$-vertex-cut of~$G$, and we have a contradiction.
Since $G$ has an even number of vertices, $P$ is of odd length.
Furthermore, either $G$ is the same as~$K$, or otherwise,
$G$ has an additional edge joining $u$ and $w$. In both cases, the graph~$G$
is bipartite; this is absurd. Thus $v_1$ and $v_j$ are nonadjacent.

\smallskip
Now, let $f$ denote the edge~$v_1w$. We will argue that $f$ is not
\Rcomp, and then use this fact to deduce that $v_1 \in V(R)$.
 Suppose instead that $f$ is \Rcomp.
Applying Lemma~\ref{lem:index-one-Rthin-edge}{\it (ii)}, with~$v_2u$
playing the role of~$e$, we conclude that $f$ is \Rthin\ and its index is one.
Let $v_0$ denote the neighbour of~$v_1$ which is distinct from~$v_2$
and $w$; note that $v_0 \in A$. By the preceding paragraph, $v_0$ is distinct
from~$v_j$, and since each vertex of~$P$ is cubic, $v_0$ is not in~$V(P)$.
Applying Lemma~\ref{lem:index-one-Rthin-edge} again, this time with~$f$ playing
the role of~$e$, we deduce that $v_0$ is cubic.
Furthermore, $v_0$ and $v_2$ have a common neighbour
whose degree is four or more; thus $v_0$ is adjacent with~$u$.
Observe that the path~$v_0v_1P$ contradicts the maximality of~$P$.
We conclude that $f=v_1w$ is not \Rcomp.
By Lemma~\ref{lem:index-one-Rthin-edge}{\it (i)},
the cubic end~$v_1$ of~$f$
lies in~$V(R)$.

\smallskip
A similar argument shows that $v_j$ lies in~$V(R)$.
Since $v_1$ and $v_j$ are nonadjacent, one of them lies in~$A$ and the other
one lies in~$B$. (As per our notation, $v_1 \in B$ and $v_j \in A$.)
In particular, $P$ is an odd path, and thus $K$ is an \Rbiwheel.
Observe that by property (i) of the path~$P$, the end~$y_1$ of $e$
is an internal vertex of~$P$, whence $e$ is an internal spoke of~$K$, as desired.
This completes the proof of Theorem~\ref{thm:Rbiwheel-configuration}.
\end{proofOf}

\subsection{Proof of the $R$-ladder Theorem}
\label{sec:proof-Rladder-theorem}

Here, we prove the \mbox{$R$-ladder} Theorem (\ref{thm:Rladder-configuration});
its proof is significantly longer than that of the
\mbox{$R$-biwheel} Theorem.
In its proof, we will need two lemmas
(\ref{lem:index-two-Rthin-edge-no-common-neighbour}
and \ref{lem:index-two-Rthin-edge-common-neighbour}),
each of which pertains to the structure of \Rthin\ edges of index two
(in an \Rbrick\ which is free of \sRthin\ edges);
these lemmas correspond to two cases that appear in the proof
of Theorem~\ref{thm:Rladder-configuration}.

\begin{lem}
\label{lem:index-two-Rthin-edge-no-common-neighbour}
Let $G$ be a simple \Rbrick\ which is free of strictly \Rthin\ edges
and $e:=y_1z_1$ an \Rthin\ edge whose index is two.
Let $y_0$~and~$y_2$ denote the neighbours of~$y_1$
which are distinct from~$z_1$,
and let $z_0$~and~$z_2$ denote the neighbours of~$z_1$
which are distinct from~$y_1$.
Suppose that $y_1$ is the only common neighbour of $y_0$~and~$y_2$,
and that $z_1$ is the only common neighbour of $z_0$~and~$z_2$.
Then there are precisely two (nonadjacent) edges, say $f$~and~$g$,
between $\{y_0,y_2\}$ and $\{z_0,z_2\}$. Adjust notation so that
$f:=y_0z_0$ and $g:=y_2z_2$. Furthermore, the following statements hold:
\begin{enumerate}[(i)]
\item if $f$ is not \Rcomp\ then an end of~$f$ is cubic and it lies in~$V(R)$,
and
\item if $f$ is \Rcomp\ then it is \Rthin\ and its index is two.
\end{enumerate}
(Similar statements also apply to~$g$.)
\end{lem}
\begin{proof}
Let $J$ denote the retract of~$G-e$, that is, $J$ is obtained from~$G-e$
by bicontracting vertices $y_1$~and~$z_1$. By hypothesis, $e$ is not
strictly thin, whence $J$ has multiple edges.
Also, as stated in the assumptions, $y_1$ is the only common
neighbour of $y_0$~and~$y_2$, and likewise, $z_1$ is the only common
neighbour of $z_0$~and~$z_2$. It follows that there are precisely two nonadjacent
edges between $\{y_0,y_2\}$ and $\{z_0,z_2\}$,
as shown in Figure~\ref{fig:f-and-g-adjacent-or-not}b.
As in the statement,
adjust notation so that $f:=y_0z_0$ and $g:=y_2z_2$.

\smallskip
First consider the case in which $f$ is not \Rcomp.
That is, either $f$ is not removable in~$H$ or it is not removable in~$G$,
and it follows from
Lemma~\ref{lem:non-removable-in-H} or from
Lemma~\ref{lem:index-two-non-removable},
respectively, that an end of~$f$ is cubic and it lies in~$V(R)$.

\smallskip
Now consider the case in which $f$ is \Rcomp.
Since $f$ and $g$ are nonadjacent,
Lemma~\ref{lem:removable-not-thin} implies that $f$ is thin,
whence it is \Rthin. It remains to argue that the index of~$f$ is
two. Suppose to the contrary that an end of~$f$, say~$z_0$, is non-cubic.
By hypothesis, $f$ is not strictly \Rthin, whence its other end~$y_0$ is cubic.
Using the fact that $y_1$ is the only common neighbour of $y_0$~and~$y_2$,
it is easily verified that the retract of~$G-f$ has no multiple edges, that is,
$f$ is strictly \Rthin; this contradicts the hypothesis.
Thus, each end of~$f$ is cubic,
whence the index of~$f$ is two.
\end{proof}

\begin{lem}
\label{lem:index-two-Rthin-edge-common-neighbour}
Let $G$ be a simple \Rbrick\ which is free of strictly \Rthin\ edges
and $e:=y_1z_1$ an \Rthin\ edge whose index is two.
Let $y_0$~and~$y_2$ denote the neighbours of~$y_1$
which are distinct from~$z_1$,
and let $z_0$~and~$z_2$ denote the neighbours of~$z_1$
which are distinct from~$y_1$.
Suppose that $y_0$~and~$y_2$ have a common neighbour~$w$
which is distinct from $y_1$. Let $f:=y_0w$ and $g:=y_2w$.
Then $w$ is non-cubic and is distinct from each of $z_0$~and~$z_2$.
Furthermore, $f$~and~$g$ are both removable, $y_0$~and~$y_2$ are both cubic,
and the following statements hold:
\begin{enumerate}[(i)]
\item one of $f$~and~$g$ is \Rcomp; adjust notation so that $f$
is \Rcomp;
\item $f$ is not thin, and its cubic end~$y_0$ is adjacent with (exactly) one
of $z_0$~and~$z_2$; and,
\item $g$ is thin but it is not \Rcomp, and its cubic end~$y_2$
lies in~$V(R)$.
\end{enumerate}
\end{lem}
\begin{proof}
Note that $f$~and~$g$ are multiple edges in the retract~$J$ of~$G-e$.
Since $f$ and $g$ are adjacent,
by Proposition \ref{prop:f-g-adjacent}, their common end~$w$ is non-cubic.
Consequently, by Lemma~\ref{lem:index-two-non-removable},
$f$~and~$g$ are both removable. Note that $y_0$~and~$y_2$ are
nonadjacent, since otherwise $e$ is non-removable. In particular,
at least one of $y_0$~and~$y_2$ does not lie in~$V(R)$.
By Lemma~\ref{lem:non-removable-in-H}, at least one of $f$~and~$g$
is \Rcomp.

\smallskip
We now argue that $w$ is distinct from each of $z_0$~and~$z_2$.
Suppose not, and assume without loss of generality that $w=z_0$.
By Lemma~\ref{lem:removable-not-thin}{\it (ii)}, $f$~and~$g$ are both thin;
in particular, at least one of them is \Rthin.
Adjust notation so that $f$ is \Rthin.
By hypothesis, $f$ is not \sRthin, whence
the retract of~$G-f$ has multiple edges; consequently, the
end~$y_0$ of~$f$ is cubic. Let $v$ denote the
neighbour of~$y_0$ which is distinct from $y_1$~and~$z_0$.
Furthermore, as $f$ is not \sRthin, we infer that $v$~and~$y_1$
have a common neighbour which is distinct from~$y_0$;
by Proposition~\ref{prop:f-g-adjacent}, such a common neighbour
is non-cubic. Since $z_1$ is cubic, we
infer that $y_2$ is non-cubic.
By Lemma~\ref{lem:non-removable-in-H}, $g$ is \Rcomp.
As noted earlier, $g$ is thin; whence $g$ is \Rthin.
Since each end of~$g$ is non-cubic,
$g$ is \sRthin, contrary to the hypothesis.
Thus $w$ is distinct from each of $z_0$~and~$z_2$;
see Figure~\ref{fig:index-two-Rthin-edge-common-neighbour}.

\begin{figure}[!ht]
\centering
\begin{tikzpicture}
\draw (0,-2) -- (-0.35,-2.35);
\draw (0,-2) -- (0.35,-2.35);

\draw (-1.5,2) -- (1.5,2);
\draw (-1.5,2) node{}node[nodelabel,above]{$z_0$};
\draw (1.5,2) node{}node[nodelabel,above]{$z_2$};

\draw (-1,-1.2)node[nodelabel]{$f$};
\draw (1,-1.2)node[nodelabel]{$g$};

\draw (-1.5,0) -- (0,-2);
\draw (1.5,0) -- (0,-2);
\draw (0,-2)node{};
\draw (0,-2.35)node[nodelabel]{$w$};

\draw (0,0) -- (0,2);
\draw (0,0) -- (-1.5,0);
\draw (0,0) -- (1.5,0);

\draw (0,0) node{};
\draw (0,-0.35)node[nodelabel]{$y_1$};
\draw (0,2) node[fill=black]{}node[nodelabel,above]{$z_1$};
\draw (-1.5,0) node[fill=black]{}node[nodelabel,left]{$y_0$};
\draw (1.5,0) node[fill=black]{}node[nodelabel,right]{$y_2$};
\draw (0.25,1)node[nodelabel]{$e$};


\end{tikzpicture}
\caption{Illustration for
Lemma~\ref{lem:index-two-Rthin-edge-common-neighbour}}
\label{fig:index-two-Rthin-edge-common-neighbour}
\bigskip
\end{figure}

Let us review what we have proved so far. We have shown that $y_0$~and~$y_2$
are not both adjacent with~$z_0$. An analogous argument shows that
$y_0$~and~$y_2$ are not both adjacent with~$z_2$.
By symmetry, $z_0$~and~$z_2$ are not both adjacent with~$y_0$;
likewise, $z_0$~and~$z_2$ are not both adjacent with~$y_2$.
In summary, there are at most two edges between
$\{y_0,y_2\}$~and~$\{z_0,z_2\}$; and if there are precisely two such edges
then they are nonadjacent.

\smallskip
Now we argue that $y_0$ and $y_2$ are both cubic.
Suppose instead that $y_0$ is non-cubic; then,
by Lemma~\ref{lem:non-removable-in-H}, $f$ is \Rcomp.
Note that since each end of~$f$ is non-cubic, if $f$ is thin
then it is \sRthin, contrary to the hypothesis. So it must be
the case that $f$ is not thin. By Lemma~\ref{lem:removable-not-thin}{\it (iv)},
$N(y_0) \subseteq N(z_1) \cup \{w\} = \{z_0,z_2,y_1,w\}$. As $y_0$
is non-cubic, it must be adjacent with each of $z_0$~and~$z_2$;
however, this contradicts what we have already established in the preceding
paragraph. We conclude that $y_0$~and~$y_2$ are both cubic.

\smallskip
As noted earlier, at least one of $f$~and~$g$ is \Rcomp.
As in statement {\it (i)} of the lemma,
adjust notation so that $f$ is \Rcomp.
We will now argue that $f$ is not thin.

\smallskip
Suppose instead that $f$ is thin.
Let $v$ denote the neighbour of~$y_0$ which is distinct from
$y_1$~and~$w$.
By hypothesis, $f$ is not \sRthin, whence $v$~and~$y_1$ have a common
neighbour which is distinct from~$y_0$;
by Proposition~\ref{prop:f-g-adjacent}, such a common neighbour
is non-cubic. However, this is not possible as each neighbour of~$y_1$
is cubic. Thus, $f$ is not thin.
An analogous argument shows that if $g$ is \Rcomp\ then $g$ is not thin.

\smallskip
Since $f$ is removable but it is not thin,
by Lemma~\ref{lem:removable-not-thin}{\it (iv)},
$N(y_0) \subseteq N(z_1) \cup \{w\} = \{z_0,z_2,y_1,w\}$.
It follows from our previous observation that $y_0$ is adjacent with
exactly one of $z_0$~and~$z_2$; adjust notation so that $y_0$ is adjacent
with~$z_0$. This proves statement {\it (ii)}.

\smallskip
Also, by Lemma~\ref{lem:removable-not-thin}, one of $f$~and~$g$ is thin;
as per our notation, $g$ is thin. Consequently, $g$ is not \Rcomp.
By Lemma~\ref{lem:non-removable-in-H}, the cubic end~$y_2$
of~$g$ lies in~$V(R)$. This proves statement {\it (iii)}, and we are done.
\end{proof}

\begin{proofOf}{the $R$-ladder Theorem~(\ref{thm:Rladder-configuration})}
As in the statement of the theorem, let $G$ be a simple \Rbrick\ which
is free of strictly \Rthin\ edges, and let $e$ denote an \Rthin\ edge
whose index is two.
We shall let $y_1$~and~$z_1$ denote the
ends of~$e$, where $y_1 \in A$ and $z_1 \in B$. Furthermore, we
let $y_0$~and~$y_2$ denote the neighbours of~$y_1$ which are distinct
from~$z_1$, and likewise, we let $z_0$~and~$z_2$ denote the neighbours
of~$z_1$ which are distinct from~$y_1$.

\smallskip
Our goal is to show that $G$ has an \Rladder\ which contains the edge~$e$.
As mentioned earlier, we will consider two separate cases which correspond
to the situations in
Lemmas~\ref{lem:index-two-Rthin-edge-no-common-neighbour}
and \ref{lem:index-two-Rthin-edge-common-neighbour}, respectively.

\bigskip
\noindent
\underline{Case 1}: $y_1$ is the only common neighbour of $y_0$~and~$y_2$, and likewise, $z_1$
is the only common neighbour of $z_0$~and~$z_2$.

\medskip
\noindent
By Lemma~\ref{lem:index-two-Rthin-edge-no-common-neighbour}, there are precisely two
nonadjacent edges between $\{y_0,y_2\}$ and $\{z_0,z_2\}$. Adjust notation so that
$y_0z_0$ and $y_2z_2$ are edges of~$G$, as shown in
Figure~\ref{fig:no-common-neighbour}. Observe that the graph in the figure is a ladder of which
$e$ is an internal rung; furthermore, it is a subgraph of~$H$.

\begin{figure}[!ht]
\medskip
\centering
\begin{tikzpicture}

\draw (-1.5,0) -- (-1.5,2);
\draw (1.5,0) -- (1.5,2);

\draw (-1.5,2)node{}node[nodelabel,above]{$z_0$} --
(1.5,2)node{}node[nodelabel,above]{$z_2$};


\draw (0,0) -- (0,2);
\draw (0,0) -- (-1.5,0);
\draw (0,0) -- (1.5,0);

\draw (0,0) node{}node[nodelabel,below]{$y_1$};
\draw (0,2) node[fill=black]{}node[nodelabel,above]{$z_1$};
\draw (-1.5,0) node[fill=black]{}node[nodelabel,below]{$y_0$};
\draw (1.5,0) node[fill=black]{}node[nodelabel,below]{$y_2$};
\draw (0.25,1)node[nodelabel]{$e$};
\end{tikzpicture}
\vspace*{-0.1in}
\caption{The situation in Case 1}
\label{fig:no-common-neighbour}
\bigskip
\end{figure}
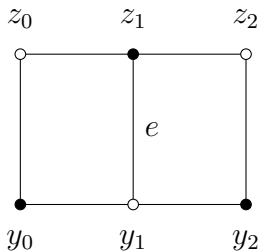

We let $K$ be a subgraph of~$H$ of maximum order that has the following properties:

\begin{enumerate}[(i)]
\item $K$ is a ladder and $e$ is an internal rung of~$K$, and
\item every internal rung of~$K$ is an \Rthin\ edge whose index is two.
\end{enumerate}

Note that the subgraph~$K$ is either an odd ladder or an even ladder;
see Figure~\ref{fig:illustration-for-Rladder-theorem-Case1}.
We shall denote by $au$ and $bw$ the external rungs of~$K$ such that
$a, w \in A$ and $b, u \in B$,
as shown in the figure.
It follows from property (ii) of~$K$ that each of its vertices,
except possibly $a, u, b$~and~$w$,
is cubic in~$G$.

\begin{rem}
\label{rem:ladder-of-order-six}
Note that, if~$|V(K)|=6$ then $K$ is the same as the subgraph of~$H$
shown in Figure~\ref{fig:no-common-neighbour};
in particular, $\{u,b\}=\{y_0,y_2\}$, and likewise,
$\{w,a\}=\{z_0,z_2\}$; consequently, by our hypothesis, $y_1$
is the only common neighbour of $u$~and~$b$, and likewise,
$z_1$ is the only common neighbour of $w$~and~$a$.
\end{rem}

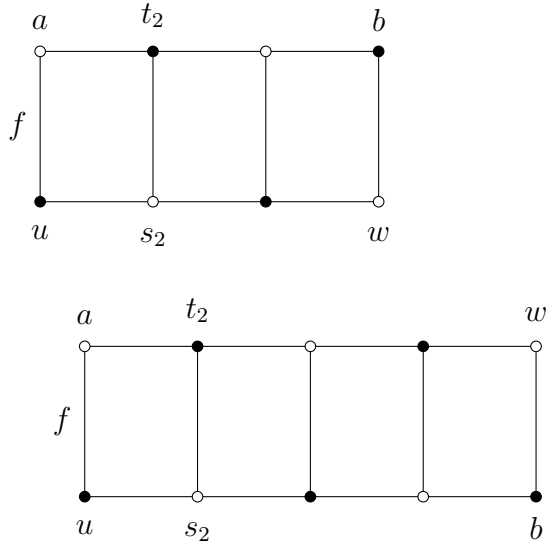
\begin{figure}[!ht]
\centering
\begin{tikzpicture}
\draw (0,0) -- (4.5,0) -- (4.5,2) -- (0,2) -- (0,0);

\draw (1.5,0)node{} -- (1.5,2)node[fill=black]{};
\draw (3,0)node[fill=black]{} -- (3,2)node{};

\draw (0,0)node[fill=black]{}node[below,nodelabel]{$u$};
\draw (0,2)node{}node[above,nodelabel]{$a$};
\draw (4.5,0)node{}node[below,nodelabel]{$w$};
\draw (4.5,2)node[fill=black]{}node[above,nodelabel]{$b$};

\draw (1.5,0)node[below,nodelabel]{$s_2$};
\draw (1.5,2)node[above,nodelabel]{$t_2$};

\draw (-0.3,1)node[nodelabel]{$f$};

\end{tikzpicture}
\hspace*{1in}
\begin{tikzpicture}
\draw (0,0) -- (6,0) -- (6,2) -- (0,2) -- (0,0);

\draw (1.5,0)node{} -- (1.5,2)node[fill=black]{};
\draw (3,0)node[fill=black]{} -- (3,2)node{};
\draw (4.5,0)node{} -- (4.5,2)node[fill=black]{};

\draw (0,0)node[fill=black]{}node[below,nodelabel]{$u$};
\draw (0,2)node{}node[above,nodelabel]{$a$};
\draw (6,2)node{}node[above,nodelabel]{$w$};
\draw (6,0)node[fill=black]{}node[below,nodelabel]{$b$};

\draw (1.5,0)node[below,nodelabel]{$s_2$};
\draw (1.5,2)node[above,nodelabel]{$t_2$};

\draw (-0.3,1)node[nodelabel]{$f$};

\end{tikzpicture}
\caption{Illustration for Case 1 of the $R$-ladder Theorem}
\label{fig:illustration-for-Rladder-theorem-Case1}
\bigskip
\end{figure}

Our goal is to show that $K$ is an \Rladder. To this end, we need to
establish that $a$~and~$b$ (or likewise, $u$ and $w$)
are both cubic in~$G$ and they lie in~$V(R)$.

\smallskip
Now, let $f$ denote the edge~$au$.
We will argue that $f$ is not \Rcomp,
and then use this fact to deduce that one of the ends of~$f$ is cubic and it lies
in~$V(R)$.
As shown in Figure~\ref{fig:illustration-for-Rladder-theorem-Case1},
let $s_2$ denote the neighbour of~$u$ in~$K$ which is distinct
from~$a$, and likewise, let $t_2$ denote the neighbour of~$a$
in~$K$ which is distinct from~$u$.

\smallskip
Suppose instead that $f$ is \Rcomp.
By Lemma~\ref{lem:index-two-Rthin-edge-no-common-neighbour}{\it (ii)},
with $s_2t_2$ playing the role of~$e$,
we conclude that $f$ is \Rthin\ and its index is two.
We shall let $s_0$ denote the neighbour of~$u$
which is distinct from $s_2$~and~$a$,
and likewise,
let~$t_0$ denote the neighbour of~$a$
which is distinct from $t_2$~and~$u$.
Note that $s_0 \in A$ and $t_0 \in B$.
It is easily seen that if $s_0$ is the same as~$w$
then $V(K) \cap A$ is a (nontrivial) barrier of~$G$;
this is absurd as~$G$ is a brick.
Thus $s_0 \neq w$, and likewise, $t_0 \neq b$.
It follows that $s_0, t_0 \notin V(K)$.

\smallskip
We will use the fact that $f$ is not strictly \Rthin\ to deduce that $s_0$~and~$t_0$
are adjacent; this will help us contradict the maximality of~$K$.
First suppose that $s_0$~and~$s_2$ have a common neighbour~$x$ which is
distinct from~$u$. By Proposition~\ref{prop:f-g-adjacent}, $x$
is non-cubic. Observe that, if $|V(K)| \geq 8$ then every neighbour of~$s_2$
is cubic; and if $|V(K)| = 6$ then $b$ is the only neighbour of~$s_2$ which
is possibly non-cubic.
We conclude that $|V(K)| = 6$ and that $x = b$. Now, $s_0$
is a common neighbour of $u$ and $b$; this contradicts the hypothesis
(see Remark~\ref{rem:ladder-of-order-six}).
We conclude that $u$ is the only common neighbour of $s_0$~and~$s_2$.
An analogous argument shows that $a$ is the only common neighbour of
$t_0$~and~$t_2$. It follows that $s_0$ and $t_0$ are adjacent, as $f$
is not strictly thin. Now, let $K'$ denote the subgraph of~$H$
obtained from~$K$ by adding the vertices $s_0$~and~$t_0$, and the
edges $us_0, s_0t_0$ and $t_0a$; then $K'$ contradicts the maximality of~$K$.

\smallskip
We thus conclude that~$f=au$ is not \Rcomp. Consequently,
by Lemma~\ref{lem:index-two-Rthin-edge-no-common-neighbour}{\it (i)},
with $s_2t_2$ playing the role of~$e$,
at least one of $a$ and $u$ is cubic and it also
lies in~$V(R)$. Adjust notation so that $a$ is cubic and it lies in~$V(R)$.

\smallskip
An analogous argument shows that at least
one of $b$ and $w$ is cubic and it lies in~$V(R)$;
we claim that $b$ must satisfy both of these properties. Suppose not; then
$w$ is cubic and it lies in~$V(R)$; this means that the edge~$\alpha$ of~$R$
joins the vertices $a$~and~$w$. Observe that $\{b,u\}$ is a $2$-vertex cut of~$G$;
this is absurd as $G$ is a brick.

\smallskip
We have shown that $a$ and $b$ both are cubic and they
lie in~$V(R)$. Thus $K$ is an \Rladder. Observe that, by property (i) of~$K$,
the edge~$e$ is an internal rung of~$K$. In particular, $e$ is an edge of~$K$,
as desired.

\bigskip
\noindent
\underline{Case 2}: $y_0$~and~$y_2$ have a common neighbour which
is distinct from~$y_1$, or likewise, $z_0$~and~$z_2$ have a common neighbour
which is distinct from~$z_1$.

\medskip
\noindent
As shown in Figure~\ref{fig:common-neighbour},
assume without loss of generality that $y_0$~and~$y_2$
have a common neighbour, say~$w$, which is distinct from~$y_1$.
We let $f:=y_0w$ and $g:=y_2w$.
We invoke Lemma~\ref{lem:index-two-Rthin-edge-common-neighbour} to
infer the following: $w$ is non-cubic and
it is distinct from each of $z_0$~and~$z_2$;
whereas $y_0$ and $y_2$ are both cubic; $f$~and~$g$ are both removable edges.
Furthermore, adjusting notation as in the lemma, $f$ is \Rcomp\ but it is not
thin and its cubic end~$y_0$ is adjacent with one of $z_0$ and $z_2$.
Assume without loss of generality that $y_0$ is adjacent with~$z_0$.
The edge~$g$ is thin but it is not \Rcomp\ and its cubic end~$y_2$ lies
in~$V(R)$. As per our notation, $y_2$ is an end of $\beta$;
we shall let $x$ denote the other end of~$\beta$.

\begin{figure}[!ht]
\bigskip
\centering
\begin{tikzpicture}[scale=1]

\draw (-1,-1.2)node[nodelabel]{$f$};
\draw (1,-1.2)node[nodelabel]{$g$};

\draw (-1.5,0) -- (0,-2);
\draw (1.5,0) -- (0,-2);

\draw (0,-2) -- (-0.35,-2.35);
\draw (0,-2) -- (0.35,-2.35);

\draw (1.5,0) -- (3,0);
\draw (0,0) -- (0,2);
\draw (0,0) -- (-1.5,0);
\draw (0,0) -- (1.5,0);
\draw (-1.5,2) -- (1.5,2);
\draw (-1.5,0) -- (-1.5,2);

\draw (0,-2)node{};
\draw (0,-2.35)node[nodelabel]{$w$};
\draw (3,0)node[fill=black]{};
\draw (3,0.35)node[nodelabel]{$x$};
\draw (2.25,-0.3)node[nodelabel]{$\beta$};
\draw (0,0) node{};
\draw (0,-0.35)node[nodelabel]{$y_1$};
\draw (0,2) node[fill=black]{};
\draw (0,2.35) node[nodelabel]{$z_1$};
\draw (-1.5,0) node[fill=black]{}node[nodelabel,left]{$y_0$};
\draw (1.5,0) node[fill=black]{};
\draw (1.5,0.35) node[nodelabel]{$y_2$};
\draw (-1.5,2) node{}node[nodelabel,left]{$z_0$};
\draw (1.5,2) node{};
\draw (1.5,1.65) node[nodelabel]{$z_2$};
\draw (0.25,1)node[nodelabel]{$e$};
\end{tikzpicture}
\caption{The situation in Case 2}
\label{fig:common-neighbour}
\bigskip
\end{figure}
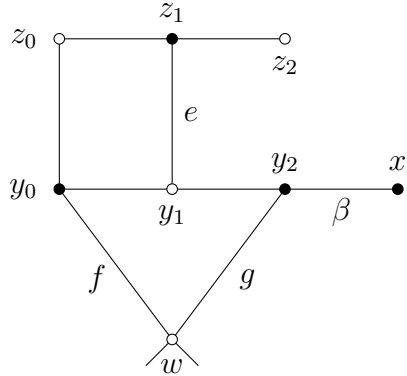

We will consider two subcases.
In the first one, we assume that $z_0$ is cubic and it lies in~$V(R)$;
and in the second case, we assume that either $z_0$ is non-cubic
or it is not in~$V(R)$.

\bigskip
\noindent
\underline{Case 2.1}: $z_0$ is cubic and it lies in~$V(R)$.

\medskip
\noindent
In this case, we shall denote by $K$ the subgraph whose vertex
set is~$\{z_0,z_1,y_0,y_1,w,y_2\}$ and edge set is~$\{e,y_1y_2,g,f,y_0z_0,z_0z_1,y_0y_1\}$.
Observe that $K$ is a ladder of order six and it is a subgraph of~$H$;
furthermore, two of its corners, namely $y_2$ and $z_0$, are cubic and they both lie in~$V(R)$.
To complete the proof in this case, we will show that $K$ is an \Rladder; for this,
we only need to prove that the internal rung~$y_0y_1$ is \Rthin\ and its index
is two.

\smallskip
We begin by showing that $y_0y_1$ is \Rcomp, that is, $y_0y_1$ is removable in~$H$ as well as in~$G$.
Here, we will not require the hypothesis that $z_0$ is cubic and it lies in~$V(R)$.

\begin{Claim}
\label{claim:y0y1-Rcompatible}
The edge~$y_0y_1$ is \Rcomp.
\end{Claim}
\begin{proof}
Note that $y_0y_1$ is removable in the subgraph~$K$.
We will argue that $K$ is a conformal subgraph of~$H$, and then use
Proposition~\ref{prop:conformal-exchange-property-removable-bipmcg}
to deduce that $y_0y_1$ is removable in~$H$.

\smallskip
Let $M$ be any perfect matching of~$H$ which contains the edge~$z_0z_1$.
Since $M$ does not contain $\alpha$ or $\beta$, it is easily verified
that~$M \cap E(K)$
is a perfect matching of~$K$, whence $K$ is a conformal subgraph of~$H$;
consequently, $y_0y_1$ is removable in~$H$.

\smallskip
To conclude that $y_0y_1$ is removable in~$G$, we will show that
$G-y_0y_1$ has a perfect matching~$M$ which contains both $\alpha$ and $\beta$.
Let $N$ be a perfect matching of~$G-\{z_1,x\}$; such a perfect matching exists
as $G$ is a brick; note that $\alpha \in N$ and $\beta \notin N$.
Clearly, either $y_1y_2 \in N$ or $g \in N$.
If $y_1y_2 \in N$, we let $M:= (N-y_1y_2) + e + \beta$.
On the other hand, if $g \in N$ then $y_0y_1 \in N$, and we let
$M:=(N-g-y_0y_1) + e + f + \beta$. In either case,
$M$ is the desired perfect matching,
and this completes the proof.
\end{proof}

We now proceed to show that $y_0y_1$ is an \Rthin\ edge.
To this end, we will use the characterization of \Rthin\
edges in terms of barriers given by
\cite[Proposition~2.9]{koth19}.

\begin{Claim}
\label{claim:y0y1-Rthin}
The edge~$y_0y_1$ is \Rthin, and its index is two.
\end{Claim}
\begin{proof}
Observe that, since $y_0$ and $y_1$ are both cubic, $G-y_0y_1$ has
two maximal nontrivial barriers; one of them, say~$S_A$,
is a subset of~$A$ and it contains $z_0$~and~$w$;
the other one, say~$S_B$, is a subset of~$B$ and it contains $z_1$~and~$y_2$.
In particular, the index of~$y_0y_1$ is two.

\smallskip
We will argue that $S_A = \{z_0,w\}$;
our argument does not use the fact that $w$ is non-cubic, and it may
be mimicked to show that $S_B = \{z_1,y_2\}$; thereafter, we apply
\cite[Proposition~2.9]{koth19}
to infer that $y_0y_1$ is \Rthin.

\smallskip
Note that $w$ is in the barrier~$S_A$. Now, let $v$ be any vertex in~$A-\{z_0,w\}$.
We will show
that $(G-y_0y_1) - \{w,v\}$ has a perfect matching~$M$; this would imply that $v$ is not in the barrier~$S_A$.
Let $N$ be a perfect matching of~$G-\{w,v\}$; note that $\beta \in N$ and
$\alpha \notin N$.
If $y_0y_1 \notin N$ then let $M:=N$, and we are done.
Now suppose that $y_0y_1 \in N$.
By our hypothesis, $z_0$ is cubic and it lies in~$V(R)$;
this means that the three edges incident at $z_0$ are $z_0y_0, z_0z_1$ and $\alpha$.
Since, $y_0y_1 \in N$ and $\alpha \notin N$ and $v \neq z_0$,
we conclude that $z_0z_1 \in N$.
Now, $M:= (N-y_0y_1-z_0z_1) + y_0z_0 + e$ is the desired perfect matching.
We conclude that $S_A = \{z_0,w\}$. As discussed in the preceding paragraph, this completes the proof.
\end{proof}

We have shown that the only internal rung of~$K$, namely $y_0y_1$, is
an \Rthin\ edge whose index is two. As discussed earlier,
$K$ is indeed an \Rladder, and since it contains~$e$,
this completes the proof in this case (2.1).

\bigskip
\noindent
\underline{Case 2.2}: Either $z_0$ is non-cubic or it does not lie in~$V(R)$, possibly both.

\medskip
\noindent
As per our notation, $z_0 \in A$; it follows from the hypothesis of this case
that $z_0$ has at least one neighbour which lies in~$B-\{z_1,y_0\}$;
we shall let $u$ denote such a neighbour of~$z_0$, as shown in
Figure~\ref{fig:ladder-of-order-eight}. Observe that $u$ is distinct
from~$y_2$; however, it is possible that $u=x$.

\begin{figure}[!ht]
\centering
\begin{tikzpicture}[scale=0.9]

\draw (-1.5,2) -- (0,4);
\draw (0,4)node[fill=black]{};
\draw (0,4.35)node[nodelabel]{$u$};

\draw (-1,-1.2)node[nodelabel]{$f$};
\draw (1,-1.2)node[nodelabel]{$g$};

\draw (-1.5,0) -- (0,-2);
\draw (1.5,0) -- (0,-2);

\draw (0,-2) -- (-0.35,-2.35);
\draw (0,-2) -- (0.35,-2.35);

\draw (1.5,0) -- (3,0);
\draw (0,0) -- (0,2);
\draw (0,0) -- (-1.5,0);
\draw (0,0) -- (1.5,0);
\draw (-1.5,2) -- (1.5,2);
\draw (-1.5,0) -- (-1.5,2);

\draw (0,-2)node{};
\draw (0,-2.35)node[nodelabel]{$w$};
\draw (3,0)node[fill=black]{};
\draw (3,0.35)node[nodelabel]{$x$};
\draw (2.25,-0.3)node[nodelabel]{$\beta$};
\draw (0,0) node{};
\draw (0,-0.35)node[nodelabel]{$y_1$};
\draw (0,2) node[fill=black]{};
\draw (0,2.35) node[nodelabel]{$z_1$};
\draw (-1.5,0) node[fill=black]{}node[nodelabel,left]{$y_0$};
\draw (1.5,0) node[fill=black]{};
\draw (1.5,0.35) node[nodelabel]{$y_2$};
\draw (-1.5,2) node{}node[nodelabel,left]{$z_0$};
\draw (1.5,2) node{};
\draw (1.5,1.65) node[nodelabel]{$z_2$};
\draw (0.25,1)node[nodelabel]{$e$};
\end{tikzpicture}
\vspace*{-0.05in}
\caption{The situation in Case 2.2 (all labelled vertices are pairwise
distinct, except possibly $u$ and $x$)}
\label{fig:ladder-of-order-eight}
\bigskip
\end{figure}
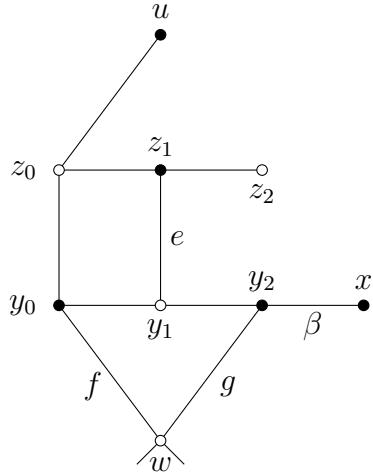

In this case, we will prove that $z_0z_1$ is an \Rthin\ edge whose index is two;
in particular, $z_0$ is cubic and $z_0 \notin V(R)$.
(If not, we will find a \sRthin\ edge contrary to the
hypothesis.) Thereafter, we argue that $u$
is adjacent with~$z_2$; this establishes a certain symmetry between
$y_0,y_1,y_2,w$ and $z_0,z_1,z_2,u$, respectively;
see Figure~\ref{fig:establish-symmetry}.
We shall exploit this to deduce that $y_0y_1$
is an \Rthin\ edge (whose index is two),
and that $z_2$ is cubic and it lies in~$V(R)$.
In the end, we will find
an \Rladder\ of order eight whose internal rungs are
$y_0y_1$~and~$z_0z_1$.

\smallskip
Our first step is to show that $z_0z_1$ is \Rcomp, that is, $z_0z_1$
is removable in~$H$ as well as in~$G$.

\begin{Claim}
\label{claim:z0z1-Rcompatible}
The edge~$z_0z_1$ is \Rcomp.
\end{Claim}
\begin{proof}
Note that $y_0y_1z_1z_0y_0$ is a $4$-cycle containing the edges
$y_0y_1$~and~$z_0z_1$. We will show that $y_0y_1$ is admissible
in~$H-z_0z_1$, and then invoke
Corollary~\ref{cor:quadrilateral-admissible-removable} to deduce that $z_0z_1$
is removable in~$H$.

\smallskip
We need to show that $H-z_0z_1$ has a perfect matching~$M$ which
contains~$y_0y_1$.
Let $N$ be any perfect matching of~$H-\{u,y_1\}$; such a perfect
matching exists by \cite[Proposition~2.1]{koth19}.
Observe that $g \in N$; consequently, $y_0z_0 \in N$.
Now, $M:=(N-y_0z_0) + uz_0 + y_0y_1$ is the desired perfect matching.
As discussed above, $z_0z_1$ is removable in~$H$.

\smallskip
To conclude that $z_0z_1$ is removable in~$G$, we will show
that $G-z_0z_1$ has a perfect matching~$M$ which contains
both $\alpha$~and~$\beta$. Let $N$ be any perfect matching
of~$G$ which contains $\alpha$ and $\beta$. If $z_0z_1 \notin N$
then let $M:=N$, and we are done.
Now suppose that $z_0z_1 \in N$.
Observe that $y_0y_1 \in N$; furthermore,
$M:= (N-y_0y_1-z_0z_1) + e + y_0z_0$ is the desired perfect
matching. This completes the proof.
\end{proof}

We proceed to prove that $z_0z_1$ is an \Rthin\ edge whose index is two.
As we did in Claim~\ref{claim:y0y1-Rthin},
we will use the characterization
of \Rthin\ edges given by
\cite[Proposition~2.9]{koth19}.
However, here we need more general arguments since we do not know
the degree of~$z_0$.

\begin{Claim}
\label{claim:z0z1-Rthin}
The edge~$z_0z_1$ is \Rthin, and its index is two.
\end{Claim}
\begin{proof}
Observe that, since $z_1$ is cubic, $G-z_0z_1$ has a maximal nontrivial
barrier, say~$S_A$, which is a subset of~$A$ and contains $y_1$~and~$z_2$.
We will first prove that $S_A=\{y_1,z_2\}$.

\smallskip
Let $v$ be any vertex in~$A-\{y_1,z_2\}$.
We will show that $(G-z_0z_1) - \{z_2,v\}$ has a perfect matching~$M$;
this would imply that $v$ is not in the barrier~$S_A$.
Let $N$ be a perfect matching of~$G-\{z_2,v\}$; note that $\beta \in N$
and $\alpha \notin N$. If $z_0z_1 \notin N$ then let $M:=N$,
and we are done.
Now suppose that $z_0z_1 \in N$,
and observe that $y_0y_1 \in N$;
consequently, $M:=(N-z_0z_1-y_0y_1)+e+y_0z_0$ is the desired
perfect matching. Thus, $S_A=\{y_1,z_2\}$.

\smallskip
Since $z_0z_1$ is \Rcomp, by \cite[Lemma~2.8]{koth19},
either~$S_A$ is the only maximal
nontrivial barrier of~$G-z_0z_1$,
or $G-z_0z_1$ has another maximal nontrivial barrier, say~$S_B$,
which is a subset of~$B$. We now argue that, in the former case,
$z_0z_1$ is \sRthin, contrary to the hypothesis.

\smallskip
Suppose that $S_A$ is the only maximal nontrivial barrier of~$G-z_0z_1$;
in this case, the index of~$z_0z_1$ is one.
By \cite[Proposition~2.9]{koth19}, $z_0z_1$ is \Rthin.
Also, $z_0$ is non-cubic, since otherwise its two neighbours distinct
from~$z_1$ would lie in a barrier.
Observe that,
since $z_1$ is the only common neighbour of $y_1$~and~$z_2$,
the retract of~$G-z_0z_1$ is simple, and thus $z_0z_1$ is \sRthin;
this is a contradiction.

\smallskip
It follows that $G-z_0z_1$ has a maximal nontrivial barrier,
say~$S_B$, which is a subset of~$B$;
in particular, the index of~$z_0z_1$ is two.
By \cite[Lemma~2.8]{koth19},
$z_0$ is isolated in~$(G-z_0z_1)-S_B$;
that is, in~$G-z_0z_1$, every neighbour
of~$z_0$ lies in the barrier~$S_B$. In particular,
$u, y_0 \in S_B$. We will prove that $S_B=\{u,y_0\}$.

\smallskip
Let $v$ be any vertex in~$B-\{u,y_0\}$. We will show that
$(G-z_0z_1) - \{u, v\}$ has a perfect matching~$M$; this
would imply that $v$ is not in the barrier~$S_B$.
Let $N$ be a perfect matching of~$G-\{u,v\}$;
note that $\alpha \in N$ and $\beta \notin N$.
If $z_0z_1 \notin N$ then let $M:=N$, and we are done.
Now suppose that $z_0z_1 \in N$.
If $y_0y_1 \in N$ then
$M:= (N-z_0z_1-y_0y_1) + e + y_0z_0$ is the desired
perfect matching.
Now suppose that $y_0y_1 \notin N$; then $f, y_1y_2 \in N$,
and $M:=(N-z_0z_1-f-y_1y_2) + y_0z_0 + g + e$ is the desired
perfect matching. Thus, as discussed above, $v \notin S_B$;
consequently, $S_B = \{u,y_0\}$.
In particular, $z_0$ is cubic.
Furthermore, by \cite[Proposition~2.9]{koth19},
$z_0z_1$ is \Rthin.
\end{proof}

We have shown that $z_0z_1$ is an \Rthin\ edge and its index is two;
in particular, both its ends are cubic. The three neighbours of $z_0$
are $y_0, z_1$ and $u$; see Figure~\ref{fig:ladder-of-order-eight}.

\smallskip
By hypothesis, $z_0z_1$ is not \sRthin; whence the retract of~$G-z_0z_1$
has multiple edges. Observe that $z_1$ is the only common neighbour
of $y_1$~and~$z_2$. Consequently, at least one of the following must hold:
either $u$~and~$y_0$ have a common neigbour which is distinct from~$z_0$,
or $u$ and $z_2$ are adjacent. We shall rule out the former case by arriving
at a contradiction.

\begin{figure}[!ht]
\smallskip
\centering
\begin{tikzpicture}[scale=0.9]

\draw (-1,-1.2)node[nodelabel]{$f$};

\draw (-1.5,0) -- (0,-2);
\draw (1.5,0) -- (0,-2);

\draw (0,-2) -- (-0.35,-2.35);
\draw (0,-2) -- (0.35,-2.35);

\draw (0,0) -- (0,2);
\draw (0,0) -- (-1.5,0);
\draw (0,0) -- (1.5,0);
\draw (-1.5,2) -- (1.5,2);
\draw (-1.5,0) -- (-1.5,2);

\draw (0,-2)node{};
\draw (0,-2.35)node[nodelabel]{$w$};
\draw (0,0) node{};
\draw (0,-0.35)node[nodelabel]{$z_0$};
\draw (0,2) node[fill=black]{};
\draw (0,2.35) node[nodelabel]{$z_1$};
\draw (-1.5,0) node[fill=black]{}node[nodelabel,left]{$y_0$};
\draw (1.5,0) node[fill=black]{};
\draw (2.8,0) node[nodelabel]{$u \in V(R)$};
\draw (-1.5,2) node{}node[nodelabel,left]{$y_1$};
\draw (1.5,2) node{};
\draw (1.5,1.65) node[nodelabel]{$z_2$};
\draw (0,-3.5)node[nodelabel]{(a)};
\end{tikzpicture}
\hspace*{0.5in}
\begin{tikzpicture}[scale=0.9]
\draw (0,4) to [out=180,in=90] (-3,1) to [out=270,in=180] (0,-2);
\draw (1.5,0) to [out=0,in=270] (2.5,2) to [out=90,in=0] (0,4);
\draw (-1.5,2) -- (0,4);
\draw (0,4)node[fill=black]{};
\draw (0,4.35)node[nodelabel]{$u=x$};

\draw (-1,-1.2)node[nodelabel]{$f$};
\draw (1,-1.2)node[nodelabel]{$g$};

\draw (-1.5,0) -- (0,-2);
\draw (1.5,0) -- (0,-2);

\draw (0,-2) -- (-0.35,-2.35);
\draw (0,-2) -- (0.35,-2.35);

\draw (0,0) -- (0,2);
\draw (0,0) -- (-1.5,0);
\draw (0,0) -- (1.5,0);
\draw (-1.5,2) -- (1.5,2);
\draw (-1.5,0) -- (-1.5,2);

\draw (0,-2)node{};
\draw (0,-2.35)node[nodelabel]{$w$};
\draw (0,0) node{};
\draw (0,-0.35)node[nodelabel]{$y_1$};
\draw (0,2) node[fill=black]{};
\draw (0,2.35) node[nodelabel]{$z_1$};
\draw (-1.5,0) node[fill=black]{}node[nodelabel,left]{$y_0$};
\draw (1.5,0) node[fill=black]{};
\draw (1.5,0.35) node[nodelabel]{$y_2$};
\draw (-1.5,2) node{}node[nodelabel,left]{$z_0$};
\draw (1.5,2) node{};
\draw (1.5,1.65) node[nodelabel]{$z_2$};
\draw (0,-3.5)node[nodelabel]{(b)};
\end{tikzpicture}
\vspace*{-0.05in}
\caption{When $u$ is adjacent with~$w$}
\label{fig:u-adjacent-with-w}
\bigskip
\end{figure}
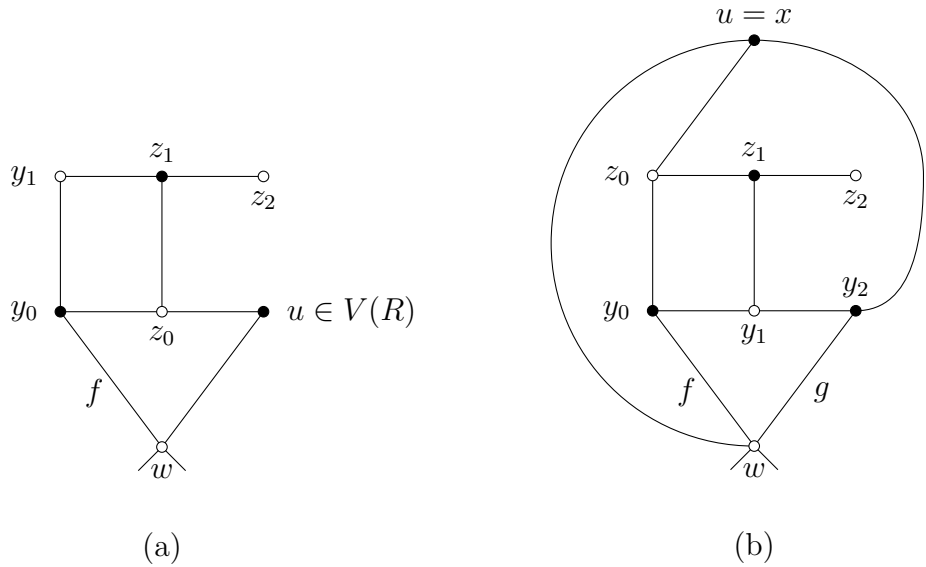

Suppose that $u$~and~$y_0$ have a common neighbour which is distinct from~$z_0$;
this is true if and only if $u$ is adjacent with~$w$.
We now invoke Lemma~\ref{lem:index-two-Rthin-edge-common-neighbour}, with $z_0z_1$ playing
the role of~$e$, with $u$ playing the role of~$y_2$, and with $uw$ playing the role of~$g$;
see Figure~\ref{fig:u-adjacent-with-w}a, and compare with Figure~\ref{fig:ladder-of-order-eight}.
The lemma implies that $u$ is a cubic vertex, and since $f$ is \Rcomp,
$uw$ is thin but it is not \Rcomp; furthermore, $u \in V(R)$.
In particular, $u$ is an end of~$\beta$ which implies that $u = x$;
see Figures~\ref{fig:ladder-of-order-eight} and \ref{fig:u-adjacent-with-w}b.
Note that all of the labelled vertices in Figure~\ref{fig:u-adjacent-with-w}b are pairwise distinct;
furthermore, each of them except $w$ and possibly~$z_2$, is cubic.
Since $z_2$ has at least one neighbour in~$B$ which is distinct from~$z_1$, the graph
has more vertices; consequently, $\{w,z_2\}$ is a $2$-vertex cut of~$G$; this
is a contradiction.

\smallskip
We have shown that $z_0$ is the only common neighbour of $u$~and~$y_0$;
as discussed earlier, this implies that $u$~and~$z_2$ are adjacent;
see Figure~\ref{fig:establish-symmetry}. Note that $u$ is now a common
neighbour of $z_0$ and $z_2$, and it is distinct from~$z_1$; this
establishes a symmetry between $y_0,y_1,y_2,w$, and $z_0,z_1,z_2,u$,
respectively.
We invoke
Lemma~\ref{lem:index-two-Rthin-edge-common-neighbour}
to conclude that $u$ is non-cubic,
whereas $z_2$ is cubic and it lies in~$V(R)$.
Using arguments analogous to those in the proofs
of Claims~\ref{claim:z0z1-Rcompatible}~and~\ref{claim:z0z1-Rthin},
we conclude that $y_0y_1$ is an \Rthin\ edge (whose index is two).

\begin{figure}[!ht]
\medskip
\centering
\begin{tikzpicture}[scale=0.9]
\draw (1.5,2) -- (0,4);
\draw (-1.5,2) -- (0,4);

\draw (0,4) -- (0.35,4.35);
\draw (0,4) -- (-0.35,4.35);

\draw (0,4)node[fill=black]{};
\draw (0,4.35)node[nodelabel]{$u$};

\draw (-1,-1.2)node[nodelabel]{$f$};
\draw (1,-1.2)node[nodelabel]{$g$};

\draw (-1.5,0) -- (0,-2);
\draw (1.5,0) -- (0,-2);

\draw (0,-2) -- (-0.35,-2.35);
\draw (0,-2) -- (0.35,-2.35);

\draw (1.5,0) -- (3,0);
\draw (0,0) -- (0,2);
\draw (0,0) -- (-1.5,0);
\draw (0,0) -- (1.5,0);
\draw (-1.5,2) -- (1.5,2);
\draw (-1.5,0) -- (-1.5,2);

\draw (1.5,2) -- (3,2);
\draw (3,2)node{};
\draw (2.25,2.25)node[nodelabel]{$\alpha$};

\draw (0,-2)node{};
\draw (0,-2.35)node[nodelabel]{$w$};
\draw (3,0)node[fill=black]{};
\draw (2.25,-0.3)node[nodelabel]{$\beta$};
\draw (0,0) node{};
\draw (0,-0.35)node[nodelabel]{$y_1$};
\draw (0,2) node[fill=black]{};
\draw (0,2.35) node[nodelabel]{$z_1$};
\draw (-1.5,0) node[fill=black]{}node[nodelabel,left]{$y_0$};
\draw (1.5,0) node[fill=black]{};
\draw (1.5,0.35) node[nodelabel]{$y_2$};
\draw (-1.5,2) node{}node[nodelabel,left]{$z_0$};
\draw (1.5,2) node{};
\draw (1.5,1.65) node[nodelabel]{$z_2$};
\draw (0.25,1)node[nodelabel]{$e$};
\end{tikzpicture}
\caption{Illustration for Case 2.2 of the $R$-ladder Theorem;
$u$ is a common neighbour of $z_0$~and~$z_2$ which
is distinct from~$z_1$}
\label{fig:establish-symmetry}
\bigskip
\end{figure}
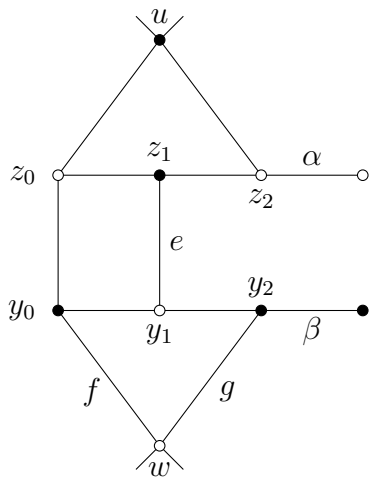

Now, let $K$ denote the subgraph which consists of all of the labelled
vertices shown in Figure~\ref{fig:establish-symmetry}, and all of the
edges between those vertices which are shown in the figure.
Note that $K$ is an \Rladder, and since it contains~$e$,
this completes the proof of the \mbox{$R$-ladder}
Theorem (\ref{thm:Rladder-configuration}).
\end{proofOf}

\section{Properties of {\Rconf}s}
\label{sec:properties-of-Rconfigurations}
In this section, we prove a few results pertaining to {\Rconf}s.
These are used in our proof of the
Strictly \Rthin\ Edge Theorem (\ref{thm:strictly-Rthin-nb-bricks}),
which appears in the next section.
We will find the following consequence of
\cite[Lemma~2.3]{koth19} useful; its proof
may be found in \cite{koth16}.
\begin{cor}
\label{cor:application-of-LV}
Let $G$ be an $R$-brick, and let $H:=G-R$. Then for any
vertex~$b$, at most two edges of $\partial_H(b)$ are
non-removable in~$H$.
\end{cor}

\smallskip
For the rest of this section, $G$ is a simple \Rbrick, and we
adopt Notation~\ref{Not:Rbrick-doubleton};
furthermore, $K_1$ is an \Rconf\ with \external\ $a_1u_1$~and~$b_1w_1$.
As usual, $u_1$~and~$w_1$ are the free corners of~$K_1$;
see Figure~\ref{fig:first-Rconfiguration}.

\begin{figure}[!ht]
\centering
\begin{tikzpicture}[scale=0.8]
\draw (1,0) -- (-0.5,0);
\draw (0.25,0.3)node[nodelabel]{$\alpha$};

\draw (6,0) -- (7.5,0);
\draw (6.75,-0.35)node[nodelabel]{$\beta$};

\draw (1,0) -- (6,0);

\draw (1,0) -- (3.5,2);
\draw (2,0) -- (3.5,-2);
\draw (3,0) -- (3.5,2);
\draw (4,0) -- (3.5,-2);
\draw (5,0) -- (3.5,2);
\draw (6,0) -- (3.5,-2);

\draw (1,0) node{}node[below,nodelabel]{$a_1$};
\draw (2,0) node[fill=black]{};
\draw (3,0) node{};
\draw (4,0) node[fill=black]{};
\draw (5,0) node{};
\draw (6,0) node[fill=black]{}node[above,nodelabel]{$b_1$};
\draw (3.5,2) node[fill=black]{}node[above,nodelabel]{$u_1$};
\draw (3.5,-2) node{}node[below,nodelabel]{$w_1$};

\draw (3.5,-3.7)node[nodelabel]{(a)};
\end{tikzpicture}

\begin{tikzpicture}[scale=0.9]
\draw (1,0) -- (-0.5,0);
\draw (0.25,0.3)node[nodelabel]{$\alpha$};

\draw (3,2) -- (4.5,2);
\draw (3.75,1.65)node[nodelabel]{$\beta$};

\draw (1,0) -- (1,2);
\draw (2,0) -- (2,2);
\draw (3,0) -- (3,2);

\draw (1,0) -- (3,0);
\draw (1,2) -- (3,2);

\draw (1,0) node{}node[below,nodelabel]{$a_1$};
\draw (1,2) node[fill=black]{}node[above,nodelabel]{$u_1$};
\draw (2,2) node{};
\draw (2,0) node[fill=black]{};
\draw (3,0) node{}node[below,nodelabel]{$w_1$};
\draw (3,2) node[fill=black]{}node[above,nodelabel]{$b_1$};

\draw (2,-1.5)node[nodelabel]{(b)};
\end{tikzpicture}
\hspace*{1in}
\begin{tikzpicture}[scale=0.9]
\draw (1,0) -- (-0.5,0);
\draw (0.25,0.3)node[nodelabel]{$\alpha$};

\draw (4,0) -- (5.5,0);
\draw (4.75,0.3)node[nodelabel]{$\beta$};

\draw (1,0) -- (1,2);
\draw (2,0) -- (2,2);
\draw (3,0) -- (3,2);
\draw (4,0) -- (4,2);

\draw (1,0) -- (4,0);
\draw (1,2) -- (4,2);

\draw (1,0) node{}node[below,nodelabel]{$a_1$};
\draw (1,2) node[fill=black]{}node[above,nodelabel]{$u_1$};
\draw (2,2) node{};
\draw (2,0) node[fill=black]{};
\draw (3,0) node{};
\draw (3,2) node[fill=black]{};
\draw (4,2) node{}node[above,nodelabel]{$w_1$};
\draw (4,0) node[fill=black]{}node[below,nodelabel]{$b_1$};

\draw (2.5,-1.5)node[nodelabel]{(c)};
\end{tikzpicture}
\vspace*{-0.1in}
\caption{The \Rconf~$K_1$}
\label{fig:first-Rconfiguration}
\bigskip
\end{figure}

Note that $K_1$ is either a ladder or a partial biwheel.
In either case, it is easily verified that the graph obtained from~$K_1$
by adding two edges, one joining $a_1$~and~$b_1$, and another joining
$u_1$~and~$w_1$, is a brace. This fact, in conjunction
with the characterization of braces provided by
\cite[Proposition~4.12]{koth16},
yields the following easy observation.

\begin{prop}
\label{prop:bracelike-property-of-Rconfigurations}
The following statements hold:
\begin{enumerate}[(i)]
\item for every pair of distinct vertices \mbox{$v_1,v_2 \in A \cap V(K_1)$},
the graph $K_1 - \{b_1,u_1,v_1,v_2\}$ has a perfect matching; and likewise,
\item for every pair of distinct vertices \mbox{$v_1,v_2 \in B \cap V(K_1)$},
the graph $K_1 - \{a_1,w_1,v_1,v_2\}$ has a perfect matching. \qed
\end{enumerate}
\end{prop}

In the following lemma, we prove some conformality
properties of {\Rconf}s; these are useful in subsequent lemmas
to show that a certain edge is \Rcomp.

\begin{lem}
\label{lem:conformality-of-Rconfigurations}
The following statements hold:
\begin{enumerate}[(i)]
\item $u_1$ lies in~$V(R)$ if and only if $w_1$ lies in $V(R)$,
\item $K_1$ is a conformal matching covered subgraph, and
\item the subgraph induced by $E(K_1) \cup R$ is conformal.
\end{enumerate}
\end{lem}
\begin{proof}
First, we prove {\it (i)}.
Suppose instead that $u_1 \in V(R)$ and $w_1 \notin V(R)$;
that is, $u_1=b_2$, whereas $w_1$~and~$a_2$ are distinct;
see Figure~\ref{fig:only-one-free-corner-in-VR}.
For $X:=V(K_1)-w_1$, note that every edge in~$\partial(X)$,
except for~$\alpha$, is either incident with~$u_1$ or with~$w_1$.
Recall that if $M$ is any perfect matching,
then $\alpha \in M$ if and only if $\beta \in M$.
Using these facts,
it is easy to see that $\partial(X)$ is a tight cut; this
is a contradiction.

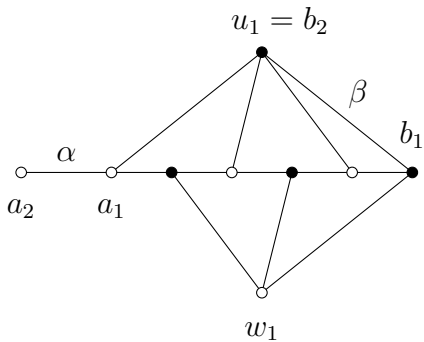
\begin{figure}[!ht]
\centering
\smallskip
\begin{tikzpicture}[scale=0.8]
\draw (3.8,1.4) node[above,nodelabel]{$u_1=b_2$};
\draw (1,0) -- (-0.5,0);
\draw (0.25,0.3)node[nodelabel]{$\alpha$};
\draw (-0.5,0)node{}node[below,nodelabel]{$a_2$};


\draw (6,0) -- (3.5,2);
\draw (5.1,1.3)node[nodelabel]{$\beta$};

\draw (1,0) -- (6,0);

\draw (1,0) -- (3.5,2);
\draw (2,0) -- (3.5,-2);
\draw (3,0) -- (3.5,2);
\draw (4,0) -- (3.5,-2);
\draw (5,0) -- (3.5,2);
\draw (6,0) -- (3.5,-2);

\draw (1,0) node{}node[below,nodelabel]{$a_1$};
\draw (2,0) node[fill=black]{};
\draw (3,0) node{};
\draw (4,0) node[fill=black]{};
\draw (5,0) node{};
\draw (6,0) node[fill=black]{}node[above,nodelabel]{$b_1$};
\draw (3.5,2) node[fill=black]{};
\draw (3.5,-2) node{}node[below,nodelabel]{$w_1$};
\end{tikzpicture}
\caption{$\partial(X)$ is a nontrivial tight cut, where $X:=V(K_1)-w_1$}
\label{fig:only-one-free-corner-in-VR}
\bigskip
\end{figure}

Now, we prove~{\it (ii)}.
Since $K_1$ is either a ladder or a partial biwheel, it is matching covered.
To show that $K_1$ is conformal, we will display a perfect matching~$M$
of~$G-V(K_1)$. Let $N$ be a perfect matching of~$H$ which contains~$a_1u_1$; observe that $M:=N-E(K_1)$ is the
desired perfect matching.

\smallskip
Note that, if $u_1,w_1 \in V(R)$,
then {\it (iii)} follows immediately from~{\it (ii)}.
Now suppose that $u_1,w_1 \notin V(R)$,
and let $N$ be a perfect matching of~$G-\{a_2,w_1\}$;
note that $\beta \in N$.
A simple counting argument shows that $M:=N - E(K_1) - R$ is a perfect
matching of~$G-V(K_1)-V(R)$; and this proves {\it (iii)}.
\end{proof}

In the following two lemmas, apart from other things, we show
that under certain circumstances there exists an \Rcomp\ edge
which is not in~$K_1$.

\begin{lem}
\label{lem:find-Rcompatible-edge-at-high-degree-free-corner}
Suppose that $u_1,w_1 \notin V(R)$.
Then at most
one edge of~$\partial(u_1)-E(K_1)$ is not \Rcomp.
{\rm (}An analogous statement holds for~$w_1$.{\rm )}
\end{lem}
\begin{proof}
Note that, by Corollary~\ref{cor:application-of-LV},
at most two edges of~$\partial(u_1)$
are non-removable in~$H$; one of these is $a_1u_1$.
Consequently, at most one edge of~$\partial(u_1)-E(K_1)$
is non-removable in~$H$.
To complete the proof we will show that if $e$ is any removable edge of~$H$
such that $e \in \partial(u_1)-E(K_1)$, then $e$ is removable in~$G$ as well;
for this, it suffices to show a perfect matching~$M$
which contains $\alpha$~and~$\beta$ but does not contain~$e$.

\smallskip
Let $M_1$ be a perfect matching of~$G-V(K_1) - V(R)$;
such a perfect matching exists by
Lemma~\ref{lem:conformality-of-Rconfigurations}{\it (iii)}.
Let $M_2$ be a perfect matching of~$K_1-\{a_1,b_1\}$;
since $K_1$ is bipartite matching covered, such a perfect matching
exists by \cite[Proposition~2.1]{koth19}.
Now, \mbox{$M:=M_1 \cup M_2 \cup R$} is the desired perfect matching
alluded to above, and this completes the proof.
\end{proof}

\begin{figure}[!ht]
\centering
\begin{tikzpicture}[scale=0.8]
\draw (3.5,2) -- (3.5,3.2);
\draw (3.5,-2) -- (3.5,-3.2);

\draw (3.15,2.6)node[nodelabel]{$\alpha'$};
\draw (3.85,-2.6)node[nodelabel]{$\beta'$};

\draw (1,0) -- (-0.5,0);
\draw (0.25,0.3)node[nodelabel]{$\alpha$};
\draw (-0.5,0)node{}node[below,nodelabel]{$a_2$};

\draw (6,0) -- (7.5,0);
\draw (6.75,-0.35)node[nodelabel]{$\beta$};
\draw (7.5,0)node[fill=black]{}node[above,nodelabel]{$b_2$};

\draw (1,0) -- (6,0);

\draw (1,0) -- (3.5,2);
\draw (2,0) -- (3.5,-2);
\draw (3,0) -- (3.5,2);
\draw (4,0) -- (3.5,-2);
\draw (5,0) -- (3.5,2);
\draw (6,0) -- (3.5,-2);

\draw (1,0) node{}node[below,nodelabel]{$a_1$};
\draw (2,0) node[fill=black]{};
\draw (3,0) node{};
\draw (4,0) node[fill=black]{};
\draw (5,0) node{};
\draw (6,0) node[fill=black]{}node[above,nodelabel]{$b_1$};
\draw (3.5,2) node[fill=black]{}node[right,nodelabel]{$u_1$};
\draw (3.5,-2) node{}node[left,nodelabel]{$w_1$};

\draw (3.5,-4.5)node[nodelabel]{(a)};
\end{tikzpicture}

\begin{tikzpicture}[scale=0.9]
\draw (1,2) -- (-0.5,2);
\draw (0.25,1.7)node[nodelabel]{$\alpha'$};

\draw (1,0) -- (-0.5,0);
\draw (0.25,0.3)node[nodelabel]{$\alpha$};
\draw (-0.5,0)node{}node[below,nodelabel]{$a_2$};

\draw (3,0) -- (4.5,0);
\draw (3.75,0.3)node[nodelabel]{$\beta'$};

\draw (3,2) -- (4.5,2);
\draw (3.75,1.65)node[nodelabel]{$\beta$};
\draw (4.5,2)node[fill=black]{}node[above,nodelabel]{$b_2$};

\draw (1,0) -- (1,2);
\draw (2,0) -- (2,2);
\draw (3,0) -- (3,2);

\draw (1,0) -- (3,0);
\draw (1,2) -- (3,2);

\draw (1,0) node{}node[below,nodelabel]{$a_1$};
\draw (1,2) node[fill=black]{}node[above,nodelabel]{$u_1$};
\draw (2,2) node{};
\draw (2,0) node[fill=black]{};
\draw (3,0) node{}node[below,nodelabel]{$w_1$};
\draw (3,2) node[fill=black]{}node[above,nodelabel]{$b_1$};

\draw (2,-1.5)node[nodelabel]{(b)};
\end{tikzpicture}
\hspace*{1in}
\begin{tikzpicture}[scale=0.9]
\draw (1,0) -- (-0.5,0);
\draw (0.25,0.3)node[nodelabel]{$\alpha$};
\draw (-0.5,0)node{}node[below,nodelabel]{$a_2$};

\draw (1,2) -- (-0.5,2);
\draw (0.25,1.7)node[nodelabel]{$\alpha'$};

\draw (4,0) -- (5.5,0);
\draw (4.75,0.3)node[nodelabel]{$\beta$};
\draw (5.5,0)node[fill=black]{}node[below,nodelabel]{$b_2$};

\draw (4,2) -- (5.5,2);
\draw (4.75,1.7)node[nodelabel]{$\beta'$};

\draw (1,0) -- (1,2);
\draw (2,0) -- (2,2);
\draw (3,0) -- (3,2);
\draw (4,0) -- (4,2);

\draw (1,0) -- (4,0);
\draw (1,2) -- (4,2);

\draw (1,0) node{}node[below,nodelabel]{$a_1$};
\draw (1,2) node[fill=black]{}node[above,nodelabel]{$u_1$};
\draw (2,2) node{};
\draw (2,0) node[fill=black]{};
\draw (3,0) node{};
\draw (3,2) node[fill=black]{};
\draw (4,2) node{}node[above,nodelabel]{$w_1$};
\draw (4,0) node[fill=black]{}node[below,nodelabel]{$b_1$};

\draw (2.5,-1.5)node[nodelabel]{(c)};
\end{tikzpicture}
\caption{When $|\partial(u_1)-E(K_1)| = |\partial(w_1)-E(K_1)| = 1$}
\label{fig:only-one-edge-at-each-free-corner}
\bigskip
\end{figure}

\begin{lem}
\label{lem:find-Rcompatible-when-free-corners-low-degree}
Suppose that $u_1,w_1 \notin V(R)$.
If {$|\partial(u_1) - E(K_1)| \leq 1$}
and {$|\partial(w_1) - E(K_1)| \leq 1$}
then the following statements hold:
\begin{enumerate}[(i)]
\item $u_1$ and $w_1$ are nonadjacent,
\item $\partial(u_1) - E(K_1)$ has exactly one member, say~$\alpha'$,
and likewise, $\partial(w_1)-E(K_1)$ has exactly one member, say~$\beta'$,
\item $\alpha$~and~$\alpha'$ are adjacent if and only if
$\beta$~and~$\beta'$ are adjacent,
\item if $\alpha$~and~$\alpha'$ are nonadjacent then
at most one edge of~$\partial(v)-\alpha'$ is not \Rcomp, where $v$
denotes the end of~$\alpha'$ which is distinct from~$u_1$;
an analogous statement holds for $\beta$~and~$\beta'$.
\end{enumerate}
\end{lem}
\begin{proof}
We first verify {\it (i)} and {\it (ii)}.
Observe that,
if $u_1$~and~$w_1$ are adjacent,
or, if the sets \mbox{$\partial(u_1)-E(K_1)$}
and \mbox{$\partial(w_1)-E(K_1)$} are both
empty, then $\{a_1,b_1\}$ is a $2$-vertex cut of~$G$; this is absurd.
This proves {\it (i)}.
Note that, if only one of \mbox{$\partial(u_1)-E(K_1)$}
and \mbox{$\partial(w_1)-E(K_1)$}
is nonempty then~$H$
has a cut-edge; this is a contradiction.
This proves {\it (ii)}.
As in the statement, let $\alpha'$ denote
the only member of~$\partial(u_1)-E(K_1)$; and likewise,
let $\beta'$ denote the only member of $\partial(w_1)-E(K_1)$.
See Figure~\ref{fig:only-one-edge-at-each-free-corner}.

\smallskip
We now show that {\it (iii)} holds.
Suppose instead that $\beta$~and~$\beta'$
are adjacent, whereas $\alpha$~and~$\alpha'$ are nonadjacent.
In particular, $\beta'$ has ends $w_1$~and~$b_2$.
We let \mbox{$T:=B - V(K_1) - b_2$}, and note that $T$ is nonempty.
Furthermore, all of the neigbours of~$T$ lie in the
set \mbox{$S:=A-V(K_1)$};
consequently, $S$ is a nontrivial barrier of~$G$;
this is absurd.

\smallskip
We now proceed to prove {\it (iv)}.
Suppose that $\alpha$~and~$\alpha'$ are nonadjacent;
and as in the statement of the lemma,
let $v$ denote the end of~$\alpha'$ which is distinct from~$u_1$.
By {\it (iii)},
$\beta$~and~$\beta'$ are also nonadjacent.
We will first argue that at most one edge of~$\partial(v)-\alpha'$
is non-removable in~$H$.

\smallskip
Observe that $\{\alpha',\beta'\}$ is a $2$-cut of~$H$; thus,
neither~$\alpha'$ nor~$\beta'$ is removable in~$H$.
By Corollary~\ref{cor:application-of-LV}, at most two edges of~$\partial(v)$
are non-removable in~$H$; one of these is~$\alpha'$.
Consequently, at most one edge of $\partial(v)-\alpha'$ is
non-removable in~$H$. To complete the proof we will show that
if $e$ is any removable edge of~$H$ such that $e \in \partial(v)-\alpha'$,
then $e$ is removable in~$G$ as well; for this, it suffices to show a
perfect matching~$M$ which contains $\alpha$~and~$\beta$ but does
not contain~$e$.

\smallskip
Let $M_1$ be any perfect matching of~$G-\{a_2,v\}$;
note that $\beta \in M_1$. A simple counting argument shows
that $\beta'$ lies in~$M_1$ as well. Now, let $M_2$ be a perfect matching
of \mbox{$K_1- \{a_1,u_1,b_1,w_1\}$};
such a perfect matching exists due to
Proposition~\ref{prop:bracelike-property-of-Rconfigurations}.
Observe that
$M:=(M_1 - E(K_1)) \cup M_2 \cup \{\alpha,\alpha'\}$ is the desired
perfect matching alluded to above.
As discussed, this completes the proof.
\end{proof}

In the previous two lemmas, we have shown that under certain circumstances
there exists an \Rcomp\ edge which is not in~$K_1$. However, in the proof of
the Strictly \Rthin\ Edge Theorem (\ref{thm:strictly-Rthin-nb-bricks}), we will
be interested in finding an \Rthin\ edge which is not in~$K_1$. To do so,
we will choose an \Rcomp\ edge appropriately, and use
Theorem~\ref{thm:rank-plus-index}, in conjunction with the following lemma,
to argue that the chosen edge is indeed \Rthin.

\begin{lem}
\label{lem:structure-of-outside-Rcompatible-edge}
Suppose that $u_1,w_1 \notin V(R)$.
Let $e$ denote an \Rcomp\ edge which does not lie in $E(K_1)$,
let $S$ denote a nontrivial barrier of~$G-e$,
and $I$ the set of isolated vertices of~$(G-e)-S$.
Then the following statements hold:
\begin{enumerate}[(i)]
\item $S \cap V(K_1)$ contains at most one vertex, and
\item $I \cap V(K_1)$ is empty.
\end{enumerate}
\end{lem}
\begin{proof}
Since $e$ is \Rcomp, $S$ is a subset of one of the two
color classes of~$H$; assume without loss of generality that $S \subset A$.
To establish {\it (i)}, we will show that if $v_1$~and~$v_2$ are any
two distinct vertices in~$V(K_1) \cap A$, then $(G-e)-\{v_1,v_2\}$
has a perfect matching~$M$.

\smallskip
Let $M_1$ be a perfect matching of~$(H-e)-\{v_1,b_2\}$ where $b_2$
is the end of~$\beta$ which is not in~$V(K_1)$;
such a perfect matching exists by
\cite[Proposition~2.1]{koth19} as $H-e$ is matching covered.
A simple counting argument shows that
$M \cap \partial(V(K_1))$ contains only one edge, and this edge
is incident with the free corner~$u_1$.
Let $M_2$ be a perfect matching of~$K_1 - \{b_1,u_1,v_1,v_2\}$;
such a perfect matching exists due to
Proposition~\ref{prop:bracelike-property-of-Rconfigurations}.
Observe that $M:= (M_1 - E(K_1)) + M_2 + \beta$
is the desired perfect matching of~$(G-e)-\{v_1,v_2\}$, and
this proves {\it (i)}.

\smallskip
We now deduce {\it (ii)} from {\it (i)}.
Suppose to the contrary that $I \cap V(K_1)$ is nonempty, and let $x$
denote any of its members. Observe that $x$ is adjacent with at least two vertices
in~$V(K_1)$, and each of these must lie in~$S$; this contradicts {\it (i)},
and completes the proof.
\end{proof}

\subsection{Proof of Proposition~\ref{prop:Rconfigurations-almost-disjoint}}
\label{sec:proof-Rconfigurations-almost-disjoint}

As in the statement of the proposition, let $G$ be a simple \Rbrick,
and let $K_1$ be an \Rconf\ with \external\ $a_1u_1$~and~$b_1w_1$,
where $u_1$~and~$w_1$ denote the free corners of~$K_1$;
see Figure~\ref{fig:first-Rconfiguration}.
Suppose that $G$ has an \Rconf~$K_2$ which is distinct from~$K_1$;
that is, $K_1$ and $K_2$ are not identical subgraphs of~$G$.
We assume that $K_1$~and~$K_2$ are not vertex-disjoint.
Our goal
is to deduce that $u_1$~and~$w_1$ are the free corners of~$K_2$,
and that $K_2$ is otherwise vertex-disjoint with~$K_1$.

\smallskip
We first argue that $u_1, w_1 \notin V(R)$.
Note that every vertex of~$K_1$, except possibly $u_1$~and~$w_1$,
is cubic in~$G$.
Consequently, if $u_1,w_1 \in V(R)$ then $V(G) = V(K_1)$, since otherwise
$\{u_1,w_1\}$ is a $2$-vertex cut of~$G$; furthermore, either $G$
is precisely the graph induced by $E(K_1) \cup R$, or otherwise, $G$ has one
additional edge joining $u_1$~and~$w_1$; in either case, it is easily seen
that $K_1$ is the only subgraph with all the properties of an \Rconf; this
contradicts the hypothesis.
By Lemma~\ref{lem:conformality-of-Rconfigurations}{\it (i)},
$u_1,w_1 \notin V(R)$.

\begin{Claim}
\label{claim:not-a-free-corner}
Let $z_1$ be any vertex of~$K_1$ which is distinct from $u_1$~and~$w_1$.
If $z_1 \in V(K_2)$ then every edge of~$K_1$ which is incident with~$z_1$
lies in~$E(K_2)$.
\end{Claim}
\begin{proof}
Assume that $z_1 \in V(K_2)$.
First consider the case in which $z_1 \in \{a_1,b_1\}$.
Note that the degree of~$z_1$ in~$H$ is two; consequently, both edges of~$H$
incident with~$z_1$ lie in~$E(K_2)$.

\smallskip
Now consider the case in which $z_1 \notin \{a_1,b_1\}$.
Note that $z_1$ is cubic.
Observe that, for an \Rconf~$K$,
any vertex of~$K$, which is not one of its corners,
is cubic in~$K$ as well as in~$G$.
Thus, it suffices to show that $z_1$ is not a corner of~$K_2$.

\smallskip
Suppose instead that $z_1$ is a corner of~$K_2$.
As $z_1 \notin V(R)$, it is a free corner.
Since $z_1$ is cubic, $K_2$ is an \Rladder.
Also, $z_1$ must be adjacent with a corner of~$K_2$ which lies in~$V(R)$;
such a corner is either $a_1$~or~$b_1$.
Adjust notation so that $z_1$ is adjacent with~$a_1$; thus,
both edges of~$H$ incident with~$a_1$ lie in~$E(K_2)$.
Note that $a_1z_1$ is an external rung of~$K_2$.
Also, since $u_1$ is not a corner of~$K_2$, it is cubic in~$K_2$
and in~$G$.
We infer that $K_1$ is also an \Rladder;
see Figure~\ref{fig:not-a-free-corner}.

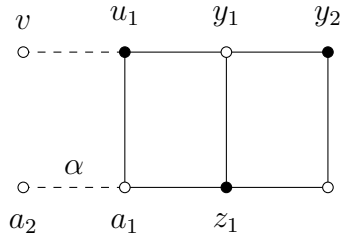
\begin{figure}[!ht]
\centering
\begin{tikzpicture}[scale=0.9]
\draw[dashed] (1,2) -- (-0.5,2);
\draw[dashed] (1,0) -- (-0.5,0);
\draw (0.25,0.3)node[nodelabel]{$\alpha$};
\draw (-0.5,0)node{}node[below,nodelabel]{$a_2$};
\draw (-0.5,2)node{}node[above,nodelabel]{$v$};

\draw (1,0) -- (1,2);
\draw (2.5,0) -- (2.5,2);
\draw (4,0) -- (4,2);

\draw (1,0) -- (4,0);
\draw (1,2) -- (4,2);

\draw (1,0) node{}node[below,nodelabel]{$a_1$};
\draw (1,2) node[fill=black]{}node[above,nodelabel]{$u_1$};
\draw (2.5,2) node{};
\draw (2.5,0) node[fill=black]{};
\draw (4,0) node{};
\draw (4,2) node[fill=black]{};

\draw (2.5,0)node[below,nodelabel]{$z_1$};
\draw (2.5,2)node[above,nodelabel]{$y_1$};
\draw (4,2)node[above,nodelabel]{$y_2$};
\end{tikzpicture}
\caption{Illustration for Claim~\ref{claim:not-a-free-corner};
the solid lines show part of the \Rconf~$K_1$}
\label{fig:not-a-free-corner}
\bigskip
\end{figure}

Let $y_1$ denote the neighbour of~$u_1$ in~$K_1$ which is distinct
from~$a_1$, and let $v$ denote the third neighbour of~$u_1$.
Note that $y_1, v \in V(K_2)$.
Since $|\partial(u_1)-E(K_1)| =1$,
Lemma~\ref{lem:find-Rcompatible-when-free-corners-low-degree}{\it (i)}
implies that $v$ is distinct from~$w_1$.
Since $K_2$ is a ladder, $a_1z_1$ lies in a $4$-cycle of~$K_2$;
this implies that $y_1z_1 \in E(K_2)$.
Note that $u_1y_1$ is an internal rung of~$K_2$.

\smallskip
Let $y_2$ denote the neighbour of~$y_1$ which
is distinct from $u_1$~and~$z_1$. Note that \mbox{$y_2 \in V(K_2)$}.
Since $a_1z_1$ and $u_1y_1$ are rungs of~$K_2$,
it must be the case that $v$ and $y_2$ are adjacent and the
edge joining them is a rung of~$K_2$;
however, it is easily seen that $v$ and $y_2$ are nonadjacent.
We thus have a contradiction.
This completes the proof
of Claim~\ref{claim:not-a-free-corner}.
\end{proof}

We will now use Claim~\ref{claim:not-a-free-corner} to deduce that,
since $K_1$~and~$K_2$ are distinct {\Rconf}s, the only vertices of~$K_1$
which may lie in~$K_2$ are its free corners (that is, $u_1$~and~$w_1$).

\smallskip
Suppose instead that
\mbox{$(V(K_1)- \{u_1,w_1\}) \cap V(K_2)$} is nonempty.
Since \mbox{$K_1 - \{u_1,w_1\}$} is connected,
Claim~\ref{claim:not-a-free-corner} implies that
\mbox{$V(K_1) \subseteq V(K_2)$} and {$E(K_1) \subseteq E(K_2)$}.
Furthermore, since \mbox{$|V(K_1) \cap V(K_2)| \geq 6$}, 
the set {$(V(K_2) - \{u_2,w_2\}) \cap V(K_1)$} is also nonempty, where
$u_2$~and~$w_2$ denote the free corners of~$K_2$.
By symmetry, \mbox{$V(K_2) \subseteq V(K_1)$}
and \mbox{$E(K_2) \subseteq E(K_1)$}.
We conclude that $K_1$~and~$K_2$ are identical subgraphs of~$G$;
contrary to our hypothesis.

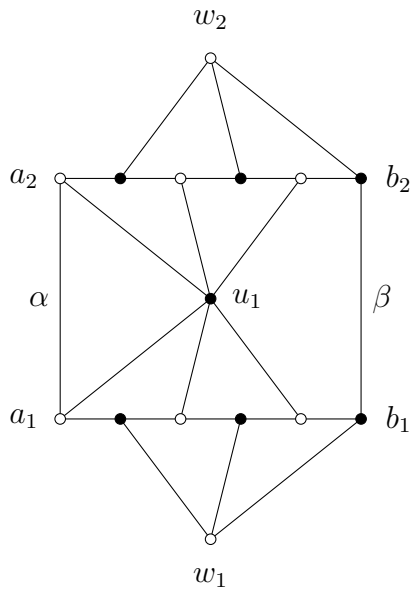
\begin{figure}[!ht]
\centering
\begin{tikzpicture}[scale=0.8]
\draw (1,4) -- (1,0);
\draw (6,4) -- (6,0);
\draw (0.65,2)node[nodelabel]{$\alpha$};
\draw (6.35,2)node[nodelabel]{$\beta$};

\draw (1,4) -- (6,4);

\draw (1,4) -- (3.5,2);
\draw (2,4) -- (3.5,6);
\draw (3,4) -- (3.5,2);
\draw (4,4) -- (3.5,6);
\draw (5,4) -- (3.5,2);
\draw (6,4) -- (3.5,6);

\draw (1,4) node{}node[left,nodelabel]{$a_2$};
\draw (2,4) node[fill=black]{};
\draw (3,4) node{};
\draw (4,4) node[fill=black]{};
\draw (5,4) node{};
\draw (6,4) node[fill=black]{}node[right,nodelabel]{$b_2$};
\draw (3.5,6) node{}node[above,nodelabel]{$w_2$};

\draw (1,0) -- (6,0);

\draw (1,0) -- (3.5,2);
\draw (2,0) -- (3.5,-2);
\draw (3,0) -- (3.5,2);
\draw (4,0) -- (3.5,-2);
\draw (5,0) -- (3.5,2);
\draw (6,0) -- (3.5,-2);

\draw (1,0) node{}node[left,nodelabel]{$a_1$};
\draw (2,0) node[fill=black]{};
\draw (3,0) node{};
\draw (4,0) node[fill=black]{};
\draw (5,0) node{};
\draw (6,0) node[fill=black]{}node[right,nodelabel]{$b_1$};
\draw (3.5,2) node[fill=black]{}node[right,nodelabel]{$u_1$};
\draw (3.5,-2) node{}node[below,nodelabel]{$w_1$};
\end{tikzpicture}
\vspace*{-0.1in}
\caption{When $K_1$~and~$K_2$ share only one free corner}
\label{fig:K1-K2-share-only-one-free-corner}
\bigskip
\bigskip
\end{figure}

Thus, each member of~$V(K_1) \cap V(K_2)$
is a free corner of~$K_1$, and by symmetry,
it is a free corner of~$K_2$ as well.
By our hypothesis, $V(K_1) \cap V(K_2)$ is nonempty;
thus, at least one of $u_1$~and~$w_1$ is a free corner of~$K_2$.
Adjust notation so that $u_1$ is a free corner of~$K_2$.
To complete the proof,
we will show that $w_1$ is also a free corner of~$K_2$.

\smallskip
Suppose not, that is, say $V(K_1) \cap V(K_2) = \{u_1\}$,
and let $w_2$ denote the free corner of~$K_2$ distinct from~$u_1$.
Observe that the ends~$a_2$ of $\alpha$ and $b_2$ of $\beta$ both
lie in~$V(K_2)$; see Figure~\ref{fig:K1-K2-share-only-one-free-corner}.
Furthermore, $|B - V(K_1 \cup K_2)| = |A - V(K_1 \cup K_2)| + 1$.
We shall let \mbox{$T:=B -V(K_1 \cup K_2)$}.
Since every vertex of~$K_1 \cup K_2$, except possibly $u_1,w_1$~and~$w_2$,
is cubic, all neighbours of~$T$ lie in the set
$S:= (A - V(K_1 \cup K_2)) \cup \{w_1,w_2\}$. Consequently,
$S$ is a nontrivial barrier of~$G$; this is absurd.

\smallskip
Thus, $u_1$~and~$w_1$ are the free corners of~$K_2$,
and $K_2$ is otherwise vertex-disjoint with~$K_1$.
This completes the proof of
Proposition~\ref{prop:Rconfigurations-almost-disjoint}. \qed

\section{Strictly \Rthin\ Edge Theorem}
\label{sec:proof-of-strictly-Rthin-edge-theorem}

As in the statement of the theorem (\ref{thm:strictly-Rthin-nb-bricks}),
let $G$ be a simple \Rbrick\ which is free of \sRthin\ edges. Our goal
is to show that $G$ is a member of one of the eleven infinite
families which appear in the statement of the theorem,
that is, to show that $G \in \mathcal{N}$.
We adopt Notation~\ref{Not:Rbrick-doubleton}.

\smallskip
We may assume that $G$ is different from $K_4$~and~$\overline{C_6}$,
and thus, by the \Rthin\ Edge Theorem (\ref{thm:Rthin-nb-bricks}),
$G$ has an \Rthin\ edge, say~$e_1$. Depending on the index of~$e_1$,
we invoke either the \mbox{$R$-biwheel} Theorem~(\ref{thm:Rbiwheel-configuration})
or the \mbox{$R$-ladder} Theorem~(\ref{thm:Rladder-configuration}) to deduce
that $G$ has an \Rconf, say~$K_1$, such that $e_1 \in E(K_1)$.
We shall let $a_1u_1$~and~$b_1w_1$ denote the
\external\ of~$K_1$,
and adjust notation so that $u_1$~and~$w_1$ are its free corners.
See Notation~\ref{Not:biwheel-ladder-convention}
and Figure~\ref{fig:first-Rconfiguration}.
Note that $a_1$ is an end of~$\alpha$ and
$b_1$ is an end of~$\beta$.

\smallskip
By Lemma~\ref{lem:conformality-of-Rconfigurations},
either both free corners $u_1$~and~$w_1$ lie in~$V(R)$,
or otherwise, neither of them lies in~$V(R)$; let us first deal
with the former case.

\begin{Claim}
\label{claim:prisms-Mobius-ladders-truncated-biwheels}
If $u_1,w_1 \in V(R)$ then $G$ is either a prism, or a M{\"o}bius ladder or a
truncated biwheel.
\end{Claim}
\begin{proof}
Suppose that $u_1,w_1 \in V(R)$; that is, $\alpha = a_1w_1$
and $\beta = b_1u_1$. Since every vertex of~$K_1$ is cubic in~$G$,
except possibly $u_1$~and~$w_1$, we conclude that $V(G) = V(K_1)$
as otherwise $\{u_1,w_1\}$ is a $2$-vertex cut of~$G$. Furthermore,
either $G$ is precisely the graph induced by $E(K_1) \cup R$,
or otherwise, $G$
has one additional edge joining $u_1$~and~$w_1$. In the latter case,
$u_1w_1$ is a \sRthin\ edge, contrary to the hypothesis.
In the former case, observe that: if $K_1$ is an \Rbiwheel, as shown in
Figure~\ref{fig:first-Rconfiguration}a, then $G$ is
a truncated biwheel; if $K_1$ is an \Rladder\ of odd parity, as shown
in Figure~\ref{fig:first-Rconfiguration}b, then $G$ is a prism;
and if $K_1$ is an \Rladder\ of even parity, as shown
in Figure~\ref{fig:first-Rconfiguration}c, then $G$ is a M{\"o}bius ladder.
\end{proof}

We may thus assume that neither $u_1$ nor $w_1$ lies in~$V(R)$.
Consequently, the end~$a_2$ of~$\alpha$ and the end~$b_2$ of~$\beta$
are both in~$V(G)-V(K_1)$.

\begin{Claim}
\label{claim:staircases-pseudo-biwheels}
Either $G$ is a staircase or a pseudo-biwheel, or otherwise, $G$
has an \Rcomp\ edge which is not in~$E(K_1)$.
\end{Claim}
\begin{proof}
We begin by noting that, if $|\partial(u_1) - E(K_1)| \geq 2$,
then by
Lemma~\ref{lem:find-Rcompatible-edge-at-high-degree-free-corner},
some edge of~$\partial(u_1)-E(K_1)$
is \Rcomp, and we are done; an analogous argument applies
when $|\partial(w_1) - E(K_1)| \geq 2$.

\smallskip
Now suppose that $|\partial(u_1)-E(K_1)| \leq 1$ and
that $|\partial(w_1)-E(K_1)| \leq 1$.
By
Lemma~\ref{lem:find-Rcompatible-when-free-corners-low-degree}
{\it (i)}~and~{\it (ii)}, $u_1$ and $w_1$ are nonadjacent;
furthermore, $\partial(u_1)-E(K_1)$ has a single element, say~$\alpha'$;
likewise, $\partial(w_1)-E(K_1)$ has a single element, say~$\beta'$;
see Figure~\ref{fig:only-one-edge-at-each-free-corner}. We let $R':=\{\alpha',\beta'\}$.
By {\it (iii)} of the same lemma, $\alpha$~and~$\alpha'$ are adjacent
if and only if $\beta$~and~$\beta'$ are adjacent.

\smallskip
First consider the case in which $\alpha$~and~$\alpha'$ are nonadjacent,
and as in the statement of
Lemma~\ref{lem:find-Rcompatible-when-free-corners-low-degree}{\it (iv)},
let $v$ denote the end of~$\alpha'$ which is distinct from~$u_1$;
note that $v \notin V(K_1)$.
By the lemma,
$\partial(v)-\alpha'$ contains an \Rcomp\ edge, and we are done.

\smallskip
Now suppose that $\alpha$~and~$\alpha'$ are adjacent;
whence $\beta$~and~$\beta'$ are also adjacent.
Note that $\alpha' = u_1a_2$ and $\beta' = w_1b_2$.
Every vertex of~$K_1$, except possibly $u_1$~and~$w_1$, is cubic
in~$G$; furthermore, $\partial(u_1)-E(K_1) = \{\alpha'\}$, and likewise,
$\partial(w_1)-E(K_1) = \{\beta'\}$.
We infer that $V(G) = V(K_1) \cup \{a_2,b_2\}$ as otherwise
$\{a_2,b_2\}$ is a $2$-vertex cut of~$G$. Furthermore,
since each of $a_2$~and~$b_2$
has degree at least three, there is an edge joining them;
and $G$ is precisely the graph
induced by~$E(K_1) \cup R \cup R' \cup \{a_2b_2\}$.
Observe that if $K_1$ is an \Rbiwheel\ of
order at least eight then $G$ is a pseudo-biwheel,
and otherwise, $G$ is a staircase.
\end{proof}

We may thus assume that $G$ has an \Rcomp\ edge which is not in~$E(K_1)$.
We will now use Theorem~\ref{thm:rank-plus-index}
and Lemma~\ref{lem:structure-of-outside-Rcompatible-edge}
to deduce that $G$ has an \Rthin\ edge which is not in~$E(K_1)$.

\begin{Claim}
\label{claim:second-Rthin-edge}
$G$ has an \Rthin\ edge, say~$e_2$, which is not in~$E(K_1)$.
\end{Claim}
\begin{proof}
Among all \Rcomp\ edges which are not in~$E(K_1)$, we choose one, say~$e_2$,
such that ${\sf rank}(e_2) + {\sf index}(e_2)$ is maximum; we intend to show that
$e_2$ is \Rthin. Suppose not; then, by Theorem~\ref{thm:rank-plus-index},
with~$e_2$ playing the role of~$e$,
there exists another \Rcomp\ edge~$f$ such that
(i) $f$ has an end each of whose neighbours in~$G-e_2$ lies in a
(nontrivial) barrier~$S$
of~$G-e_2$, and (ii) ${\sf rank}(f) + {\sf index}(f) > {\sf rank}(e_2) + {\sf index}(e_2)$.

\smallskip
Let $I$ denote the set of isolated vertices of~$(G-e_2)-S$. Condition (i) above
implies that $f$ has one end in~$I$ and another end in~$S$.
By Lemma~\ref{lem:structure-of-outside-Rcompatible-edge},
with $e_2$ playing the role of~$e$,
the set $I \cap V(K_1)$ is empty. Since $f$ has one end in~$I$,
we infer that $f$ is not in~$E(K_1)$; this, combined with condition (ii) above,
contradicts our choice of~$e_2$.
We thus conclude that $e_2$ is \Rthin.
\end{proof}

\smallskip
Now, depending on the index of~$e_2$, we invoke either the
$R$-biwheel Theorem~(\ref{thm:Rbiwheel-configuration})
or the $R$-ladder Theorem~(\ref{thm:Rladder-configuration})
to deduce that $G$ has an \Rconf, say~$K_2$, such that $e_2 \in E(K_2)$.
As $e_2$ is not in~$E(K_1)$ but it is in~$E(K_2)$, the {\Rconf}s
$K_1$~and~$K_2$ are clearly distinct.
By Proposition~\ref{prop:Rconfigurations-almost-disjoint}:
either $K_1$~and~$K_2$ are vertex-disjoint; or otherwise,
$u_1$~and~$w_1$ are the free corners of~$K_2$, and
$K_2$ is otherwise vertex-disjoint with~$K_1$.
In either case, the end~$a_2$ of~$\alpha$ and the end~$b_2$ of~$\beta$
are the two corners of~$K_2$ which are distinct from its free corners.
Let us first deal with the case in which $K_1$~and~$K_2$ are not vertex-disjoint;
Figure~\ref{fig:Rconfigurations-not-disjoint} shows an example in which $K_1$~and~$K_2$
are both {\Rbiwheel}s.

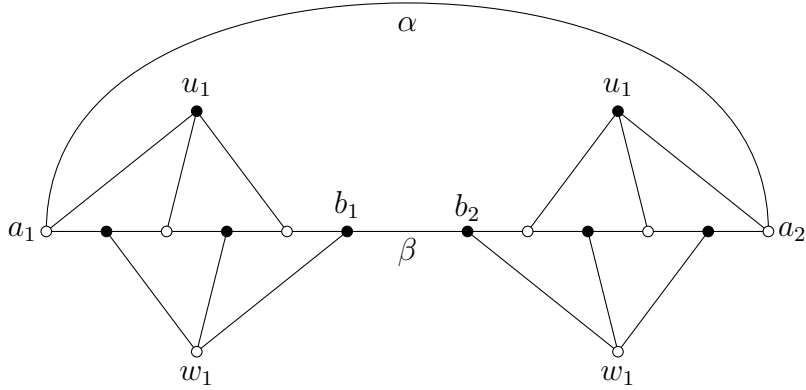
\begin{figure}[!ht]
\centering
\begin{tikzpicture}[scale=0.8]

\draw (1,0) to [out=90,in=180] (7,3.8) to [out=0,in=90] (13,0);
\draw (7,3.45)node[nodelabel]{$\alpha$};



\draw (6,0) -- (8,0);
\draw (7,-0.35)node[nodelabel]{$\beta$};

\draw (1,0) -- (6,0);

\draw (1,0) -- (3.5,2);
\draw (2,0) -- (3.5,-2);
\draw (3,0) -- (3.5,2);
\draw (4,0) -- (3.5,-2);
\draw (5,0) -- (3.5,2);
\draw (6,0) -- (3.5,-2);

\draw (1,0) node{};
\draw (0.6,0) node[nodelabel]{$a_1$};
\draw (2,0) node[fill=black]{};
\draw (3,0) node{};
\draw (4,0) node[fill=black]{};
\draw (5,0) node{};
\draw (6,0) node[fill=black]{};
\draw (6,0.45) node[nodelabel]{$b_1$};
\draw (3.5,2) node[fill=black]{};
\draw (3.5,2.4) node[nodelabel]{$u_1$};
\draw (10.5,2) node[fill=black]{};
\draw (10.5,2.4) node[nodelabel]{$u_1$};
\draw (3.5,-2) node{};
\draw (3.5,-2.4) node[nodelabel]{$w_1$};
\draw (10.5,-2.4) node[nodelabel]{$w_1$};

\draw (8,0) -- (13,0);

\draw (8,0) -- (10.5,-2);
\draw (9,0) -- (10.5,2);
\draw (10,0) -- (10.5,-2);
\draw (11,0) -- (10.5,2);
\draw (12,0) -- (10.5,-2);
\draw (13,0) -- (10.5,2);

\draw (10.5,-2)node{};
\draw (8,0) node[fill=black]{};
\draw (8,0.4) node[nodelabel]{$b_2$};
\draw (9,0) node{};
\draw (10,0) node[fill=black]{};
\draw (11,0) node{};
\draw (12,0) node[fill=black]{};
\draw (13,0) node{};
\draw (13.4,0) node[nodelabel]{$a_2$};

\end{tikzpicture}
\vspace*{-0.1in}
\caption{When the two {\Rconf}s are not disjoint; the vertices with the same labels are to be identified}
\label{fig:Rconfigurations-not-disjoint}
\bigskip
\end{figure}

The proof of the following claim closely resembles
that of Claim~\ref{claim:prisms-Mobius-ladders-truncated-biwheels}.

\begin{Claim}
\label{claim:type-I-families}
If $K_1$~and~$K_2$ are not vertex-disjoint
then $G$ is either a double biwheel or a double ladder or a laddered biwheel,
each of type~I.
\end{Claim}
\begin{proof}
As noted above, $u_1$~and~$w_1$ are the free corners of~$K_2$,
and $K_2$ is otherwise vertex-disjoint with~$K_1$.
Consequently, the \external\ of~$K_2$ are $a_2u_1$~and~$b_2w_1$;
see Figure~\ref{fig:Rconfigurations-not-disjoint}. Since every vertex of~$K_1 \cup K_2$
is cubic in~$G$, except $u_1$~and~$w_1$, we infer that
$V(G) = V(K_1) \cup V(K_2)$, as otherwise $\{u_1,w_1\}$ is a
$2$-vertex cut of~$G$. Furthermore, either $G$ is precisely the
graph induced by~$E(K_1 \cup K_2) \cup R$, or otherwise,
$G$ has one additional edge joining $u_1$~and~$w_1$.
In the latter case, $u_1w_1$ is a \sRthin\ edge, contrary to the hypothesis.
In the former case, observe that
if $K_1$~and~$K_2$ are both {\Rbiwheel}s then $G$ is a double biwheel
of type~I; likewise, if $K_1$~and~$K_2$ are both {\Rladder}s then $G$
is a double ladder of type~I; finally, if one of $K_1$~and~$K_2$ is an
\Rladder\ and the other one is an \Rbiwheel\ then $G$ is a laddered biwheel
of type~I.
\end{proof}

We may thus assume that $K_1$ and $K_2$ are vertex-disjoint;
and we shall let $a_2u_2$ and $b_2w_2$ denote
the \external\ of~$K_2$; in particular, $u_2$~and~$w_2$
denote the free corners of~$K_2$. Figure~\ref{fig:Rconfigurations-disjoint}
shows an example in which $K_1$ is an \Rladder\ and $K_2$ is an \Rbiwheel.

\begin{figure}[!ht]
\centering
\bigskip
\begin{tikzpicture}[scale=1]

\draw (-0.4,0)node[nodelabel]{$a_1$};
\draw (-0.4,2)node[nodelabel]{$u_1$};
\draw (2,-0.3)node[nodelabel]{$w_1$};
\draw (2.2,2.22)node[nodelabel]{$b_1$};

\draw (6.5,3.1)node[nodelabel]{$w_2$};
\draw (6.5,-1.1)node[nodelabel]{$u_2$};
\draw (4,0.7)node[nodelabel]{$b_2$};
\draw (9,1.3)node[nodelabel]{$a_2$};

\draw (4.5,-2)node[above,nodelabel]{$\alpha$};
\draw (3.1,1.77)node[nodelabel]{$\beta$};

\draw (0,0) to [out=270,in=180] (4.5,-1.8) to [out=0,in=270] (9,1);

\draw (2,2) -- (4,1);

\draw (0,0) -- (2,0)node{} -- (2,2)node[fill=black]{} -- (0,2)node[fill=black]{} -- (0,0)node{};
\draw (1,0)node[fill=black]{} -- (1,2)node{};

\draw (4,1) -- (9,1);

\draw (4,1)node[fill=black]{} -- (6.5,2.8);
\draw (5,1)node{} -- (6.5,-0.8);
\draw (6,1)node[fill=black]{} -- (6.5,2.8);
\draw (7,1)node{} -- (6.5,-0.8);
\draw (8,1)node[fill=black]{} -- (6.5,2.8)node{};
\draw (9,1)node{} -- (6.5,-0.8)node[fill=black]{};

\end{tikzpicture}
\caption{When the two {\Rconf}s are disjoint}
\label{fig:Rconfigurations-disjoint}
\bigskip
\end{figure}

We now find the remaining three families, or show the existence of an
\Rcomp\ edge which is not in~$E(K_1 \cup K_2)$; the proof is
similar to that of Claim~\ref{claim:staircases-pseudo-biwheels}.

\begin{Claim}
\label{claim:type-II-families}
Either $G$ is a double biwheel or a double ladder or a laddered biwheel,
each of type~II, or otherwise, $G$ has an \Rcomp\ edge which is not
in~$E(K_1 \cup K_2)$.
\end{Claim}
\begin{proof}
We begin by noting that, if $|\partial(u_1) - E(K_1)| \geq 2$,
then by Lemma~\ref{lem:find-Rcompatible-edge-at-high-degree-free-corner},
some edge of~$\partial(u_1)-E(K_1)$ is \Rcomp, and since
$u_1 \notin V(K_2)$, such an edge is not in~$E(K_2)$, and we are done;
an analogous argument applies when $|\partial(w_1) - E(K_1)| \geq 2$,
or when $|\partial(u_2)-E(K_2)| \geq 2$ or when
$|\partial(w_2)-E(K_2)| \geq 2$.

\smallskip
Now suppose that, for $i \in \{1,2\}$, $|\partial(u_i) - E(K_i)| \leq 1$
and $|\partial(w_i) - E(K_i)| \leq 1$;
by Lemma~\ref{lem:find-Rcompatible-when-free-corners-low-degree}
{\it (i)}~and~{\it (ii)},
$u_i$~and~$w_i$ are nonadjacent;
furthermore, each of these inequalities holds with equality.
Let $\alpha'$ denote the only member of~$\partial(u_1)-E(K_1)$,
and let $\beta'$ denote the only member of~$\partial(w_1)-E(K_1)$.

\smallskip
First consider the case in which either $w_2$ is not an end of~$\alpha'$
or $u_2$ is not an end of~$\beta'$. Assume without loss of generality
that $w_2$ is not an end of~$\alpha'$; thus the end of~$\alpha'$ distinct
from~$u_1$, say~$v$, is not in~$V(K_1 \cup K_2)$.
By Lemma~\ref{lem:find-Rcompatible-when-free-corners-low-degree}{\it (iv)},
$\partial(v)-\alpha'$ contains an \Rcomp\ edge;
such an edge is not in~$E(K_1 \cup K_2)$, and we are done.

\smallskip
Now suppose that $w_2$ is an end of~$\alpha'$ and $u_2$ is an end of~$\beta'$.
Note that every vertex of~$K_1 \cup K_2$, except possibly $u_1, w_1, u_2$
and $w_2$, is cubic in~$G$; furthermore,
$\partial(u_1)-E(K_1) = \partial(w_2)-E(K_2) = \{\alpha'\}$, and likewise,
$\partial(w_1)-E(K_1) = \partial(u_2)-E(K_2) = \{\beta'\}$.
We conclude that $V(G) = V(K_1 \cup K_2)$
and $E(G) = E(K_1 \cup K_2) \cup R \cup R'$.
Observe that: if $K_1$~and~$K_2$ are both {\Rbiwheel}s then
$G$ is a double biwheel of type~II; likewise, if $K_1$~and~$K_2$
are both {\Rladder}s then $G$ is a double ladder of type~II;
finally, if one of $K_1$~and~$K_2$ is an \Rladder\ and the other one
is an \Rbiwheel\ then $G$ is a laddered biwheel of type~II.
\end{proof}

We may thus assume that $G$ has an \Rcomp\ edge
which is not in~$E(K_1 \cup K_2)$.
We will now use Theorem~\ref{thm:rank-plus-index}
and Lemma~\ref{lem:structure-of-outside-Rcompatible-edge}
to deduce that $G$ has an \Rthin\ edge which is not in~$E(K_1 \cup K_2)$.
The proof is almost identical to that of Claim~\ref{claim:second-Rthin-edge},
except that now we have to deal with two {\Rconf}s instead of just one.

\begin{Claim}
\label{claim:third-Rthin-edge}
$G$ has an \Rthin\ edge, say~$e_3$, which is not in~$E(K_1 \cup K_2)$.
\end{Claim}
\begin{proof}
Among all \Rcomp\ edges which are not in~$E(K_1 \cup K_2)$,
we choose one, say~$e_3$,
such that ${\sf rank}(e_3) + {\sf index}(e_3)$ is maximum; we intend to show that
$e_3$ is \Rthin. Suppose not; then, by Theorem~\ref{thm:rank-plus-index},
with~$e_3$ playing the role of~$e$,
there exists another \Rcomp\ edge~$f$ such that
(i) $f$ has an end each of whose neighbours in~$G-e_3$ lies in a
(nontrivial) barrier~$S$
of~$G-e_3$, and (ii) ${\sf rank}(f) + {\sf index}(f) > {\sf rank}(e_3) + {\sf index}(e_3)$.

\smallskip
Let $I$ denote the set of isolated vertices of~$(G-e_3)-S$. Condition (i) above
implies that $f$ has one end in~$I$ and another end in~$S$.
By Lemma~\ref{lem:structure-of-outside-Rcompatible-edge},
with $e_3$ playing the role of~$e$,
the set $I \cap V(K_1)$ is empty;
likewise, the set~$I \cap V(K_2)$
is empty.
Since $f$ has one end in~$I$,
we infer that $f$ is not in~$E(K_1 \cup K_2)$;
this, combined with condition (ii) above,
contradicts our choice of~$e_3$.
We thus conclude that $e_3$ is \Rthin.
\end{proof}

Now, depending on the index of~$e_3$, we invoke either the
$R$-biwheel Theorem~(\ref{thm:Rbiwheel-configuration})
or the $R$-ladder Theorem~(\ref{thm:Rladder-configuration})
to deduce that $G$ has an \Rconf, say~$K_3$, such that $e_3 \in E(K_3)$.
As $e_3$ is not in~$E(K_1 \cup K_2)$ but it is in~$E(K_3)$,
the \Rconf~$K_3$ is distinct from each of $K_1$~and~$K_2$.
We have thus located three distinct {\Rconf}s in the brick~$G$;
namely, $K_1, K_2$~and~$K_3$.
However, this contradicts Corollary~\ref{cor:at-most-two-Rconfigurations},
and completes the proof of the Strictly \Rthin\ Edge Theorem
(\ref{thm:strictly-Rthin-nb-bricks}). \qed

\bigskip
\bigskip
\noindent
{\bf Acknowledgments}:
This work commenced in April 2015 when the first author was a Ph.D. candidate, and the second author
was visiting the University of Waterloo for a fortnight. We are thankful to Joseph Cheriyan who helped facilitate this research visit,
and participated in several of our discussions. We are indebted to both Cl{\'a}udio L. Lucchesi and U. S. R. Murty for their constant guidance and support,
and for sharing their invaluable insights.

\bibliographystyle{plain}
\bibliography{clm}

\end{document}